\documentclass[10pt]{amsart}

\usepackage{amsmath, amsthm, latexsym, amsfonts, amssymb, amscd}
\usepackage[all]{xy}
\usepackage{hyperref}
\usepackage{color}
\usepackage[margin=1.5in]{geometry}

\def \c{\mathbb{C}}
\def \z{\mathbb{Z}}
\def \r{\mathbb{R}}
\def \n{\mathbb{N}}
\def \p{\mathbb{P}}

\def \A{\mathcal{A}}

\def \L{\mathcal{L}}
\def \E{\mathcal{E}}
\def \O{\mathcal{O}}

\def \GL{\textup{GL}}

\def \w{{\underline{w}_0}}

\def \dim{\textup{dim}}

\def \ord{\textup{ord}}
\def \conv{\textup{conv}}

\def \vol{\textup{vol}}

\def \Lie{\textup{Lie}}

\def \codim{\textup{codim}}

\def \Spec{\textup{Spec}}

\def \End{\textup{End}}

\theoremstyle{plain}
\newtheorem{Th}{Theorem}[section]
\newtheorem{Lem}[Th]{Lemma}
\newtheorem{Prop}[Th]{Proposition}
\newtheorem{Cor}[Th]{Corollary}
\newtheorem{THM}{Theorem}

\theoremstyle{definition}
\newtheorem{Ex}[Th]{Example}
\newtheorem{Def}[Th]{Definition}
\newtheorem{Rem}[Th]{Remark}

\begin{document}

\title{Complete intersections in spherical varieties}

\dedicatory{Dedicated to Joseph Bernstein on the occasion of his 70th birthday}

\author{Kiumars Kaveh}
\address{Department of Mathematics, University of Pittsburgh,
Pittsburgh, PA, USA.}
\email{kaveh@pitt.edu}

\author{A. G. Khovanskii}
\address{Department of Mathematics, University of Toronto, Toronto,
Canada; Moscow Independent University, Moscow, Russia.}
\email{askold@math.utoronto.ca}

\begin{abstract}
Let $G$ be a complex reductive algebraic group. We study complete intersections in a spherical homogeneous space $G/H$ defined by a 
generic collection of sections from $G$-invariant linear systems. Whenever nonempty, all such complete intersections are smooth varieties. 
We compute their arithmetic genus as well as some of their $h^{p,0}$ numbers. The answers are given in terms of the moment polytopes and Newton-Okounkov polytopes
associated to $G$-invariant linear systems. We also give a necessary and sufficient condition on a collection of linear systems so that the corresponding generic 
complete intersection is nonempty. This criterion applies to arbitrary quasi-projective varieties (i.e. not necessarily spherical homogeneous spaces). When the spherical
homogeneous space under consideration is a complex torus $(\c^*)^n$, our results specialize to well-known results from the Newton polyhedra theory 
and toric varieties. 
%We give a formula for genus of a generic complete intersection in a spherical homogeneous space $G/H$
%in terms of number of lattice points in polytopes. This generalizes a result of the second author for genus 
%of complete intersections in a complex torus. Also for an arbitrary variety, we give necessary and sufficient conditions on a collection 
%of linear systems so that a generic complete intersection from this collection be nonempty. Moreover we give bounds on the $h^{p,0}$ numbers of such a generic 
%complete intersection.
\end{abstract}

\thanks{The first author is partially supported by a National Science Foundation Grant 
(Grant ID: 1200581).}

\thanks{The second author is partially supported by the Canadian Grant No. 156833-12.}

\date{\today}
\maketitle

\setcounter{tocdepth}{1}
\tableofcontents

%\color{blue}
%Check if our formula for the genus holds if a generic complete intersection is empty.
%Simply add to Examples sections that we assume that a generic complete intersection is nonempty and refer to the remark in Genus section.
%Remark in Genus section that it can be seen from the associated polytopes if this is the case. In particular this is the case if 
%all the associated polytopes are full dimensional.
%Add to the section about generalities that if the linear systems are big then they are independent.
%\color{black}

\section{Introduction}
%The paper is concerned with geometry of complete intersections in a spherical homogeneous space over $\c$.  
%We give formulae for the arithmetic and geometric genus of a complete intersection, as well as its $h^{p,0}$ numbers, in terms of 
%number of integral points in some associated polytopes. The results of the paper are generalization of the similar results for complete intersections in a 
%torus $(\c^*)^n$ in \cite{Askold-toroidal, Askold-genus} to spherical homogeneous spaces.
The main objective of the present paper is to study complete intersections in a spherical homogeneous space $G/H$ where $G$ is a complex 
connected reductive algebraic group. We compute the arithmetic genus as well as many of the $h^{p,0}$ numbers of a generic complete intersection in
$G/H$. Our results generalize the similar results from the Newton polyhedra theory and toric varieties obtained in \cite{Askold-toroidal, Askold-genus, Askold-new}.

We first get a convex geometric formula for the Euler characteristic of a $G$-linearized line bundle over a projective spherical variety. This then allows us to 
represent the arithmetic genus of a complete intersection also in terms of convex geometric data. In many cases of interest our formula is quite computable (Section \ref{sec-examples}).
Our approach is based on the notion of virtual polytope as developed in \cite{Kh-P}. Moreover, we use Newton-Okounkov bodies/polytopes associated to linear systems 
(\cite{Okounkov-spherical, AB, KKh-reductive}) as well as string polytopes (in particular Gelfand-Zetlin polytopes) associated to irreducible representations of $G$
(\cite{Littelmann, BZ} and \cite{Kiumars-string}). 
 
Let $G$ be a reductive algebraic group. A variety $X$ with an action of $G$ is called {\it spherical} if a Borel subgroup of $G$ has a dense open orbit. 
Spherical varieties are generalizations of toric varieties (where $G = (\c^*)^n$ is a torus) on one hand and the flag varieties $G/P$ on the other hand. Similar to toric varieties, geometry 
of spherical varieties and their orbit structure can be read off from combinatorial and convex geometric data of fans and convex polytopes.

We begin by recalling results about complete intersections in a torus (see also Section \ref{sec-toric-12}). 
Let $A_1, \ldots, A_k$ be finite subsets of $\z^n$ where $k \leq n$. For each $i=1, \ldots, k$ let 
$f_i(x) = \sum_{\alpha \in A_i} c_{i, \alpha} x^\alpha$ be a Laurent 
polynomial in $\c[x_1^{\pm 1}, \ldots, x_n^{\pm 1}]$ with generic coefficients $c_{i, \alpha}$. Let $$X_k = \{ x \in (\c^*)^n \mid f_1(x) = \cdots = f_k(x) = 0\}$$ 
be the complete intersection in the torus $(\c^*)^n$ defined by the $f_i$. For each $i=1, \ldots, k$, let $\Delta_i$ denote the convex hull of $A_i$. It is an integral convex polytope
in $\r^n$. 

In general a generic complete intersection $X_k$ in $(\c^*)^n$ may be empty. A beautiful result of David Bernstein \footnote{David Bernstein is the younger brother of 
Joseph Bernstein. He discovered the famous formula for the number of solutions in $(\c^*)^n$ of a
generic system of $n$ polynomial equations with fixed Newton polytopes  \cite{Bernstein}.
This amazing formula inspired much activity that eventually lead to
the creation of Newton polyhedra theory and the theory of Newton-Okounkov bodies.} 
gives a necessary and sufficient condition
for $X_k$ to be nonempty, in terms of the dimensions of the polytopes $\Delta_i$ and their Minkowski sums. It relies on a theorem of Minkowski which 
gives a necessary and sufficient condition for the mixed volume of $n$ convex polytopes to be nonzero (see \cite{Askold-new}).

We recall that the {\it arithmetic genus} of a smooth complete variety $Z$ of dimension $d$ is by definition $\chi(Z) = \sum_{i=0}^d (-1)^i h^{i,0}(Z)$. The 
{\it geometric genus} $p_g(Z)$ is the number $h^{d,0}(Z)$, i.e. the dimension of the space of holomorphic top forms. One knows that $h^{p,0}$ numbers 
are birational invariants and hence the $h^{p,0}$ numbers and the arithmetic and geometric genus can be defined for non-smooth and non-complete varieties as well.

In \cite{Askold-genus} the following formula for the arithmetic genus of $X_k$ is proved. It computes the arithmetic genus in terms of the number of integral points 
in the relative interior of $\Delta_i$ and their Minkowski sums:
\begin{equation} \label{equ-intro-genus-torus}
\chi(X_k) = 1 - \sum_{i_1} N'(\Delta_{i_1}) + \sum_{i_1 < i_2} N'(\Delta_{i_1} + \Delta_{i_2}) - \cdots + (-1)^k N'(\Delta_1 + \cdots + \Delta_k),
\end{equation}
where for a polytope $\Delta$, $N'(\Delta)$ denotes the number of integral points in the interior of $\Delta$ times $(-1)^{\dim(\Delta)}$. Here the interior is with respect to the 
topology of the affine span of $\Delta$. Moreover, if all the polytopes $\Delta_i$ have full dimension $n$, all the $h^{i,0}(X_k)$ are $0$ except $h^{0,0}(X_k) = 1$ and 
$h^{n-k,0}(X_k)$ which can be computed from \eqref{equ-intro-genus-torus}. In fact, the condition that the polytopes have full dimension can be substitute with a weaker condition
(see Corollary \ref{cor-32}).

As above let $G$ be a complex connected reductive algebraic group. Let $G/H$ be a spherical homogeneous space of dimension $n$. 
Let $\E_1, \ldots, \E_k$, $k \leq n$, be $G$-linearized globally generated line bundles on $G/H$. For each $i=1, \ldots, k$, let $E_i \subset H^0(G/H, \E_i)$ be a nonzero 
$G$-invariant linear system. 
That is, $E_i$ is a nonzero finite dimensional linear subspace of $H^0(G/H, \E_i)$ and stable under the action of $G$. Also for $i=1, \ldots, k$, let $f_i$ be a generic section from 
the linear system $E_i$. We are interested in a generic complete intersection:
$$X_k = \{ x \in G/H \mid f_1(x) = \cdots = f_k(x) = 0\}.$$

Let $\E$ be a $G$-linearized line bundle on $G/H$ and $E \subset H^0(G/H, \E)$ a nonzero $G$-invariant linear system.
Generalizing the notion of Newton polytope of a Laurent polynomial, to $E$ one associates two polytopes: the {\it moment polytope} $\Delta(E)$ and the {\it Newton-Okounkov polytope} 
$\tilde{\Delta}(E)$ (see Section \ref{subsec-NO-poly}). The moment polytope $\Delta(E)$ is defined as:
\begin{equation} \label{equ-intro-moment}
\Delta(E) = \overline{\bigcup_{m > 0} \{ \lambda / m \mid V_\lambda \textup{ appears in } \overline{E^m} \}}.
\end{equation}
Here $E^m \subset H^0(G/H, \E^{\otimes m})$ is the linear system spanned by all the products of $m$ elements from $E$ and $\overline{E^m}$ is the completion of $E^m$ as a linear system.
The moment polytope $\Delta(E)$ contains asymptotic information about the highest weights appearing in the complete linear systems $\overline{E^m}$ for large $m$. The name moment polytope comes from symplectic geometry since it coincides with the moment polytope in the sense of Hamiltonian group actions (see Remark \ref{rem-moment-polytope-symplectic}). In the context of reductive group actions on varieties, the notion of moment polytope, as defined in \eqref{equ-intro-moment}, goes back to M. Brion (\cite{Brion-moment}).

The Newton-Okounkov polytope $\tilde{\Delta}(E)$ is a polytope fibered over the moment polytope $\Delta(E)$ with string polytopes as fibers (see Section \ref{subsec-moment}). 
It has the property that for each $m > 0$, the dimension of the complete linear system $\overline{E^m}$ is equal to the number of integral points in the dilated polytope 
$m\tilde{\Delta}(E)$. From this it follows that the number of intersections of $n$ generic hypersurfaces in  the linear system $E$ is equal to $n! \vol(\tilde{\Delta}(E))$.

The definition of Newton-Okounkov polytope of a spherical variety goes back to 
A. Okounkov for when $G$ is a classical group (see \cite{Okounkov-spherical}) and V. Alexeev and M. Brion for a general reductive group (see \cite{AB}). 
It was generalized to arbitrary $G$-varieties in \cite{KKh-reductive}. 

The main results of the paper are as follows: 

Generalizing the Bernstein's theorem, we give a necessary and sufficient condition for $X_k$ to be nonempty. The conditions are in terms of the dimensions of the 
Newton-Okounkov polytopes of the $E_i$ and the products of the $E_i$ (see Theorem \ref{th-sph-nonempty}). The conditions can also be formulated only in terms of the moment polytopes. In fact, the necessary and sufficient conditions for the nonemptiness of a generic complete intersection can be formulated to apply to an arbitrary quasi-projective variety 
(see Theorems \ref{th-nec-cond-nonempty} and \ref{th-suff-cond-nonempty}).

Moreover, whenever a generic complete intersection $X_k$ is not empty, we give a formula for its genus. 
To $G/H$ there corresponds a sublattice in $\z^n$ and to $E$ there corresponds an integral point $\alpha \in \z^n$. For a polytope $\tilde{\Delta}$ we count the number of 
points in $\tilde{\Delta}$ lying in the sublattice shifted by $\alpha$. We denote by $N'(\tilde{\Delta}, \alpha)$ the number of points in the sublattice shifted by $\alpha$ which 
lie in the interior of $\tilde{\Delta}$ times $(-1)^{\dim(\tilde{\Delta})}$ (as before the interior is with respect to the topology of the affine span of $\tilde{\Delta}$).
\begin{THM}[Theorem \ref{th-main}]
The arithmetic genus $\chi(X_k)$ is given by:
\begin{multline}  \label{equ-intro-genus-sph}
\chi(X_k) = 1 - \sum_{i_1} N'(\tilde{\Delta}_{i_1}(E_{i_1}), \alpha_{i_1}) + \sum_{i_1 < i_2} N'(\tilde{\Delta}(E_{i_1} E_{i_2}), \alpha_{i_1}+\alpha_{i_2}) - \cdots \\
+ (-1)^k N'(\tilde{\Delta}(E_1\cdots E_k), \alpha_1 + \cdots + \alpha_k).
\end{multline}
Again the above can also be formulated only in terms of the moment polytopes.
\end{THM}

Finally we give estimates for many of the $h^{i,0}$ numbers of $X_k$ (Theorem \ref{th-h-p-complete-int-sph}). In particular when all the polytopes $\tilde{\Delta}(E_i)$ have full 
dimension we have the following:
%We need a bit of notation:
%For a subset $J \subset \{1, \ldots, k\}$ let $\tilde{\Delta}_J$ denote the polytope associated to the product linear system $\prod_{i \in J} E_i$.
\begin{THM}[Corollary \ref{cor-h-p-complete-int-sph}]
%Let us assume that $p$ is nonnegative integer such that there is no nonempty $J \subset \{1, \ldots, k$ with $N^\circ(\tilde{\Delta}_J) > 0$ and 
%$\dim(\tilde{\Delta}_J) - |J| = p$. 
With notation as above, suppose all the polytopes $\tilde{\Delta}(E_i)$, $i=1, \ldots, k$, have full dimension equal to $n= \dim(G/H)$. Then for any $0 \leq p < n-k$:
$$h^{p,0}(X_k) = \begin{cases} 
1, \quad p = 0 \\
0 \quad p \neq 0.\\
\end{cases}$$
Moreover, $h^{n-k, 0}(X_k)$ can be computed from \eqref{equ-intro-genus-sph}.
\end{THM}
In fact, Corollary \ref{cor-h-p-complete-int-sph} is a slightly stronger version of the above theorem.

In the last section (Section \ref{sec-examples}) we consider three important classes of spherical varieties (aside from toric varieties): (1) horospherical varieties, 
(2) group embeddings, and (3) flag varieties. In particular our formula give very computable formulae for the following concrete examples of complete intersections in spherical 
homogeneous spaces:

\begin{itemize}
\item (A horospherical example)
Let $V$ be a finite dimensional $G$-module. Let $v_1, \ldots, v_s$ be highest weight vectors of $V$ with highest weights $\lambda_1, \ldots, \lambda_s$ respectively. 
Put $v = v_1 + \cdots + v_s$ and let $X$ be the closure of the $G$-orbit of $v$ in $V$. It is an affine spherical subvariety of $V$. Let $L$ be the linear subspace of $\c[X]$
consisting of linear functions in $V^*$ restricted to $X$. Let $\Delta = \conv\{ \lambda_i \mid i = 1, \ldots, s\}$. Also let $\tilde{\Delta}$ denote the corresponding 
Newton-Okounkov polytope. 
Let $f$ be a generic element in $L$ defining a hypersurface $H_f = \{ x \in X  \mid f(x) = 0\}$. Then 
the geometric genus of $H_f$ is equal to the number of integral points in the interior of the polytope $\tilde{\Delta}$ (see Section \ref{subsec-horosph}).

\item (A group example) Let $\pi: G \to \GL(V)$ be a finite dimensional faithful representation of $G$. Let $H$ be a hypersurface in $G$ defined by $f=0$ where $f$ is a 
generic matrix element of $\pi$. Then the the geometric genus of $H$ is given by the number of points in the interior of the polytope $\tilde{\Delta}_\pi$ (see Section \ref{subsec-gp-embed}).

\item (A flag variety example)
Let $G = \GL(n, \c)$. Let $\L_\lambda$ be the $G$-line bundle on the variety of complete flags associated to a dominant weight $\lambda$. 
If $H$ is a generic divisor of $\L_\lambda$ then the geometric genus of $H$ is equal to the number of integral points in $\r^{n(n-1)/2}$ lying in the interior of the Gelfand-Zetlin polytope $\Delta_{\textup{GZ}}(\lambda)$. In other words, 
the number of integral points in $\r^{n(n-1)/2}$ which satisfy the inequalities in \eqref{equ-GZ} where all the inequalities are strict (see Section \ref{subsec-flag}).
\end{itemize}

The second author would like to emphasize that he enjoyed a great amount of support from Joseph Bernstein during his works \cite{Askold-toroidal, Askold-genus}
and his support played an important role in completion of these papers.

In the late 70's  the second author stated and widely advertised the problem
of extending the results in \cite{Askold-toroidal, Askold-genus} about geometry of 
complete intersections in a torus $(\c^*)^n$ to other reductive groups
(see Section \ref{sec-toric-12}). The first result in this direction was obtained by 
B. Kazarnovskii (\cite{Kazarnovskii}) who computed the number of points in a zero dimensional
complete intersection in any reductive group (the answer is expressed as the integral of a certain polynomial over an associated moment polytope). 
Around the same time, M. Brion generalized Kazarnovskii's result to a zero dimensional complete intersection in any spherical
variety (\cite{Brion-Picard}). Much later the first author showed (\cite{Kiumars-thesis}) that the straightforward generalization
of the formula for the (topological) Euler characteristic of a complete intersection in a torus $(\c^*)^n$ to a reductive group fails 
and such a formula should be more complicated than in the torus case. The corresponding
formula was soon found by V. Kiritchenko (\cite{Valentina1, Valentina2}). Her unexpected and beautiful
result was at the same time slightly disappointing: the formula turns out to be
unavoidably too complicated. This somewhat reduced the hope and suggested that perhaps extensions of formulae for other geometric invariants 
from the torus case to the reductive case may be too complicated. Nevertheless, in the present paper we give formulae for the
arithmetic and geometric genus of complete intersections in a spherical homogeneous space, exactly extending the similar formulae for complete intersections 
in a torus $(\c^*)^n$ in terms of the number of integral points in certain associated polytopes (namely Newton-Okounkov polytopes). 

The results of this paper use basic facts about the theory of virtual polytopes and convex chains as developed in \cite{Kh-P}. 
The second author would like to point out that A. Pukhlikov and him arrived at these ideas thinking about the Euler characteritic of $T$-linearized 
line bundles on toric varieties.

Beside the techniques from \cite{Askold-toroidal, Askold-genus}, our results strongly rely on the results of Michel Brion on 
the cohomology of $G$-line bundles on projective spherical varieties (Theorem \ref{th-coh-line-bundle}), as well as the equivariant resolution of singularities of spherical varieties.\\

\noindent{\bf Acknowledgement:}
We would like to thank Michel Brion for telling us about some references regarding moment polytopes of $G$-varieties.\\

\section{Transversality and complete intersections in general varieties} \label{sec-transverse}
In this section we discuss some results on transversality and complete intersections in general varieties. We will use them later in Section \ref{subsec-genus-sph} to prove our 
main results about complete intersection in a spherical homogenous space.
\subsection{Stratifications and complete intersections} \label{subsec-strat}
%\subsubsection {Stratifications} \label{subsec-1.1}
Let $X$ be a complex quasi-projective algebraic variety. A finite collection  $\{Y_i\}$ of its quasi-projective subvarieties  is called a 
{\it stratification of $X$}, and each $Y_i$ is a {\it stratum}, if the following conditions hold: (1) The union of all the strata is $X$.
(2) The intersection of any two different strata is empty. (3) Each stratum is a smooth quasi-projective variety. 

Our definition of stratification, which suffices for our purposes, requires only mild assumptions on the strata with no condition on how a stratum approaches another stratum on its 
boundary (for example Whitney A and B conditions in a Whitney stratification). One can prove that any quasi-projective variety admits a stratification.

The following is easy to prove:
\begin{Lem} \label{lem-1}
Let $X$ be an irreducible quasi-projective $n$-dimensional variety. Then any stratification $\{Y_i\}$ of $X$  has exactly one $n$-dimensional stratum $X_0$ and it is dense in $X$ .
\end{Lem}

%\begin{proof} The only $n$-dimensional closed subvariety of the closure $\bar M$ of $M$ is $\bar M$ itself. So the closure of any $n$-dimensional %stratum $Y_i$ is $\bar M$. Algebraic varieties with the same closure have nonempty intersection. So there could be at most one $n$-dimensional %stratum $Y_i=M_0$ in $M$. On the other hand at least one stratum has to have dimension $n$, otherwise the dimension of $M$ would be smaller then %$n$. Any Zariski open set $U$ in $M$ is itself an $n$-dimensional variety. So $U\cap M_0\neq\emptyset$ and $M_0$ is dense in $M$.
%\end{proof}

%\subsubsection{1.2. Locally complete intersection transversal to stratification} \label{subsec-1.2}

Consider a local  complete intersection $Z$  of codimension $k$ in an $n$-dimensional variety $X$ with a stratification $\{Y_i\}$. The subvariety $Z$  
is said to be {\it transverse to a stratum $Y_j$} of the stratification if for any point $a \in X \cap Y_j$ there is  a Zariski open set $U\subset X$, 
containing $a$, and a system of equations $f_1=\dots=f_k=0$ in $U$ defining  $Z \cap U$ such that the differentials of the restrictions of 
$f_1,\dots,f_k$ to $Y_j$ are independent in the tangent space to $T_a(Y_j)$. The subvariety $Z$ is {\it transverse to the stratification $\{Y_i\}$} if it is transverse to all the strata of the stratification.

The following is straightforward:
\begin{Th} \label{th-2} 
Let $Z$ be a local complete intersection of codimension $k$ in an $n$-dimensional variety $X$ which is transverse to a stratification $\{Y_i\}$ of $X$. Then: (1) If $Z \cap Y_j\neq \emptyset$ the variety $Z \cap Y_j$ is a smooth local complete intersection of codimension $k$ in $Y_j$.
(2) The set of all nonempty intersections $\{Z \cap Y_j\}$ form a stratification of $Z$.
(3) If $X$ is irreducible then  $Z_0 = Z \cap X_0$, where $X_0$  is the stratum of dimension $n$, is dense in $Z$.
(4) If $X$ is smooth then $Z$ is also smooth.
\end{Th}
%\begin{proof} 1)  By implicity function theorem  $X\cap Y_j$ is a smooth subvariety in $Y_j$. Obviously $X\cap Y_j$ is  a locally complete intersection in $Y_j$ of codimension $k$.
%2) Follows from 1).
%3) By assumption each component of  $X$  has codimension $k$ in $X$, so  dimension of each component is  $n-k$.    From   1) it follows that $\dim (X\cap Y_j) =\dim Y_j-k< n-k$ for each stratum $Y_j$ with $\dim Y_j< n$. Any Zariski open set $V\subset X$  is an $n-k$ dimensional variety. So the sets $V\subset X$ and $X_0=X \cap X_0$ have a non empty intersection. Therefor ${X_0}$ is dense in $X$.
%4) Take a point $a\in X$. By definition there is a neighborhood $U$, $a\in U$, and a system of equations $f_1=\dots=f_k=0$  defining $X\cap U$. such that the differentials $d f_1,\dots, d f_k$ are independent in the tangent space at  $a$ to the stratum containing $a$. Therefore the differentials are independent in the tangent space at $a$ to $X$. By the Implicit Function Theorem $X$ is smooth in a neighborhood of the point $a$.
%Theorem 2 is proved.
%\end{proof}

%\subsubsection {1.3. Proper mappings respecting stratifications} \label{subsec-1.3}
Let $X$ and $\tilde{X}$ be quasi-projective varieties with stratifications  $\{Y_i\}$, $\{\tilde{Y}_j\}$ respectively.
\begin{Def} \label{def-resp-stratification}
A  morphism $\pi:\tilde{X}\to X$ {\it respects the stratifications} $\{\tilde{Y}_j\}$ and $\{Y_i\}$ if the following hold: (1) $\pi$ is surjective. (2) Its restriction to each stratum $\tilde{Y}_j$ is a 
surjective map from $\tilde{Y}_j$ to some other stratum $Y_i$. (3) For every $x \in \tilde{Y}_j$ the differential 
$d\pi_x: T_x\tilde{Y}_j \to T_{\pi(x)}Y_i$ is surjective.
\end{Def}

The following is easy to check:
\begin{Th} \label{th-3}
Assume that $\pi: \tilde{X} \to X$ respects the stratifications $\{\tilde{Y}_j\}$ and $\{Y_i\}$. Let $Z \subset X$ be a local complete intersection of codimension $k$ transverse to the stratification 
$\{Y_i\}$. Then $\pi^{-1}(Z)= \tilde{Z}\subset \tilde{X}$ is  a local complete intersection of codimension $k$ transverse to the stratification $\{\tilde{Y}_j\}$.
\end{Th}
%\begin{proof} Consider a point $a \in \pi^{-1}(Z)$. Assume that $a \in \tilde{Y}_j$ and $\pi(a) \in Y_i$. In a neighborhood $U$ of the point $\pi(a)$ the variety $Z$ can be defined by a 
%system of equations $f_1=\dots=f_k=0$ such that the $k$-form $df_1\wedge \cdots \wedge df_k$ on ${Y_i}$ does not vanish at $\pi(a) \in Y_i$. Therefore in the neighborhood 
%$\tilde{U}=\pi^{-1}(U)$ the variety $\tilde{Z}$ can be define by the system of equations $\pi^{*}(f_1)= \cdots = \pi^{*}(f_k)=0$. The $k$-form $d( \pi^{*}(f_1)) \wedge \cdots \wedge 
%d(\pi^{*}(f_k))$ on $\tilde{Y}_i$ does not vanish at $a \in \tilde{Y}_j$ because the differential of $\pi: \tilde{Y}_i\to Y_j$ is an onto map from the 
%tangent bundle of $\tilde{Y}_i$ to  the tangent bundle of $Y_j$.
%\end{proof}

%\subsubsection{$G$-equivariant morphisms} \label{subsec-1.3}
Now let $G$ be a complex algebraic group. In this section we consider $G$-varieties, i.e. varieties equipped with an algebraic action of $G$.
Any $G$-variety is the union of $G$-orbits. We will always assume that the action of $G$ on the variety has only finitely many orbits. It is clear that then the orbits of $G$ give a stratification of
the variety. We refer to this stratification as a {\it $G$-stratification}.
The following is straightforward.
\begin{Lem} \label{lem-4} 
Let $\pi: \tilde{X} \to X$ be a $G$-equivariant surjective morphism of $G$-varieties $\tilde{X}$ and $X$. Then $\pi$ respects the $G$-stratifications of $\tilde{X}$ and $X$.
\end{Lem}
\begin{proof} Consider  a point  $a \in \tilde{X}$ and its image $b=\pi(a)\in X$. Because $\pi$ is $G$-equivariant it maps the $G$-orbit of the point $a$ onto the $G$-orbit of the point $b$. 
A tangent vector $\xi_2$ at $b$ to the orbit of $b$ is a velocity vector of $b$  under an action of some one-parameter subgroup $G_1\subset G$. The vector $\xi_2$ is the image under 
$d\pi$ of the velocity vector $\xi_1$ at $a$ of the action of $G_1$ on $\tilde{X}$. This finishes the proof.
\end{proof}

%\subsubsection{Transversality theorem} \label{subsec-1.4}
For the sake of completeness we recall some basic theorems about transversality.
First we recall the Bertini-Sard theorem. It is an algebraic version of the classical theorem of Sard on critical values of smooth maps on manifolds.
%Instead of smooth maps and measure zero sets it deals with regular mappings and semi-algebraic sets of codimension at least one. 
%For a proof see \cite{}.
\begin{Th}[Bertini-Sard theorem] Let $ F:U \to \c^k$ be a morphism from a smooth algebraic variety $U$ to $\c^k$ and let $\Sigma\subset \c^k$ be the set of critical values of
$F$. Then $\Sigma$ is a semi-algebraic subset of $\c^k$ of  codimension at least one.
\end{Th}
We will say that some property holds  for a {\it generic point} of an irreducible algebraic variety $T$ if there is a proper closed algebraic subset $\Sigma \subset T$ 
such that the property holds for all the points in $X \setminus \Sigma$ (or equivalently if there is a semialgebraic set of codimension at least one such that the property holds in its complement). 
In the notation of the Bertini-Sard theorem, the set $F(U)\subset \c^k$ is semialgebraic. Therefore according to the theorem a generic point of $\c^k$ is not a critical value of  $F$, or equivalently, 
a generic point of $\c^k$ is a regular value of  $F$.

\begin{Lem} \label{lem-5} 
Suppose $T$ is an irreducible variety. Let $\Sigma \subset T$ be a semi-algebraic subset.
Then either $x \in \Sigma$ holds for generic points of $T$, or $x \notin \Sigma$ holds for generic points of $T$.
\end{Lem}
\begin{proof} 
If $\dim(\Sigma)=\dim(T)$ then $\dim(T \setminus \Sigma)<\dim (T)$. If $\dim(\Sigma)<\dim(T)$ then $\dim(T\setminus \Sigma)=\dim(T)$. This proves the lemma.
\end{proof}

%We will use the Bertini-Sard theorem to prove a versions of the Thom transversality theorem below (Theorem \ref{th-7}).  
Let $\L_1, \ldots, \L_k$ be line bundles on $X$, and for $i=1, \ldots, k$ let $E_k \subset H^0(X, \L_i)$ be a finite dimensional linear subspace of sections.
Let ${\bf E}$ denote the $k$-fold product $E_1\times \dots \times E_k$. The following is an immediate corollary of Lemma \ref{lem-5}:
\begin{Lem} \label{lem-6}
Let $Z_{{\bf g}}\subset X$ be the subvariety defined by  $g_1 = \cdots = g_k = 0$ for some 
${\bf g}=(g_1,\dots,g_k)\in {\bf E}$. Then either $Z_{{\bf g}}$ is empty for generic ${\bf g}\in {\bf E}$, or $Z_{{\bf g}}$ is nonempty for generic ${\bf g}\in {\bf E}$.
\end{Lem}
%\begin{proof} In  $X\times {\bf L}$ consider a semi-algebraic set $Z$ defined by the following condition $ (x,{\bf g} )\in Z \Leftrightarrow {\bf g} (x)=0$. 
%The image $\pi( Z)\subset {\bf L}$ of $Z$ under  projection $\pi:X\times {\bf L}\to  {\bf L}$ is  semi-algebraic set and $Z_{{\bf g}}\neq \emptyset \Leftrightarrow {\bf g}\in \pi (Z)$. Now lemma 6 follows %from lemma 5.
%\end{proof}

%\color{blue}
%Now consider a smooth section $g$ of a line bundle over a smooth variety $Y$. If $g(a)=0$ for some $a\in Y$, the differential $dg$  in the tangent space at $a$ to $Y$ is well-defined up to a 
%nonzero factor. Therefore linear independence of differential of sections $g_1, \dots,g_k$ of line bundles $L_1,\dots, L_k$ in a tangent space at $a$ to $Y$ is well-defined if 
%$g_1(a)=\dots=g_k(a)=0$. A  linear subspace $L$ in a space  of regular sections of a  line bundle $\L$ on $X$ is {\it globally generated} if for any point $a\in X$ there is a section $\varphi\in L$ such that $\varphi(a)\neq 0$.
%\color{black}

With notation as in Lemma \ref{lem-6} we have the following version of the Thom transversality theorem. 
It is a corollary of the Bertini-Sard theorem. We skip the details.
\begin{Th} [A version of Thom's transversality theorem] \label{th-7}
Let $X$ be a quasi-projective variety equipped with a stratification $\{Y_i\}$. Assume that the following hold: (1) The linear systems $E_1,\dots, E_k$  are base point free, and  
(2) $Z_{{\bf g}}\neq \emptyset $ for generic ${\bf g}\in {\bf E} = E_1 \times \cdots \times E_k$. Then, for  generic  ${\bf g} \in{\bf E}$, the subvariety $Z_{{\bf g}}$ is a local complete intersection of codimension $k$ which is transverse to $\{Y_i\}$.
\end{Th}

\subsection {When is a generic complete intersection nonempty?} \label{subsec-nonempty}
Suppose we are given $k$ linear systems on a variety. In this section we give a necessary and sufficient condition for a generic complete intersections from these
linear systems to be nonempty. We would need the notion of the Kodaira map of a linear system which we briefly explain below:

%\subsubsection {Kodaira maps whose image have smaller dimension} \label{subsec-2.1}
As above let $X$ be an $n$-dimensional quasi-projective variety. Let $E$ be a linear system on $X$, that is, a finite dimensional linear subspace of global sections of a line  bundle 
$\L$ on $X$. Assume that $E$ is base point free. One can then define a morphism $\Phi_E: X \to \p(E^*)$
called the {\it Kodaira map of $E$}. It is defined as follows: $\Phi_E(x)$ is the point in the projective space $\p(E^*)$ represented by the hyperplane $H_x$ in $E$ consisting of all the sections
which vanish at $x$. We denote the closure of the image of $\Phi_E$ by $Y_E$. It is a projective subvariety of $\p(E^*)$. The following is easy to prove from the definition of the Kodaira map:
\begin{Lem} \label{lem-8}
For $a, b \in X$ we have $\Phi_E(a)= \Phi_E(b)$, if and only if the sets $\{ g \in E \mid g(a)=0\}$ and $\{g \in E \mid g(b)=0\}$ coincide.
\end{Lem}
The notion of a Kodaira map is very useful in the theory of Newton-Okounkov bodies. One usually assumes that the linear system under consideration is large enough so that the 
Kodaira map $\Phi_E$ is an isomorphism (or at least a birational isomorphism) between $X$ and $Y_E$. In this section we relax this and would work with 
general base point free linear systems $E$ such that $Y_E$ can have smaller dimension than that of $X$.

Let $\L_1, \ldots, \L_k$ be a collection of globally generated line bundles on $X$. For $i=1, \ldots, k$ let $E_i \subset H^0(X, \L_i)$ be a finite dimensional subspaces
of global sections of $\L_i$ without base point. We will use the following notation: $I$ denotes the set of indices $\{1, \ldots, k\}$ and $J = \{i_1, \ldots, i_j\}$ 
is a nonempty subset of $I$. We write $\L_J$ for the line bundle $\L_{i_1} \otimes \cdots \otimes \L_{i_j}$ and $E_J$ is the subspace of $H^0(X, \L_J)$ 
spanned by all the tensor products $g_{i_1}\otimes \cdots \otimes g_{i_j}$, where $g_{i_\ell}$ is a section of $E_{i_\ell}$ for $\ell=1, \ldots, j$.  We will denote the Kodaira map 
of the linear system $E_J$ simply by $\Phi_J: X \to \p(E_J^*)$. We have the following extension of Lemma \ref{lem-8}:
\begin{Lem} \label{lem-9-Kodaira}
For $a, b \in X$ we have $\Phi_I(a) = \Phi_I(b)$ if and only if for every 
$i \in I$ the sets $\{g_i \in E_i \mid g_i(a)=0\}$ and $\{g_i \in E_i \mid g_i(b)=0\}$ coincide.
\end{Lem}
\begin{proof} Let us prove that if $\Phi_I(a)=\Phi_I(b)$ and $g_i(a)=0$ for some $g_i\in E_i$ then $g_i(b)=0$. For every $j\neq i$ fix a section $f_j\in E_j$ such that $f_j(a)\neq 0$ and $f_j(b)\neq 0$. For any $g_i \in E_i$ consider the section $\phi = f_1\otimes\dots\otimes f_{i-1}\otimes g_i\otimes f_{i+1}\otimes \dots\otimes f_k \in E_I$. 
By Lemma \ref{lem-8} the conditions $\phi(a)=0$ and $\phi(b)=0$ are equivalent. So the equations $g_i(a)=0$ and $g_i(b)=0$ on a section $g_i \in E_i$ are equivalent.
Conversely, assume that for every $i$ the equations $g_i(a)=0$ and $g_i(b)=0$ for  $g_i \in E_i$ are equivalent. Represent each linear space $E_i$ in the form $E_i^0 \oplus E_i^1$ 
where every section from $E_i^0$ vanishes at the points $a$ and $b$ and the one-dimensional subspace $E_i^1$ is spanned by a section $s_i$ not vanishing at $a$ and $b$. 
The linear space $E_I$ is a sum of $2^k$ subspaces $E^{n_1,\dots,n_k}_I$ where the sum is taken over all $2^k$  $k$-tuples $(n_1,\dots, n_k)$ of indexes $n_i=0,1$, 
and each linear space $E^{n_1,\dots,n_k}_I$ is spanned by the tensor products $f_1\otimes \dots\otimes f_k$ where $f_i\in E_i^0$ if $n_i=0$ and $f_i=s_i$ if $n_i=1$. 
The sections belonging to each summand but not to the one-dimensional subspace $E^{1,\dots, 1}_I$ spanned by the section $s_1\otimes \cdots \otimes s_k$ vanish  
at $a$ and $b$. So the conditions $\phi(a)=0$ and $\phi(b)=0$ for $\phi \in E_I$ are equivalent. By Lemma \ref{lem-8} 
we have $\Phi_I(a)=\Phi_I(b)$ which finishes the proof.
\end{proof}

%\subsubsection{Sufficient  condition  for emptiness} \label{subsec-2.2}
%Below we will use  notations introduced right before theorem 9.
\begin{Def} \label{def-defect}
With notation as above, we define the {\it defect} $d(J)$ of a subset $J \subset I=\{1, \ldots, k\}$ 
to be the number: $$d(J)=\tau_J-|J|,$$ where $\tau_J$ is the dimension of $Y_J$, the closure of the image of $X$ under the Kodaira map 
$\Phi_J$ and $|J|$ is the number of elements in $J$. 
\end{Def}

\begin{Def} \label{def-independent}
With notation as above, we say that the linear systems $E_1, \ldots, E_k$ are {\it independent} if any subset $J \subset I = \{1, \ldots, k\}$ has nonnegative defect.
\end{Def}

\begin{Th} [Necessary condition for a generic complete intersection to be nonempty] \label{th-nec-cond-nonempty}
Let $E_1, \ldots, E_k$ be base point free linear systems. Suppose $Z_{{\bf g}} \subset X$ is nonempty for a generic choice of ${\bf g} \in {\bf E} = E_1 \times \cdots \times E_k$. 
Then $E_1, \ldots, E_k$ are independent (in the sense of Definition \ref{def-independent}). 
\end{Th}
\begin{proof} Assume that the set $J=\{i_1,\dots,i_j\}$ has a negative defect $d(J)$ with respect to the collection of linear systems $E_1,\dots,E_k$. Let us show that for a generic choice of 
${\bf g}_J=(g_{i_1},\dots, g_{i_k}) \in E_{i_1} \times \cdots \times E_{i_j}$, the subvariety $Z_{{\bf g}_J} \subset X$ defined by the equations $g_{i_1}=\dots= g_{i_k}=0$ is empty. 
Suppose this is not the case. Fix a stratification of $X$ with the largest stratum $X_0$. Then by Theorem \ref{th-7} 
the variety $Z_{{\bf g}_J}$ is a local complete intersection of dimension $n-|J|$ and the intersection $Z^0_{{\bf g}_J}=Z_{{\bf g}_J}\cap X_0$ is nonempty and smooth . 
Let $a$ be a point in $Z^0_{{\bf g}_J}$. By Lemma \ref{lem-9-Kodaira} the smooth variety $Z^0_{{\bf g}_J}$ should contain the set of all points  $x\in X_0$ such that 
$\Phi_J(a)=\Phi_J(x)$. The dimension of the set $K_a=\Phi_J^{-1}(\Phi_J(a)) \cap X_0$ is greater than or equal to $n-\tau_J$. Indeed, $\Phi_J$ restricted to $X_0$ is a surjective morphism from the smooth variety $X_0$ of dimension $n$ to  the variety $\Phi_J(X_0)$ of dimension $\tau_J$ . But  $d(J)=\tau_J-|J|<0$ so  $n-\tau_J> n-|J|$ which is impossible because $K_a\subset Z^0_{{\bf g}_J}$. The contradiction proves that $Z_{{\bf g}_J}$ is empty. This shows that $Z_{{\bf g}}$ is empty as well.
\end{proof}

%\subsubsection{A foliation related to a Kodaira map} \label{subsec-2.3}
The rest of this section is devoted to proving the converse of Theorem \ref{th-nec-cond-nonempty} (Theorem \ref{th-suff-cond-nonempty}).

First we introduce a foliation on $X$ using the Kodaira map of a linear system. 
Let $E$ be a base point free finite dimensional linear subspace of global sections of a line bundle $\L$ on $X$.  
Let $\Phi_E: X\to \p(E^*)$ be the corresponding Kodaira map. Below we use the following notation:
$S$ is the singular locus of the variety $X$. The number $\tau_E$ is the dimension of $Y_E$, the closure of the image of the Kodaira map $\Phi_E$. The set 
$S_E$ is the singular locus of $Y_E$ and $U$ is the Zariski open set $X \setminus (S \cup \Phi_E^{-1}(S_E))$ in $X$. Finally, $\Sigma_c \subset U$ is the set of critical points of  
$\Phi_E$ restricted to $U \subset X$.

%\begin{itemize}
%\item $S \subset X$ is the set of  singular points of $X$;
%\item $X_{L}$ is the  image  $\Phi_L(X)$ of $X$ under the Kodaira map;
%\item $\tau_L=\dim X_L$;
%\item $S_L\subset X_L$ is the set of  singular points of $X_L$;
%\item $U$ is the Zariski open set $X \setminus (S \cup \Phi_L^{-1}(S_L))$ in $X$
%\item $\Sigma_c\subset U$ is the set of critical points of  $\Phi_L$ restricted to $U\subset X$;
%\end{itemize}

\begin{Def} 
Let $a \in U \setminus \Sigma_c$. Let $\tilde{F}(a)$ be the subspace of the tangent space $T_aU$ defined by the linear equations $dg_a = 0$ for all $g \in E$.
The collection of subspaces $\tilde{F}(a)$ defines an $(n - \tau_E)$-dimensional distribution $\tilde{F}$ on the Zariski open set $U \setminus \Sigma_c$ in $X$.
%On the Zariski open set $U\setminus \Sigma_c\subset X$ we define  the following $(n - \tau _L)$-dimensional distribution $\tilde{F}$ 
%(i.e. an $(n - \tau_L)$-dimensional subbundle of the tangent bundle of the manifold $U \setminus \Sigma_c$): For each $a \in U \setminus \Sigma_c$, the subspace 
%$\tilde{F}(a)$ of the tangent space $T_aU$ is given by the linear equations $dg_a=0$ for all $g\in L$.
\end{Def}

The next lemma is a corollary of the Implicit Function Theorem.
\begin{Lem} \label{lem-11}
The foliation $\tilde{F}$ in $U\setminus \Sigma_c$  is completely integrable. Its leaves are  connected components of the preimages under the Kodaira map 
$\Phi_E: U \setminus \Sigma_c \to Y_E$ of the points in $\Phi_E( U\setminus \Sigma_c)$.
\end{Lem}
%\begin{proof} Lemma follows from the implicit  function theorem.
%\end{proof}

Take a point $a \in U \setminus \Sigma_c$ and a section $g \in E$ such that $g(a)=0$ and $dg(a)\neq 0$. 
Let $H$ be the hyperplane in the tangent space $T_aU$  defined by $dg=0$. The following is straightforward:
\begin{Lem} \label{lem-12}
1) The hyperplane $H$ contains the $(n-\tau_E)$-plane $\tilde{F}(a)$. 2) For  any hyperplane $H \subset T_aU$ which contains the plane $\tilde{F}(a)$ there is a section 
$g \in E$ such that $g(a)=0$ and $H$ is defined by the equation $dg_a=0$.
\end{Lem}

%\begin{proof} The statement 1) is obvious. To prove 2) note that that the differential $DK$ of the map $K$ at the point $a$ identifies  the factor space 
%$T_a/\tilde{F}(a)$ with the tangent space $\tilde T_b$  to $X_L$ at the point $b=K(a)$. Let $\tilde H\subset \tilde T_b$ be the image of $H$ under the differential 
%$DK:T_a\to \tilde T_b$. Consider $\tilde H$ and  $\tilde T_b$ as the subspaces in the projective space $\p(L^*)$ and take dual spaces $\tilde H^*\supset \tilde T^*_b$ 
%in the dual projective space $\p(L)$. Any element $p\in H^*\setminus \tilde T^*_b$ represent a section $g\in L$ defined up to a nonzero factor which satisfies the needed 
%conditions: $g(a)=0$ and the equation $dg(a)=0$ in $T_a$ determines the hyperplane $H\subset T_a$.
%\end{proof}

%\subsubsection{A problem from linear algebra} \label{subsec-2.4}

We will need a linear algebra statement about transversality of a collection of hyperplanes.
Let $F_1,\dots,F_k\subset T$ be  subspaces in a vector space $T$. For any nonempty set $J=\{i_1,\dots,i_j\} \subset \{1, \ldots, k\}$ let 
$F_J$ denote the the subspace $F_{i_1} \cap \dots \cap F_{i_j}$.
The following theorem gives a necessary and sufficient condition for the existence of hyperplanes 
$H_1, \dots, H_k \subset T$ such that: (1) $F_i\subset H_i$ for $1\leq i\leq k$, and (2) the hyperplanes $H_1,\dots, H_k$ are mutually transverse.
\begin{Th} \label{th-13}
The hyperplanes $H_1,\dots,H_k$ satisfying the above conditions exist if and only if for any subset $J$ the codimension of $F_J$ in $T$ is greater than or equal to $|J|$.
\end{Th}
\begin{proof} 
%Assume that $H_1,\dots, H_k$ are  needed hyperplanes.  For a nonempty set $J=\{i_1,\dots,i_j\}$  denote by $H_J$ the space $H_J=H_{i_1} \cap \dots \cap H_{i_j}$. 
%Then  $F_{J}\subset H_{J}$ and codimension of $H_{J}$ in $T$  is  $|J|$. So codimension of $F_{J}$ in $T$ is  $\geq |J|$.
First suppose the subspaces $H_1, \ldots, H_k$ satisfying the above conditions exist. For any nonempty subset $J = \{i_1, \ldots, i_j\}$ let $H_J$ denote the 
subspace $H_{i_1} \cap \cdots \cap H_{i_j}$. Then $F_J \subset H_J$ and the codimension of $H_J$ in $T$ is $|J|$. So the codimension of $F_J$ is greater than or equal to $|J|$. 
We prove the converse statement by induction on $k$.
%The converse statement we will proof by induction in $k$. Assume that the statement is proved for $(k-1)$ subspaces.  As $H_k$ let us choose a hyperplane satisfying the following condition: $H_k$ %contains $F_k$ but does not contain any space $F_J$ not belonging to $F_k$ (so $F_J \cap H_k$ has codimention one in $F_J$).
Suppose that the statement holds for any collection of $(k-1)$ subspaces. Choose $H_k$ to be a hyperplane such that $H_k$ contains $F_k$ but it does not contain any 
subspace $F_J$ not inside $F_k$ (so $F_J \cap  H_k$ has codimension $1$ in $F_J$).
%Now let us apply the statement for $(k-1)$ subspaces $F'_1,\dots F'_{k-1}$ in an ambient space $H_k$ where $F_i'=F_i\cap H_k$  for $1\leq i\leq k-1$. 
%Let us check that the subspaces $F'_1,\dots,F_{k-1}'\subset H_k$ satisfy the conditions of the theorem. For any set $J$ define a set $J^*$ as follows: a) if $k\notin J$ define 
%$J^*$ as $J$; b) if $k\in J$ define $J^*$ 
%as $J\setminus \{k\}$.
Now we apply the induction hypothesis to the $k-1$ hyperplanes $F'_i = F_i \cap H_k$, $i=1, \ldots, k-1$, in the vector space $H_k$. Let us verify that these satisfy the conditions in the theorem.
For any subset $J \subset \{1, \ldots, k\}$ let $J^* = J \setminus \{k\}$. 
If $F_{J^*}\subset F_k$ then $F_{J^*}=F_{J^*}\cap F_k=F_J$ where $J = J^* \cap \{k\}$. The codimension of $F_{J^*} \cap H_k$ in $H_k$ is equal to the codimension of $F_J$ in $T$ minus $1$. By 
the assumption, $\codim(F_J) - 1$  is greater than or equal to  $(|J^*|+1)-1=|J^*|$. On the other hand, 
if $F_{J^*}$ is not contained in $F_k$ then its codimension in $T$ is equal to the codimension of  
$F_{J^*}\cap H_k$ in $H_k$. Again by the assumption in the theorem this is greater than or equal to $|J^*|$.
Finally by the induction hypothesis there are mutually transverse hyperplanes $H_1',\dots, H_{k-1}'$ in $H_k$, such that $F_{i} \cap H_k \subset H_i'$ for $i<k$. Enlarge each $H_i'$ to a 
hyperplane $H_i$ in $T$ such that $H_i$ contains $F_i$. The collection of hyperplanes $H_1, \ldots, H_k$ has the required properties.
\end{proof}

%\subsubsection {Sufficient condition for non emptiness} \label{subsec-2.5}
%In this section we prove the following converse of Theorem \ref{th-nec-cond-nonempty}:
We can now prove the converse of Theorem \ref{th-nec-cond-nonempty}.
\begin{Th} [Sufficient condition for a generic complete intersection to be nonempty] \label{th-suff-cond-nonempty}
Let $E_1, \ldots, E_k$ be a collection of base point free linear systems. Suppose that $E_1, \ldots, E_k$ are independent in the sense of Definition \ref{def-independent}. 
Then a generic complete intersection $Z_{\bf g} \subset X$, where ${\bf g} \in {\bf E} = E_1 \times \cdots E_k$, is nonempty.
 %If for the collection of
%spaces $L_1,\dots,L_k$ the defect $d_J$ of each nonempty subset $J\subset I=\{1,\dots,k\}$ is positive then the generic variety $Z_{\bf g} \subset X$ is not empty.
\end{Th}
\begin{proof} 
Let $J=\{i_1,\dots,i_j\}$ be a nonempty subset of $I=\{1,\ldots,k\}$. Consider the Kodaira map $\Phi_J$ associated with the space 
$E_J$.  Using  Lemma \ref{lem-11} one can find a smooth Zariski open subst $U\subset X$ and integrable foliations $\tilde{F}_J$ in $U$ of codimensions $\tau_J$ whose leaves are connected components of the manifolds $\Phi^{-1}_J(b) \cap U$ where $b \in \Phi_J(U)$. Take a point $a \in U$ and let $T=T_aU$ be the tangent space to $U$ at $a$. Let
$F_1,\dots, F_k$ be the subspaces in $T$ tangent to the leaves of these foliations. For any nonempty $J = \{i_1, \ldots, i_j\}$ 
the intersection $F_J = F_{i_1} \cap \cdots \cap F_{i_j} \subset T$ coincides with the subspace tangent to the leaf of the foliation $\tilde{F}_J$ passing through $a$ 
(Lemma \ref{lem-9-Kodaira}). 
The codimension $\tau_J$ of $\tilde{F}_J$ in $T$ is  $d(J)+|J|\geq |J|$. Thus applying Theorem \ref{th-13} we can find mutually transverse hyperplanes $H_1\dots, H_k\subset T$ 
such that $\tilde{F}_i \subset H_i$ for $i=1, \ldots, k$. By Lemma \ref{lem-12} there are sections $g_i \in E_i$, $i=1, \ldots, k$ such that $g_i(a) = 0$ and 
the tangent hyperplane to $\{g_i = 0\}$ at $a$ is $\tilde{F}_i$. Hence the hypersurfaces $g_i=0$ at a neighborhood of the point $a$ are smooth and mutually transverse. 
By the Implicit Function Theorem we then know that for a $k$-tuple of sections $\tilde {\bf g}= (\tilde g_1,\dots, \tilde g_k)$ close enough to $g$ the variety $Z_{\tilde g}$ is nonempty. 
Thus it is not true that for a generic ${\bf g}$ the variety $Z_{{\bf g}}$ is empty. This finishes the proof of the theorem.
\end{proof}

\subsection{The $h^{p,0}$ numbers of a complete intersection} \label{subsec-h-p0}
%With a divisor $D$ on a smooth algebraic variety $X$ one associates a sheaf  $\Omega (X,D)$ of germs of regular sections of the line bundle $\L$ on $X$ corresponding to the divisor $D$. Denote %by $H^*(X,D)$ cohomology of $X$ with coefficients in
%$\Omega (X,D)$. Zero divisor we will omit in our notations and will denote by  $H^*(X)$  cohomology of $X$ with the coefficients  in the sheaf of germs of regular functions. 
%Cohomology groups $H^*(X)$ are responsible for the dimensions $h^{p}(X)$ of the spaces of holomorphic  $p$-forms on  $X$, i.e.
%$h^{p}(X)=\dim H^p(X)$. In particular   $\dim H^0(X)$
%is equal to the number of irreducible components of $X$ and $H^p(X)=0$ for
%$p>\dim X$.
The material in this section are taken from \cite{Askold-new}.
Let $\L$ be a line bundle on a smooth projective variety $X$ and let $D$ be a divisor of $\L$. We denote by $\O(X, \L)$ or $\O(X, D)$ the germ of regular sections of 
$\L$. We denote the $i$-th sheaf cohomology group of $X$ with coefficients in the sheaf $\O(X, \L)$ 
by $H^i(X, \L)$ or $H^i(X, D)$. We will write $H^i(X)$ for the $i$-th cohomology group of the zero divisor, that is, $H^i(X)$ is
the $i$-th  cohomology of $X$ with coefficients in $\mathcal{O}_X$, the sheaf of germs of regular functions. For an integer $p \geq 0$, we have $\dim(H^p(X)) = h^{p, 0}(X)$, 
the dimension of the space of holomorphic $p$-forms on $X$. In particular, $h^{0,0}(X)$ is 
the number of irreducible components of $X$ and $h^{p,0}(X) = 0$ for $p > \dim(X)$.

%Let $L_1, \ldots, L_k$ be linear subspaces in a finite dimensional vector space $L$. Recall that the $L_i$ are called \emph{mutually transverse} if for any subset $J \subset \{1, \ldots, k\}$ 
%the codimension of the intersection $\bigcap_{i \in J} L_i$ is equal to the sum of codimensions of the $L_i$, $i \in J$. A collection of submanifolds $X_1, \ldots, X_k$ in a smooth manifold
%$X$ are \emph{mutaullay transverse} if at any point $a$ in the intersection of any subset of the $X_i$ the tangent spaces to the $X_i$ are mutually transverse in $T_aX$.

%Let $\L_1,\dots,\L_k$ be  fixed linear bundles on $X$. We will be interested in a smooth complete intersection $X_k$ defined by a system $f_1=\dots=f_k=0$ where 
%$f_i$ is a global section of $\L_i$. Assume that the equation $f_i=0$  defines a smooth divisor $D_i$ and that all divisors $D_1,\dots, D_k$ are mutually transverse.

Fix a collection $\L_1, \ldots, \L_k$ of globally generated line bundles on $X$. Recall that a line bundle is globally generated if for any $x \in X$ there is 
a global section that does not vanish at $x$. For each $i=1, \ldots, k$ let 
$f_i \in H^0(X, \L_i)$ be such that the divisor $D_i$ defined by $f_i = 0$ is a smooth hypersurface. Moreover, assume that the divisors $D_1, \ldots, D_k$ intersect transversely.
We will be interested in the local complete intersection $X_k = D_1 \cap \cdots \cap D_k$. From transversality it follows that this intersection is a smooth subvaritey of $X$. 

%One can obtain a significant information about the invariants
%$h^{p}(X_k)$ of the complete intersection $X_k$, knowing dimensions of the cohomology groups
%$$H^*(X,\L_1^{-n_1}\otimes \dots \otimes \L_k^{-n_m}),\,\, n_i=\{0,1\},$$ of the ambient manifold $X$. Let us remind how to do it
%(see \cite{[6]}).
Given the dimensions of the cohomology groups: 
$$H^i(X, \L_1^{\otimes m_1} \otimes \cdots \otimes \L_k^{\otimes m_k}), \quad m_1, \ldots, m_k \in \{0, -1\},$$ of the ambient variety $X$, one 
can obtain much information about the $h^{p,0}$ numbers of the complete intersection $X_k$. In particular one can compute the arithmetic genus of $X_k$.
Below we recall how this can be done (\cite{Askold-genus}).

%\subsubsection{Exact sequences} \label{subsec-exact-seq} 
%For every  $m$, $1 \leq m \leq k,$ defined the manifold $X_m$ by  $X_m=D_1^0\bigcap\dots,
%\bigcap D_m^0$.  In the sequence  $X=X_0\supset \ldots \supset X_k$
%each next manifold is a hypersurface in the previous manifold. For every linear bundle  $\L$ for every $m$, $1\leq m\leq
%k$, consider an exact sequence of sheafs

For $m=1, \ldots, k$ let $X_m = D_1 \cap \cdots \cap D_m$. We then have a sequence of smooth subvarieties 
$X_k \subset \cdots \subset X_0 = X$ where each variety is a hypersurface in the next one. For a line bundle $\L$ consider the exact sequence of sheaves: 
$$ 0\to \O(X_{m-1}, \L \otimes \L^{-1}_m) \stackrel{i}\to \O(X_{m-1}, \L)  \stackrel{j}\to \hat{\O}(X_{m-1}, \L)\to 0.$$

Here $\O (X_{m-1},\L\otimes \L^{-1}_m)$  is the sheaf on $X_{m-1}$ of  regular sections of the line bundle $\L\otimes \L^{-1}_m$  on $X_{m-1}$.
The sheafs $\O(X_{m-1}, \L)$ and $\O(X_{m},  \L)$ have analogues definitions. The sheaf $\hat{\O}(X_{m-1}, \L)$ is the trivial  extension of the sheaf  $\O(X_{m},\L)$ on 
$X_m$ to a sheaf on $X_{m-1}$. The homomorphism  $i$ maps a section $g$ of $\L\otimes \L_m^{-1}$ to a section $g \otimes f_m$ where $f_m$ is a global section of 
$\L_m$ defining the divisor $D_m$, and the homomorphism $j$ at $a \in X_m$ maps a section of $\L$ on  $X_{m-1}$ to its restriction to $X_m$, and at a point 
$a \in X_{m-1} \setminus X_m$ the homomorphism $j$ is trivial. The long exact sequence of the cohomology groups corresponding to the above exact sequence of sheaves 
is as follows:
\begin{equation} \label{equ-11}
0 \to H^0(X_{m-1}, \L \otimes \L^{-1}_m) \to H^0(X_{m-1}, \L) \to H^0 (X_{m}, \L) \to \cdots,
\end{equation}
(the cohomology of the sheaves $\O(X_{m}, \L)$ and  $\hat{\O}(X_{m-1}, \L)$ are canonically isomorphic).

For a complete variety $X$ with a sheaf $\mathcal{F}$ we denote the Euler characteristic of $X$ with coefficients in $\mathcal{F}$ by $\chi(X, \mathcal{F})$:
$$\chi(X, \mathcal{F}) = \sum_{i=0}^n (-1)^i \dim(H^i(X, \mathcal{F})).$$
In particular, we write $\chi(X, \L)$ for the Euler characterisitic of $X$ with coefficients in the sheaf $\O(X, \L)$ of sections of a line bundle $\L$.

The Euler characteristic is additive, i.e. if $\mathcal{G}$ is a sub-sheaf of sheaf $\mathcal{F}$ on $X$ then:
$$\chi(X, \mathcal{F}) = \chi(X, \mathcal{G}) + \chi(X, \mathcal{F}/\mathcal{G}),$$ where $\mathcal{F} / \mathcal{G}$ is the quotient sheaf.
The exact sequence \eqref{equ-11} then allows us to find the Euler characteristic $\chi (X_k, \L)$. We give the answer for the trivial bundle, that is, $\chi (X_k)$:
\begin{Th} \label{th-genus-complete-int}
The arithmetic genus  $\chi(X_k)$ of the smooth variety $X_k$ is equal to:
$$\chi (X) - \sum_{i_1} \chi(X,\L_i^{-1}) + \sum_{i_1<i_2} \chi(X,\L_{i_1}^{-1} \otimes \L_{i_2}^{-1}) - \cdots + (-1)^k \chi(X, \bigotimes_{1\leq i \leq k} \L_i^{-1}).$$
\end{Th}

For a nonempty set $J \subset \{1,\dots, k\}$ let $\L_J^{-1}=\bigotimes_{i \in J}\L_i^{-1}$.
\begin{Th} \label{th-22}
We have the following upper bound for the $h^{i,0}$ numbers of the complete intersection $X_k$:
\begin{equation} \label{equ-*}
h^{i,0} (X_{k}) \leq h^{i,0}(X)+ \sum_{J \neq \emptyset} \dim(H^{i+|J|}(X, \L_J^{-1})).
\end{equation}
\end{Th}
\begin{proof} 
%We will rewrite the estimation (*) generalize it and proof the generalized estimation.   For $J=\emptyset$ let us defined
%$\L_J^{-1}$ as the trivial line bundle over $X$. Because $\dim
%H^{i}(X)= h^i(X)$ the estimation (*) can be rewritten  in the form $h^i (X_{k})\leq \sum_J\dim
%H^{i+|J|}(X, \L_J^{-1})$ (where the summation is taken over all subsets $J$ including  $J=\emptyset$). 
We can rewrite \eqref{equ-*} as $h^i (X_{k})\leq \sum_J \dim(H^{i+|J|}(X, \L_J^{-1}))$. Let $\L$ be any line bundle on $X$.
We will prove the following more general inequality which coincides with \eqref{equ-*} when $\L$ is the trivial line bundle:
%Now let us a more general inequality, which coincides with (*) 
%when  the trivial  $\L$ is trivial:  for any line bundle $\L$ the following inequality holds:
\begin{equation} \label{equ-**}
\dim(H^i (X_{k}, \L)) \leq \sum_J \dim(H^{i+|J|}(X, \L\otimes \L_J^{-1})). 
\end{equation}
We prove \eqref{equ-**} by induction on $k$.
%Let us proof (**) by induction in $k$.
%$\L$ and  $(j, m)$ where $0\leq j$ and  $1\leq m\leq k,$  an inequality
Let $j \geq 0$ and $1 \leq m \leq k$. From the piece:
$$
\to H^j (X_{m-1},\L)\to H^j (X_{m},\L)\to
H^{j+1} (X_{m-1},\L\otimes \L_m^{-1})\to \dots
$$
of the exact sequence \eqref{equ-11} we obtain that:
\begin{equation} \label{equ-12}
\dim(H^j (X_{m},\L)) \leq \dim(H^j(X_{m-1},\L)) + \dim(H^{j+1} (X_{m-1},\L\otimes \L^{-1}_m)).
\end{equation}
%follows from the piece
%$$
%\to H^j (X_{m-1},\L)\to H^j (X_{m},\L)\to
%H^{j+1} (X_{m-1},\L\otimes \L_m^{-1})\to \dots $$ of the exact sequence (11).
For $k=1$,  the inequality \eqref{equ-**} coincides with  \eqref{equ-12} for
$j=i$ and $m=1$. Assume that \eqref{equ-**} is proved for $k-1$.
For  $J\subset \{1,\dots,k\}$ let $J^*=J \cap \{1,\dots,k-1\}$. 
Then  either $J=J^*$ or $J = J^* \cup \{k\}$. In the first case we have:
\begin{equation} \label{equ-13}
|J| = |J^*|   \textup{ and } \L \otimes \L_{J}^{-1}=\L \otimes \L_{ J^*}^{-1}. 
\end{equation}
In the second case we have:
\begin{equation} \label{equ-14}
|J|=|J^*|+1  \textup{ and } \L\otimes \L_ {J}^{-1}=\L\otimes \L_ {J^*}^{-1}\otimes \L_ k^{-1}.
\end{equation}
By induction hypothesis we can assume that for any line bundle $\L$ and $i \geq 0$ the following inequality holds:
\begin{equation} \label{equ-15}
\dim(H^i (X_{k-1},\L)) \leq \sum_{J^*} \dim(H^{i+|J^*|}(X,\L\otimes \L_{J^*}^{-1})),
\end{equation}
where the summation is over all the subsets  $J^*\subset \{1,\dots,k-1 \}$.
Also by induction hypothesis we know that for the line bundle $\L\otimes  \L_k^{-1}$ and any 
$i+1$ the inequality:
\begin{equation} \label{equ-16}
\dim(H^{i+1}(X_{k-1}, \L\otimes \L_k^{-1})) \leq \sum_{J^*} \dim
(H^{i+1+|J^*|}(X, \L \otimes \L_k^{-1} \otimes \L_{J^*}^{-1})),
\end{equation}
holds. Here also the summation is taken over all subsets $J^*\subset \{1,\dots,k-1\}$.
Now instead of the numbers $\dim(H^i(X_{k-1}, \L))$ and $\dim(H^{i+1}(X_{k-1}, \L \otimes \L_k^{-1}))$ 
plug the righthand sides of \eqref{equ-15} and \eqref{equ-16} into \eqref{equ-12}. Using \eqref{equ-13} and \eqref{equ-14} we obtain the required 
inequality \eqref{equ-**}. The theorem is proved.
%Instead of the numbers $\dim H^i (X_{k-1},\L)$ and $\dim H^{i+1}
%(X_{k-1},\L\otimes \L_k^{-1})$ plugging into  (12) the righthand sides of (15) and
%(16) and using  (13) and (14),
%we obtain the needed inequalities for any $k$, $i$ and any line bundle $\L\leq 0$. The theorem is proved.
\end{proof}

We have the following direct corollary of Theorem \ref{th-22}:
\begin{Cor} \label{cor-23}
Assume that for some integer $i$ the cohomology groups $H^{i+|J|}(X, \L_J^{-1})$ vanish for any nonempty $J \subset \{1, \ldots, k\}$. Then $h^{i,0}(X_k) \leq h^{i,0}(X)$.
\end{Cor}

\section{Complete intersections in toric varieties} \label{sec-toric-12}
The material in this section are taken from \cite{Askold-genus} and \cite{Askold-new}. We explain the results on genus of complete intersections in 
torus $(\c^*)^n$ from \cite{Askold-genus} concerning the case when the Newton polytopes involved are full dimensional, as well as  
their extensions in \cite{Askold-new} to the case when the polytopes are not necessarily full dimensional.

%I am not translating this section. Just will translate a small part of it. Let me start with a notation I am often referring to.
Let $\Delta_1, \ldots, \Delta_k$ be integral polytopes in $\r^n$ (i.e. with vertices in $\z^n$). Let $f_1, \ldots, f_k$ be a generic $k$-tuple of Laurent polynomials
with Newton polytopes $\Delta_1, \ldots, \Delta_k$ respectively. Let $Z$ be the subvariety of $(\c^*)^n$ defined by $f_1 = \cdots = f_k = 0$.
 
%Consider an algebraic variety~$Z$ defined in $(\c^*)^n$ by a generic system of equations $$f_1=\dots=f_k=0,$$ where $f_1,\dots,f_k$ --
%is a generic $k$-tuple of  Laurent polynomials with fixed Newton polyhedra  $\Delta_1,\dots, \Delta_k\subset
%\r^n$.
%Now I translate the very end of the section 12.
%{\rm Because the divisor $\Delta^{\infty}$ is invariant under the torus action, the theory of toric varieties allows to compute the dimensions of all cohomology groups  $H^*(X,\Delta^{\infty})$
%(see for example [7]). The answer turns out to be  a specially simple if the $F_X$ is sufficiently complete for the polyhedron $\Delta$.
%Recall that {\it $N^\circ(\Delta)$ is the number of integral points, lying in the interior of the polyhedron
%$\Delta$ (the interior is understood here in the topology of the  minimal
%affine space  $L(\Delta)$ containing  $\Delta$.}
%The following theorem holds  (see theorem from the section $\S 4$ in [1] \footnote{ Here and below we use notations a little bit different from the notations in [1], [2]. 
%The main difference is in sign.  Divisors and cohomology that we denote by  $D^\infty$, $\Delta^\infty$, $H^*(X,\Delta^\infty)$ in [1], [2] were denoted by $-D_\infty$, $-\{\Delta\}$, 
%$H^*(X,\{-\Delta\})$.}.

Let $X$ be a smooth projective toric variety which is {\it sufficiently complete} with respect to the polytopes $\Delta_1, \ldots, \Delta_k$ in the sense of \cite{Askold-toroidal}.
We let $D_\infty$ denote the divisor at infinity of a generic Laurent polynomial with Newton polytope $\Delta = \Delta_1 + \cdots + \Delta_k$ on $X$. It is a divisor supported
on the complement of the open orbit $(\c^*)^n$ and hence is torus invariant. Let $\L_\infty$ be the line bundle
corresponding to $D_\infty$.

Since the divisor $D_\infty$ is torus invariant one can use theory of toric varieties to compute the sheaf cohomology groups of $D_\infty$ in terms of 
the polytope $\Delta$ (\cite[Section 4]{Askold-toroidal}). We recall the answer below.
We need a bit of notation: 
for a polytope $\Delta$ we denote the number of integral points in $\Delta$ by $N(\Delta)$. Also $N^\circ(\Delta)$ denotes the number of integral points 
in the interior of the polytope $\Delta$. Here the interior is with respect to the topology of the affine span of $\Delta$. We also write $N'(\Delta)$ for 
$(-1)^{\dim(\Delta)} N^\circ(\Delta)$. \footnote{In \cite{Askold-toroidal, Askold-genus} instead of our notation $N(\Delta)$, $N^\circ(\Delta)$ and $N'(\Delta)$ respectively
the notation $T(\Delta)$, $B^+(\Delta)$ and $B(\Delta)$ is used.} 

\begin{Th} \label{th-24}
We have:
$$ \dim(H^i(X, \L_\infty^{-1})) = 
\begin{cases} 
0, \quad i \neq  \dim(\Delta)\\
N^\circ(\Delta), \quad i=\dim(\Delta) \\
\end{cases}$$
\end{Th}

%\begin{Th} \label{th-24}
%Let  $F_X$  be a  sufficiently full fan for $\Delta$. Then $H^i(X,\Delta^{\infty})=0$ for
%$i\neq d$ where $d=\dim \Delta$ and  $\dim
%H^d(X,\Delta^{\infty})=N^\circ(\Delta).$
%\end{Th}

We note that if $X$  is sufficiently complete for the polytope $\Delta = \Delta_1+\cdots+\Delta_k$ then it is also sufficiently complete for $m_1\Delta_1+ \cdots + m_k\Delta_k$ 
for any integers $m_i \geq 0$. In particular, the cohomology groups of the line bundles associated to $m_1\Delta_1+ \cdots + m_k\Delta_k$ are also given by Theorem \ref{th-24}.
%\begin{Cor} \label{cor-25}
%If $X$ is sufficiently complete
%for $T$ then for any polyhedron $\Delta=n_1\Delta_1+\dots+ n_k\Delta_k$ for $n_i\geq 0$
%all dimensions of the cohomology groups  $H^*(X,\Delta^\infty)$ are listed in the theorem~24.
%\end{Cor}
%To simplify the notation, for a convex polytope $\Delta$ we denote $(-1)^{\dim \Delta}N^\circ(\Delta)$ by $N'(\Delta)$. 
%Recall that $N^\circ(\Delta)$ is the number of integral points in the interior of $\Delta$.
\begin{Cor} \label{cor-26}
Let  $X$ be a sufficiently complete for
$\Delta$. Then the Euler characteristic $\chi(X, \L_\infty^{-1})$ is equal to $N'(\Delta)$.
\end{Cor}

\begin{Cor} \label{cor-27} 
For any smooth projective toric variety $X$ we have $h^{0, 0}(X) = 1$ and $h^{i, 0}(X) = 0$ for $i>0$.
\end{Cor}
\begin{proof} The numbers $h^{i,0}$ are birational invariants and any $n$-dimensional toric variety is birationaly equivalent to $\c\p^n$. 
For $\c\p^n$ the corollary is obvious. Nevertheless let us deduce the corollary from Theorem \ref{th-24}. Let $\Delta=\{0\}$ be the polytope consisting of the single point $0$. 
Then $\dim(\Delta)=0$ and $N^\circ(\Delta)=1$.  Any smooth projective toric variety $X$ is sufficiently complete for $\Delta=\{0\}$ and the divisor $D_\infty$ on 
$X$ is $\{0\}$. Now apply Theorem \ref{th-24} for $\Delta=\{0\}$.
\end{proof}

%It is well known that (see for example \cite{7}) that a toric compactification $X$ is smooth projective manifold if its fan $F_X$  satisfies the following conditions:
%(1) every cone $\sigma\in F_X$ is a simpletial  cone;
%(2) for every $\sigma\in F_X$ the primitive integral vectors
%belonging to the edges of $\sigma$ generate the lattice  $\z^n\cap L(\sigma)$ where $L(\sigma)$ is the minimal linear space which contains $\sigma$;
%(3) the fan $F_X$ is dual to some integral polyhedron $\Delta$ (i.e. $F_X=\Delta^\bot$).
%It is well known that any fan $F$ could be subdivided  in  such a way that the subdivision become a fan of a smooth projective variety
%((there are infinitely many such subdivisions, see, for example [7]).

%\subsubsection{Sufficient condition for irreducibility} \label{subsec-13}
%Below using results of the sections 11 and 12  we prove the theorem on irreducibility and we present some computations related to the $h^{p}(Z)$ numbers of the 
%algebraic variety $Z$ defined by the system (1).
We now use the results in Section \ref{subsec-h-p0} to give a condition for when a generic complete intersection $Z$ is irreducible. We also 
prove a result about the $h^{p,0}$ numbers of $Z$.

%Let  $\Delta_1,\dots, \Delta_k$ be the Newton polyhedra of  Laurent polynomials $f_1,\dots,f_k$ appearing in the system
%(1). We will assume that the set of polyhedra is independent (i.e. the defect of every subset is not negative).  Otherwise the generic system of 
%equations (1) has no solutions. In this section by $X$ we mean any (fixed) smooth projective toric compactification of  $(\c^*)^n$, whose fan $F_X$ is a subdivision 
%of the fan $\Delta^\perp$, where
%$\Delta=\Delta_1+\dots+\Delta_k$. We will denote by  $X^\infty$ union of closures  of all  $(n-1)$-dimensional orbit of $X$. Since $X$ is a smooth toric variety  $X^\infty$ is 
%a normal crossing divisor  (see for example [7]).
For an integral polytope $\Delta \subset \r^n$ let $L_\Delta$ denote the linear subspace of Laurent polynomials in $(\c^*)^n$ spanned by all the monomials $x^\alpha$ where 
$\alpha \in \Delta \cap \z^n$. It is easy to see that $L_\Delta$ has no base point on $(\c^*)^n$. 
Let $\Phi_\Delta: (\c^*)^n \to \p(L_\Delta^*)$ denote its Kodaira map. One observes that 
the dimension of the image of $\Phi_\Delta$ is equal to $\dim(\Delta)$.

As above, let $\Delta_1, \ldots, \Delta_k$ be integral polytopes in $\r^n$. For each $i=1, \ldots, k$ let $L_i = L_{\Delta_i}$ be the corresponding subspace of Laurent polynomials. 
From the above it follows that the defect of a subset $J \subset \{1, \ldots, k\}$ is equal to:
$$d(J) = \dim(\Delta_J) - |J|,$$
where $\Delta_J = \sum_{i \in J} \Delta_i$. We call $\Delta_1, \ldots, \Delta_k$ {\it independent} if the corresponding subspaces $L_1, \ldots, L_k$ 
are independent (see Definition \ref{def-independent}). In other words, 
$\Delta_1, \ldots, \Delta_k$ are independent if $\dim(\sum_{i \in J} \Delta_i) \geq |J|$ for any subset $J \subset \{1, \ldots, k\}$.

Now let 
\begin{equation} \label{equ-1-f_i}
f_1(x) = \cdots = f_k(x) = 0 
\end{equation}
be a generic system of Laurent polynomials with Newton polytopes
$\Delta_1, \ldots, \Delta_k$ respectively defining a complete intersection $Z$ in $(\c^*)^n$.
We will assume that the polytopes $\Delta_1, \ldots, \Delta_k$ are independent. 
This guarantees that $Z$ is nonempty (Theorem \ref{th-nec-cond-nonempty}). 
As before we let $X$ be a fixed smooth projective toric variety with $(\c^*)^n$ as the open orbit whose fan is a subdivision of the normal fan of
the polytope $\Delta = \Delta_1 + \cdots + \Delta_k$ (then $X$ is sufficiently complete with respect to the polytopes $\Delta_1, \ldots, \Delta_k$). Also let
$X_\infty$ denote the sum of 
prime divisors in $X$ which lie in the complement of the open orbit $(\c^*)^n$. It is well-known that since $X$ is smooth, the divisor 
$X_\infty$ has normal crossings.

%The open manifold $X \setminus D^\infty$ is the torus $(\c^*)^n$. We will use notations from the section 11 keeping in mind  $X$ and $X^\infty$  introduced above. 
%For  applicability of the results from the sections 11--12 to the manifold $Z$, defined by a generic system (1) the following lemma is needed
To apply the result in Sections \ref{subsec-h-p0} to the variety $Z$ we need the following lemma. For a proof see \cite{Askold-toroidal}.
\begin{Lem} \label{lem-28}
With notation as above, let $D_i$ be the closure of the hypersurface defined by $f_i = 0$ in $X$, for $i=1, \ldots, k$. Then $D_i$ is a smooth hypersurface and 
moreover all the  divisors $D_i$ and the closures of all the $(n-1)$-dimensional orbits are mutually transverse in $X$.
\end{Lem}
%\begin{proof} For every nonempty subset $J=\{i_1,\dots, i_j\}$ of the set  $\{1,\ldots,k\}$ consider the corresponding subsystem $f_{i_1}=\dots=f_{i_k}=0$ of the system \eqref{equ-1-f_i}.  
%The following is a necessary and sufficient condition on the $k$-tuple of Laurent polynomials $f_1,\dots, f_k$ to be able to apply Lemma \ref{lem-28}: for every subset $J$ the 
%corresponding subsystem has to be {\it $\Delta$-nondegenerate} with respect to its Newton polyhtopes $\Delta_{i_1},\dots, \Delta_{i_j}$ (see \cite{Askold-toroidal} for the definition of a 
%non-degenerate system). This statement follows automatically from the theorem from the section   2  of the paper [1].
%\end{proof}
Assume that the variety $Z$ is nonempty (i.e. the generic system \eqref{equ-1-f_i} has solutions) and the conditions of Lemma \ref{lem-28} hold. 
Then the closure of $Z$ in $X$ is the intersection of the smooth divisors  $D_1, \ldots, D_k$. As before, for each $i$ let $D_{i, \infty}$ denote the 
divisor at infinity on the toric variety $X$ associated to the polytope $\Delta_i$ and let $\L_{i, \infty}$ be its corresponding line bundle.
%The corresponding  $k$-tuple of invariant divisors  $D_1^\infty, \ldots,D_k^\infty$ is the $k$-tuple  $\Delta_1^\infty,\dots,\Delta_k^\infty$. 
%We now apply the results of Section \ref{subsec-11} to $Z$. 
Theorem \ref{th-24} and the remark after it, give us the needed information about the dimensions of the cohomology groups 
$H^i(X, \L_{1, \infty}^{\otimes m_1} \otimes \cdots \otimes \L_{k, \infty}^{\otimes m_k})$, $m_1, \ldots, m_k \in \{0,-1\}$.
\begin{Th} \label{th-29}
The arithmetic genus $\chi(Z)$ of $Z$, defined by a genreic system \eqref{equ-1-f_i}, is given by: 
\begin{equation} \label{equ-17}
\chi(X)=1 - \sum _{i_1} N'(\Delta_{i_1}) + \sum_{i_1<i_2} N'(\Delta_{i_1} + \Delta_{i_2})  - \cdots
 + (-1)^k N'(\Delta_1 + \cdots + \Delta_k).
\end{equation}
Recall that $N'(\Delta)$ denotes $(-1)^{\dim(\Delta)} N^\circ(\Delta)$ and $N^\circ(\Delta)$ is the number of integral points in the interior of $\Delta$.
\end{Th}
\begin{proof} 
Theorem follows immediately from Theorem \ref{th-genus-complete-int} and Corollary \ref{cor-26}.
\end{proof}

\begin{Rem}
For $k=n$ the righthand side of the \eqref{equ-17} is equal to $n!V(\Delta_1,\dots, \Delta_n)$ where $V$ denotes the mixed volume of convex bodies (\cite{Bernstein}). 
So the Bernstein-Kushnirenko theorem follows from 
Theorem \ref{th-29} (note that if $\dim(Z)=0$ then $\chi(Z)$ is equal to the number of points in $Z$).
The formula for $\chi (Z)$ above and deduction of the Bernstein-Kushnerenko theorem from this formula are from \cite{Askold-genus}.
\end{Rem}

Finally we have the following theorems about the $h^{p,0}$ numbers of the complete intersection $Z$:
\begin{Th} \label{th-30}
For any integer $i \geq 0$ the following holds: 
$$h^{i,0}(Z)\leq \sum_{\{J| \dim(\Delta_J)-|J| = i \textup{ and } J \neq \emptyset \}} N^\circ(\Delta_J)+ \delta_0^i,$$  
where  $\delta_0^i=0$  for $i\neq 0$ and $\delta_0^0=1$.
\end{Th}
\begin{proof} The theorem follows
from Theorem \ref{th-22}, Theorem \ref{th-24} and  the identity  $h^{i,0}(X)=\delta_0^i$  (see Corollary \ref{cor-27}).
\end{proof}

\begin{Def} \label{def-critical-number-toric} 
Let $\Delta_1, \ldots, \Delta_k$ be a $k$-tuple of independent integral polytopes.
We say that a number  $i \geq 0$ is {\it critical}
for $\Delta_1, \ldots, \Delta_k$ if there is a nonempty set $J\subset \{1,\ldots, k\}$ such that $N^\circ(\Delta_J)>0$ and $\dim(\Delta_J )-|J|=i$.
\end{Def}

\begin{Th} \label{th-31}
Let $\Delta_1, \ldots, \Delta_k$ be a $k$-tuple of independent integral polytopes.
\begin{itemize}
\item[(a)] If $0$ is a non-critical number for the collection of the $\Delta_i$ then the variety $Z$ defined by a generic system \eqref{equ-1-f_i} is irreducible.
\item[(b)] If $i > 0$ is a non-critical number for
$\Delta_1,\dots, \Delta_k$ then $h^{i,0}(Z)=0$.
\end{itemize}
\end{Th}
\begin{proof} For a non-critical $i$ the inequality in Theorem \ref{th-30} becomes $h^{i,0}(Z) \leq \delta_0^i$.  But the number $h^{0,0}(Z)$ is equal to the number of irreducible components of  
$Z$ and therefore it is strictly positive. The numbers $h^{i,0}(Z)$ are nonnegative.
\end{proof}

%Theorem \ref{th-17} on irreducibility  follows from Theorem \ref{th-31}. 
Theorem \ref{th-30} implies the following  improvement of a result in \cite{Askold-genus}:
\begin{Cor} \label{cor-32} 
If all the numbers $0\leq
i<n-k$ are not critical for $\Delta_1,\dots, \Delta_k$ then  $h^{0,0}(Z)=1$, $h^{n-k,0}(Z)=(-1)^{(n-k)}(\chi (Z)-1)$ and $h^{p,0}(Z)=0$ for $p \neq 0$ and $p \neq n-k$.
\end{Cor}
\begin{proof} 
From Theorem \ref{th-30} we have $h^{0,0}(Z)=1$ and $h^{p,0}(Z)=0$ for $0 < p < n-k$. The dimension of the smooth variety $Z$ is $n-k$ and therefore $h^{p,0}(Z)=0$ for $n-k < p$.
\end{proof}

%%%%%%%%%%%%%%%%%%%%%%%%%%%%%%%%%%%%%%%%%%%%%%%%%%%%%%%%%%%%%%

\section{Virtual polytopes} \label{sec-virtual}
%Here we recall in a form we need below  some facts from the theory of finitely additive measures on virtual polyhedra developed in \cite{[]}   
In this section we recall some basic facts from the theory of finitely additive measures on virtual polytopes developed in \cite{Kh-P}. We will need them later in 
Section \ref{sec-spherical}. This theory extends the theory of valuations on convex polyhtopes due to Peter McMullen (\cite{McMullen}) and uses the integration with respect to the 
Euler characteristic developed by Oleg Viro (\cite{Viro}).

\subsection{Ring $Z(\Lambda)$ of convex $\Lambda$-chains} \label{subsec-gp-convex-chains}
%For a fixed natural number $N$ denote by  $\frac{1}{N} \z^n\subset \r^n$ the lattice of rational points $a$ in $\r^n$ such that $Na\in \z^n$. Consider an additive subgroup  
%$\Lambda \subset \r^n$  where  either $\Lambda=\r^n$   or $\Lambda= \frac{1}{N} \z^n$.
%We denote by  $\mathcal{P}(\Lambda)$  the  set of  bounded  closed  convex  polyhedra  with  vertices  in $\Lambda$. We call the elements of  $\mathcal{P}(\Lambda)$  the 
%$\Lambda$-polyhedra.  If $\Lambda=\z^n$, then  we speak of  {\it integral polyhedra}. If $\Lambda=\r^n$, then  we speak of  {\it polyhedra} and use the notation 
%$\mathcal{P}(\r^n)=\mathcal{P}$.

For a fixed natural number $N$ let $\frac{1}{N} \z^n$ be the lattice:
$$\frac{1}{N} \z^n = \{  a/N \mid a \in \z^n\} \subset \r^n.$$
In the rest of this section $\Lambda \subset \r^n$ denotes an additive subgroup of $\r^n$ which is either equal to $\r^n$ itself or is equal to $\frac{1}{N} \z^n$ for some natural number $N$.
We denote by $\mathcal{P}(\Lambda)$ the collection of all convex polytopes in $\r^n$ with vertices in $\Lambda$. We call the elements of $\mathcal{P}(\Lambda)$ the 
\emph{$\Lambda$-polytopes}. When $\Lambda = \r^n$ (respectively $\Lambda = \z^n$) we refer to the $\Lambda$-polytopes simply as polytopes (respectively integral polytopes).

%A convex $\Lambda$-chain  is a function  $\alpha: \r^n\to \z$  of the form  $\alpha = \sum_{i=1}^k n_i \chi_ {\Delta_i}$, where $\chi_Y$ denotes the characteristic function  
%of a se $Y$,  $\Delta_i$ is a convex polytope in $\mathcal{P}(\Lambda)$ and  $n_i \in \z$.  We denote  the  additive  group  of convex  $\Lambda$-chains  by  $Z(\Lambda)$.
\begin{Def}
A \emph{convex $\Lambda$-chain} (or a \emph{$\Lambda$-chain} for short) is a function $\alpha: \r^n \to \z$ which can be represented as a finite sum $\alpha = \sum_i n_i \chi_{\Delta_i}$, 
where $\chi_Y$ denotes the characteristic function of a set $Y$,  the $\Delta_i$ are convex polytopes in $\mathcal{P}(\Lambda)$ and  $n_i \in \z$.  
We denote  the  additive  group  of convex  $\Lambda$-chains  by  $Z(\Lambda)$.
\end{Def}

Let $\Delta^\circ$ be the set of interior points of a polyhedron $\Delta$, in the topology of the affine space spanned by $\Delta$. It is easy to see that 
$$\chi_{\Delta^\circ}= \sum_{\Delta_i \in \Gamma (\Delta)} (-1)^{\dim \Delta_i}\chi_{\Delta_i},$$ 
where $\Gamma(\Delta)$ is the set of all faces of $\Delta$, including $\Delta$ itself. Therefore if 
$\Delta$ is in $\mathcal{P}(\Lambda) $ then $\chi_{\Delta^\circ}$ is also in $Z(\Lambda)$.

%\subsection {Ring $Z(\Lambda)$ of convex $\Lambda$-chains} \label{subsec-ring-convex-chains}
It is shown in \cite{Kh-P} that the Minkowski sum of convex polytopes extends in a unique way to an opration $*$: 
$$*: Z(\Lambda) \times  Z(\Lambda) \to  Z(\Lambda),$$
on the additive group of $\Lambda$-convex chains. That is, for two convex polytopes $\Delta_1, \Delta_2 \in \mathcal{P}(\Lambda)$ we have: 
$$\chi_{\Delta_1} * \chi_{\Delta_2} = \chi_{(\Delta_1+\Delta_2)}.$$
We call the operation $*$ the \emph{(Minkowski) multiplication of chains}. 
The group $Z(\Lambda)$  together with this multiplication is a commutative ring called the \emph{ring of $\Lambda$-convex chains}.

The characteristic function of the origin $\chi_{\{0\}}$ is the unit in the ring $Z(\Lambda)$. We call any invertible element of the ring of $\Lambda$-chains a {\it $\Lambda$-virtual polytope}.
When $\Lambda = \r^n$ (respectively $\Lambda = \z^n$) we simply call an invertible element a {\it virtual polytope} (respectively an {\it integral virtual polytope}).
%for $\Lambda=\r^n$ and by {integral virtual polyhedron} for $\Lambda=\z^n$) any  invertible element in $Z(\Lambda)$. 
It turns out that the characteristic function $\chi_{\Delta}$ of a polytope $\Delta\in \mathcal{P}(\Lambda)$ is invertible in the ring $Z(\Lambda)$:
\begin{Th} [$\Lambda$-virtual polytopes] \label{th-lambda-virtual}
The following statements hold:
\begin{itemize}
\item[(1)] For $\Delta\in \mathcal{P}(\Lambda)$ the multiplicative inverse $\chi^{-1}_{\Delta}$ of $\chi_{\Delta} \in Z(\Lambda)$ is  $$(-1)^{\dim \Delta} \chi _{(-\Delta^\circ)}.$$ 
Here $-\Delta^\circ$ is the set of interior points of $-\Delta$.
\item[(2)] Each $\Lambda$-virtual polytope can be written in the form $\chi_{\Delta_1} * \chi_{\Delta_2}^{-1}$, where $\Delta_1, \Delta_2 \in \mathcal{P}(\Lambda)$.
\end{itemize}
\end{Th}

\subsection{Integration over the Euler characteristic and multiplication in the ring $Z(\Lambda)$} \label{subsec-int-Euler-char}
%There is a  transparent invariant  definition of the multiplication in the ring of convex chains. It  uses the {\it integration over Euler characteristic}  (see, for example \cite{}). 
%Let us say that a set $Y\subset V$ in a real finite dimensional vector space $V$ is {\it semi-convex} if $Y$ can be represent as a disjoin union $$Y=\cup \Delta^\circ_i$$ 
%of finitely many interiors $\Delta_i^0$ of polyhedra  $\Delta_i\in \mathcal{P}$. By definition the {\it Euler characteristic} $\mu(Y)$ of a semi-convex set $Y$ defined by 
%(1) is equal to $\sum (-1)^{\dim \Delta_i}$. It is known (see [..], [..]) that $\mu(Y)$ is well defined (i.e. is independent of a choice of the representation (1)). 
%Euler characteristic is a finitely additive measure on semi-convex sets. On  closed semi-convex set this measure coincides with the usual Euler characteristic, 
%but for a general semi-convex set it has no topological meaning. One can defined an integral over Euler characteristic of any convex chain $f\in Z(\r^n)$ by 
%the formula $$ \int f=\sum_{n\in \z} n \chi (f^{-1}(n)).$$ (Note that $f$ takes only finitely many values, and each level set  $f^{-1}(n)$ is a semi-convex set.)

There is a way to define the multiplication $*$ in the ring of convex chains $Z(\Lambda)$ which does not require representation of convex chains as linear combinations 
of characteristic functions of polytopes. This definition instead uses \emph{integration with respect to the Euler characteristic} (see \cite{Viro}). Below we explain this notion in more detail.

Let us say that a subset $Y$ in a finite dimensional real vector space $V$ is \emph{semi-convex} if $Y$ can be represented as a disjoint union
\begin{equation} \label{equ-semi-convex} 
Y = \bigcup_{i} \Delta^\circ_i,
\end{equation}
of relative interiors of finitely many convex polytopes $\Delta_i \in \mathcal{P}(\Lambda)$. 
By definition the {\it Euler characteristic} $\mu(Y)$ of a semi-convex set $Y$ is $\sum_i (-1)^{\dim \Delta_i}$. It is known (see \cite{Viro}) that $\mu(Y)$ is well-defined, i.e. is 
independent of a choice of the representation \eqref{equ-semi-convex}. The Euler characteristic is a finitely additive measure on the collection of semi-convex sets. For closed semi-convex 
sets this measure coincides with the usual Euler characteristic in topology, but for a general semi-convex set it is not the topological Euler characteristic. 
One defines the integral $\int fd\mu$ of a convex chain $f \in Z(\r^n)$ with respect to the Euler characteristic by: $$\int f d\mu = \sum_{a \in \z} a \mu(f^{-1}(a)).$$ 
(Note that $f$ takes only finitely many values and that each level set  $f^{-1}(a)$ is a semi-convex set.)

\begin{Th} [Multiplication in $Z(\Lambda)$ using the Euler characteristic] 
The multiplication $\alpha\ast\beta$ of $\alpha,\beta\in Z(\Lambda)$ is equal to the convolution with respect to the integration over Euler characteristic of the functions $\alpha,\beta$:
$$(\alpha * \beta)(x)=\int \alpha(z)\beta (x-z)d\mu(z).$$
\end{Th}

\subsection{Finitely additive measures on convex $\Lambda$-chains} \label{subsec-finitely-add-measure}
A {\it (real valued  finitely additive) measure} on $Z(\Lambda)$ is by definition an additive
homomorphism $\phi: Z(\Lambda) \to \r$. A measure is called an {\it invariant measure} (respectively a {\it polynomial measure of degree
$\leq d$}) if for each $\alpha \in Z(\Lambda)$ the function $h \mapsto \phi(\alpha_h)$ on $h \in \Lambda$ is constant (respectively is the restriction of a 
polynomial of degree $\leq d$ to $\Lambda \subset \r^n$). Here $\alpha_h \in Z(\Lambda)$ is the convex chain defined by 
$\alpha_h(x)=\alpha(x-h)$ (i.e. $\alpha_h$ is the shift of $\alpha$ by $h$).

One can check the following easy facts:
\begin{Lem} \label{lem-virtual-1}
Let $\phi: Z(\r^n)\to \r$ be a measure. Then:
\begin{itemize}
\item[(1)] For  $\beta\in Z(\r^n)$ the function $\phi_{\beta}:Z(\r^n)\to \r$  defined by
$\phi_{\beta}(\alpha)=\phi(\alpha \ast \beta)$ is also a measure on $Z(\r^n)$.
\item[(2)] The restriction of $\phi$ to $Z(\Lambda)\subset Z(\r^n) $ is a measure on $Z(\Lambda)$.
\end{itemize}
\end{Lem}

\begin{Lem} \label{lem-virtual-2}
Let $f: \r^n \to \r$ be a polynomial of degree $\leq d$. Then the function  $\phi=\int f$ on $Z(\r^n)$ defined by
$\phi(\alpha)= \int f(x)\alpha(x)dx,$ where  $dx$ is the standard Lebesgue measure on $\r^n$, is a polynomial measure of degree $\leq d$ on $Z(\r^n) $.
\end{Lem}
In particular, Lemma \ref{lem-virtual-2} states that the map $\Delta \mapsto \int_{\Delta}f(x)dx$ on the space of polytopes $\mathcal{P}(\r^n)$ uniquely extends to a measure 
on virtual polytopes. This measure is the usual volume of a polytope if $f$ is the constant polynomial $1$.

\begin{Lem} \label{lem-virtual-3}
Let $f:\r^n\to \r$ be a polynomial of degree $\leq d$ and let $\beta \in Z(\r^n)$ be a convex chain. Then the function $\phi$
on integral convex chains $Z(\z^n)$ defined by: $$\phi(\alpha) = \sum_{x\in \z^n} f(x) (\beta * \alpha) (x),$$ is a polynomial measure of degree 
$\leq d$ on integral convex chains.
\end{Lem}
In particular Lemma \ref{lem-virtual-3} implies the following: let $\Delta_\beta$ be a fixed (not necessarily integral) convex polytope. Consider 
the map $\phi : \Delta \mapsto \sum_{x \in (\Delta + \Delta_\beta) \cap \z^n} f(x)$ on convex polytopes. Then this map extends uniquely to a measure on virtual 
convex polytopes. If $f$ is the constant polynomial $1$ then $\phi(\Delta)$ is the number of integral points in $\Delta + \Delta_\beta$.
 
%that provides the extension on integral virtual polyhedra of the measure $\mu$ which is defined on an integral polyhedra  
%by the following formula: $\mu(\Delta)=\sum f(x)$ where the sum is taken over all integral point $x$ in $\Delta +\Delta_{\beta}$ where $\Delta_{\beta}$ is a fixed 
%(not necessarily integral) convex polyhedron (if $P\equiv 1$ then $\mu(\Delta)$ equals to the number of integral points in the polyhedron $\Delta +\Delta_{\beta}$).

%Now let us state the main result about the polynomial measures on virtual polyhedra that we would need later.
%Let us fix a $k$-tuple of $\Lambda$-virtual polyhedra $\alpha_1,\dots, \alpha_k\in Z(\Lambda)$ and a polynomial of degree $\leq d$  measure $\phi$ on $Z(\Lambda)$. 
%Let $\Phi$ be the function on $\z^k\subset \r^k$ defined by the formula $$\Phi (\bold p)= \phi (\alpha_1^{p_1}\ast\dots\ast\alpha^{p_k}),$$
%where $\bold p=(p_1,\dots,p_k)\in \z^k$ and $\alpha_i^{p_i}$ means $p_i$-th degree of the element $\alpha_i$ in the multiplicative group of invertible elements of the ring $Z(\Lambda)$.

We now state the main statement about polynomial measures on virtual polytopes.
Fix a $k$-tuple of $\Lambda$-virtual polytopes $\alpha_1, \ldots, \alpha_k \in Z(\Lambda)$ as well as a polynomial measure $\phi$ of degree $\leq d$ on $Z(\Lambda)$.
Let $\Phi$ be the function on $\z^k$ defined by:
$$\Phi(m_1, \ldots, m_k) = \phi(\alpha_1^{m_1} * \cdots * \alpha^{m_k}),$$
where $(m_1, \ldots, m_k) \in \z^k$ and $\alpha_i^{m_i}$ denotes the $m_i$-th power of $\alpha_i$ in the ring $Z(\Lambda)$.

\begin{Th} \label{th-virtual-1}
With the above notation, the function $\Phi$ is the restriction of a polynomial of degree $\leq n+d$ on $\r^k$ to $\z^k$.
\end{Th}

The next corollary follows from Theorem \ref{th-virtual-1} applied to the measure in Lemma \ref{lem-virtual-3}.
\begin{Cor} \label{cor-virtual-1} 
Let $f: \r^n \to \r$ be a degree $d$ polynomial, $\beta$ a  convex chain and $\alpha_1, \ldots, \alpha_k$ a $k$-tuple of integral virtual polytopes. Then the function $\Phi$ 
on the lattice $\z^k$ defined by: $$\Phi(m_1, \ldots, m_k)=\sum_ {x \in \z^n} f(x)(\beta * \alpha_1^{m_1} * \cdots * \alpha_k^{m_k})(x),$$ where $(m_1,\dots, m_k) \in \z^k$, 
is the restriction of a polynomial on $\r^k$ of degree $\leq n+d$ to $\z^k$.
\end{Cor}

Finally we have the following variation of Corollary \ref{cor-virtual-1}. We will use it later in Section \ref{subsec-Euler-char-sph} and Section \ref{subsec-genus-sph}.
Below $\Lambda = \frac{1}{N} \z^n \subset \r^n$ for some ineteger $N > 0$. 
%Finally we have the following variations of Corollary \ref{cor-virtual-1}. We will use them later in Section \ref{subsec-Euler-char-sph} and Section \ref{subsec-genus-sph}.
%Below $\Lambda = \frac{1}{N} \z^n \subset \r^n$ for some ineteger $N > 0$. 
%Also $G = N \z^k \subset \z^k$ (consisting of integral points with coordinates divisible by $N$).

%\begin{Cor} \label{cor-virtual-2}
%Let $f: \r^n \to \r$ be a polynomial of degree $d$ and let $\gamma_1, \ldots, \gamma_k$ be a $k$-tuple of $\Lambda$-virtual polyhedra. 
%Take ${\bf m}_0 \in \z^k$ and consider the coset ${\bf m}_0 + G$.
%Define the function $\Phi$ on the coset ${\bf m}_0 + G$ by: 
%\begin{equation} \label{equ-virtual}
%\Phi({\bf m})=\sum_ {x \in \z^n} f(x)(\gamma_1^{m_1} * \cdots * \gamma_k^{m_k})(x),
%\end{equation}
%where ${\bf m}=(m_1,\ldots, m_k) \in {\bf m}_0 + G$. 
%Then $\Phi$ is the restriction of a polynomial on $\r^k$ of degree $\leq n+d$ to the coset ${\bf m}_0 + G \subset \r^k$.
%\end{Cor}
%\begin{proof}
%Let $\alpha_1,\ldots, \alpha_k$ denote the chains $\gamma_1^N,\ldots, \gamma_k^N$ respectively. 
%Since $\gamma_1,\ldots,\gamma_k$ are convex $(\frac{1}{N}\z^n)$-chains, $\alpha_1,\dots,\alpha_k$ are integral chains. 
%Let ${\bf m}_0=(m_1^0, \ldots,m_k^0)$ and let $\beta$ denote the $(\frac{1}{N}\z^n)$-chain $\gamma_1^{m_1^0} * \cdots * \gamma_k^{m_k^0}$. The claim 
%follows from Corollary \ref{cor-virtual-1} applied to $\alpha_1,\ldots,\alpha_k$ and $\beta$. 
%\end{proof}

\begin{Cor} \label{cor-virtual-3}
Let $f:\r^n \to \r$ be a polynomial of degree $d$ and let $\gamma_1, \ldots, \gamma_k \in \mathcal{P}(\Lambda)$ be a $k$-tuple of $\Lambda$-virtual polytopes. 
Let $a_1, \ldots, a_k$ be fixed points in the lattice $\Lambda$ and consider the function $\Psi: \z^k \to \r$ defined by:
$$\Psi(m_1, \ldots, m_k)=\sum_ {x \in (m_1a_1+ \cdots + m_ka_k) + \z^n} f(x)(\gamma_1^{m_1} * \cdots * \gamma_k^{m_k})(x).$$ 
Then $\Psi$ is the restriction of a polynomial on $\r^k$ to $\z^k$.
%of degree $\leq n+d$ to $\z^k$.
%Let $f:\r^n \to \r$ be a polynomial of degree $d$ and let $\gamma_1, \ldots, \gamma_k \in \mathcal{P}(\Lambda)$ be a $k$-tuple of $\Lambda$-virtual polytopes. 
%Let $a_1, \ldots, a_k$ be fixed points in $\r^n$ and consider the function $\Psi: \z^k \to \r$ defined by:
%$$\Psi(m_1, \ldots, m_k)=\sum_ {x \in \z^n} f(x + m_1a_1+ \cdots + m_ka_k)(\gamma_1^{m_1} * \cdots * \gamma_k^{m_k})(x),$$ 
%Then $\Psi$ is the restriction of a 
%polynomial on $\r^k$ of degree $\leq n+d$ to $\z^k$.
\end{Cor}
\begin{proof} 
%For each $i=1, \ldots, k$ let $\alpha_i = \gamma_i * \delta_{-a_i}$ where $\delta_{-a_i}$ is the characteristic function of the point $-a_i$. Clearly 
%$\alpha_i(x) = \gamma_i(x - a_i)$. The claim now follows from Corollary \ref{cor-virtual-1} applied to $\alpha_1, \ldots, \alpha_k$ and $\beta = 0$.
We reduce the claim to Corollary \ref{cor-virtual-1}. Consider the space $\r^{n+k}=\r^n \times \r^k$, the lattice $\Lambda_1=\frac {1}{N} \z^{n+k}\subset \r^{n+k}$ 
and the $\Lambda_1$-chains $\rho_1, \ldots, \rho_k$ defined by $\rho_i=\gamma_i * \delta_{(-a_i, e_i)}$. Here $e_i$ is the $i$-th standard basis vector in $\r^k$ and 
$\delta_{(a,b)}$ denotes the characteristic function of a point $(a, b) \in \r^n \times \r^k$. Then $\rho_1^{m_1} * \cdots * \rho_k^{m_k}=\gamma_1^{m_1} 
* \cdots * \gamma_k^{m_k} * \delta_{(-m_1a_1 - \cdots - m_ka_k, m_1e_1+ \cdots + m_ke_k)}$. 
Also consider the polynomial $h: \r^n \times \r^k \to \r$ defined by: $h(x, m_1, \ldots, m_k)= f(x+a_1m_1+\cdots+a_km_k)$. The claim now follows from 
Corollary \ref{cor-virtual-1} applied to the vector space $\r^{n+k}$, the lattice $\Lambda_1=\frac{1}{N} \z^{n+k} \subset \r^{n+k}$, $\Lambda_1$-chains 
$\rho_1, \ldots, \rho_k$ and the polynomial $h$ (we take $\beta$ in Corollary \ref{cor-virtual-1} to be $0$).
\end{proof}

\section{Complete intersections in spherical varieties} \label{sec-spherical}
In the rest of the paper we will use the following notation about reductive groups:
\begin{itemize}
\item[-] $G$ denotes a connected complex reductive algebraic group.
\item[-] $B$ a Borel subgroup of $G$ and $T$, $U$ the maximal torus
and maximal unipotent subgroups contained in $B$ respectively. 
%\item[-] $B^{-}$ and $U^{-}$ are the opposite subgroups of $B$ and $U$ respectively.
%The Lie algebra of $U$ and $T$ will be denoted respectively
%by $\mathfrak{u}$ and $\t$.
%\item[-] $\Phi = \Phi(X, T)$ denotes the root system with $\Phi^+ = \Phi^+(X, T)$ the subset of positive
%roots for the choice of $B$.  
%\item[-] $\alpha_1, \ldots, \alpha_r$ denote the simple roots where $r$ is the
%semi-simple rank of $G$. 
%\item[-] $W$ is the Weyl group of $(G,T)$. The simple reflection associated with
%a simple root $\alpha$ is denoted by $s_\alpha$. 
%\item[-] $w_0$ is the unique longest element in $W$. $N$ denotes the length of $w_0$ which is equal to the number of
%positive roots as well as the dimension of the flag variety $G/B$.
%\item[-] $E_\alpha$, $F_\alpha$ are the Chevalley generators for a root $\alpha$, which are the generators for
%the root subspaces $\Lie(G)_\alpha$ and $\Lie(G)_{-\alpha}$ respectively. 
%\item[-] $U_\alpha = \{ \exp(tE_\alpha) \mid t \in \c \}$,  $U^-_{\alpha} = \{ \exp(tF_\alpha) \mid t \in \c \}$ 
%denote the root subgroups corresponding to a root $\alpha$.
\item[-] {$\Lambda$ is the weight lattice of $G$, 
%We denote the rank of $\Lambda$, equal to $\dim(T)$, by $n$. 
$\Lambda^+$ is the subset
of dominant weights and $\Lambda_\r = \Lambda \otimes_{\z} \r$. The cone generated by $\Lambda^+$ is the positive Weyl chamber denoted by $\Lambda^+_\r$.}
\item[-] $V_\lambda$ denotes the irreducible $G$-module corresponding to a dominant weight $\lambda$. Also
$v_\lambda$ denotes a highest weight vector in $V_\lambda$. 
%\item[-] For a dominant weight $\lambda$, $-w_0\lambda$ is denoted by $\lambda^*$.  It is dominant and
%$V_\lambda^* \cong V_{\lambda^*}$.
%\item[-] $o = eB$ is the unique $B$-fixed point
%in the flag variety $G/B$. 
%\item[-] $C_w$, $X_w$ denote the
%Schubert cell and the Schubert variety in $G/B$ corresponding to $w \in W$ respectively.
\item[-] $G/H$ denotes a spherical homogeneous space.
\item[-] $\Lambda' = \Lambda(G/H) \subset \Lambda$ denotes the weight lattice of $G/H$, i.e. the sublattice of $\Lambda$ consisting of 
weights of $B$-eigenfucntions in $\c(G/H)$.
\end{itemize}

\subsection{Preliminaries on spherical varieties} \label{subsec-prelim-sph}
A $G$-variety $X$ is called spherical if a Borel subgroup (and hence any Borel subgroup) has a dense orbit. 
If $X$ is spherical it has a finite number of $G$-orbits as well as a finite number of 
$B$-orbits. Spherical varieties are a generalization of toric varieties for actions of reductive groups. Analogous to toric varieties, the geometry of spherical varieties 
can be read off from associated convex polytopes and convex cones. For a nice overview of the theory of spherical varieties we refer the reader to \cite{Perrin}.

It is a well-known fact that if $\L$ is a $G$-linearized line bundle on a spherical variety then the space of sections $H^0(X, \L)$ is a multiplicity free $G$-module.
For a quasi-projective $G$-variety $X$ this is actually equivalent to $X$ being spherical.

Below are some important examples of spherical varieites and spherical homogeneous spaces:
\begin{itemize}
\item[(1)] When $G$ is a torus, the spherical $G$-varieties are exactly toric varieties. 
\item[(2)] The flag variety $G/B$ and the parital flag varieties $G/P$ are spherical $G$-varieties by the Bruhat decomposition.
\item[(3)] Let $G \times G$ act on $G$ from left and right. Then the stabilizer of the identity is $G_{diag} = \{(g, g) \mid g \in G\}$. Thus $G$ can be identified with the homogeneous space 
$(G \times G) / G_{diag}$. Again by the Bruhat decomposition this is a spherical $(G \times G)$-homogeneous space.
\item[(4)]
Consider the set $\mathcal{Q}$ of all smooth quadrics in $\p^n$. The group $G = \textup{PGL}(n+1,\c)$ acts transitively on $\mathcal{Q}$. 
The stabilizer of the quadric $x_0^2 + \cdots + x_n^2 = 0$ (in the homogeneous coordinates) is $H = \textup{PO}(n+1, \c)$ and hence $\mathcal{Q}$ can be identified with the 
homogeneous space $\textup{PGL}(n+1, \c) / \textup{PO}(n+1, \c)$ . The subgroup $\textup{PO}(n+1,\c)$ is the fixed point set of the involution $g \mapsto (g^t)^{-1}$ of $G$ and hence 
$\mathcal{Q}$ is a symmetric homogeneous space. In particular, $\mathcal{Q}$ is spherical. 
Let $V$ be the vector space of all quadratic forms in $n+1$ variables and $V^*$ its dual. The map which assigns to a quadric $C$ its homogeneous equation 
(respectively equation of the dual quadric $C^*$ ) gives an embedding of $\mathcal{Q}$ in $\p(V)$ (respectively $\p(V^*)$). Let $X$ be the closure of the set of all quadrics 
$(C, C^*)$ in $\p(V) \times \p(V^*)$. It is called the \emph{variety of complete quadrics}. It is well-known that $X$ is a smooth variety (see \cite[Theorem 3.1]{DP}). 
This variety plays an important role in classical enumerative geometry.
\end{itemize}

Throughout the rest of the paper we will fix a spherical homogeneous space $G/H$. 

\begin{Def}
Let $\Lambda(G/H)$ be the lattice of $B$-weights for the action of $B$ on the field of rational functions $\c(G/H)$, i.e. the set of all 
$\lambda \in \Lambda$ which appear as the weight of a 
$B$-eigenfunction in $\c(G/H)$. Clearly $\Lambda(G/H)$ is a sublattice of $\Lambda$. We will denote the lattice of $B$-weights of $G/H$ simply by $\Lambda'$. 
\end{Def}

\begin{Rem} \label{rem-Lambda'}
Let $\c(G/H)^{(B)}$ denote the multiplicative group of nonzero $B$-eigenfunctions in $\c(G/H)$. If two $B$-eigenfunctions $f$ and $g$ have the same weight then 
$f/g$ is a $B$-invariant rational function on $G/H$. Since $X$ has an open $B$-orbit we conclude that $f/g$ is constant. This proves that the map which sends a $B$-eigenfunction to 
its weight gives an isomorphism between $\c(G/H)^{(B)} / \c^*$ and the lattice $\Lambda'$.
\end{Rem}

The following theorem about $h^{p,0}$ numbers of spherical varieties will be used later. It is a generalization of Corollary \ref{cor-27} for toric varieties. 
\begin{Th}[$h^{p,0}$ numbers of spherical varieties] \label{th-h-p-sph}
Let $X$ be a smooth projective spherical variety. Then $h^{0,0}(X) = 1$ and $h^{p,0}(X) = 0$ if $p > 0$. 
\end{Th}
\begin{proof}
One shows that if $X$ is a spherical variety then the maximal torus $T$ has isolated fixed points.
The theorem then follows from the more general theorem that if $X$ is a smooth projective variety with an action of a torus with isolated fixed points then $h^{p,q}(X) = 0$ if 
$p \neq q$. We refer to \cite{CL} for this vanishing result as well as its generalization to holomorphic vector fields on K\"ahler manifolds.
\end{proof}

\subsection{Cohomology of line bundles} \label{subsec-coh-line}
Let $X$ be an $n$-dimensional projective spherical variety with a globally generated $G$-linearized line bundle $\L$.
The following result of Michel Brion determines when the cohomology groups of the line bundles $\L$ and $\L^{-1}$ vanish (\cite{Brion-vanishing}).
It generalizes similar statements for toric varieties as well as the Borel-Weil-Bott theorem for flag varieites.
Below $\kappa = \kappa(\L)$ denotes the \emph{Itaka dimension} of $\L$, i.e. the dimension of the image of the Kodaira map of $\L^{\otimes m}$ for sufficiently large $m$.
\begin{Th}[Brion] \label{th-coh-line-bundle}
Let $X$ be a projective spherical variety with a globally generated $G$-linearized line bundle $\L$. Then:
%the polytope $\tilde{\Delta}(X, \L)$ is the Newton-Okounkov polytope and
%$\Delta^\circ$ denotes the interior of $\Delta$.
\begin{itemize}
\item[(a)] For any $i>0$, $H^i(X, \L) = \{0\}$.
%$$\dim(H^i(X, \L)) = \begin{cases} 
%0, \quad i \neq 0 \\
%S(\Delta(X, \L)) = \#(\tilde{\Delta}(X, \L) \cap \z^n), \quad i = 0\\
%\end{cases}
%$$
\item[(b)] For all $i \neq \kappa$, $H^i(X, \L^{-1}) = \{0\}$.
%$$ \dim(H^i(X, \L^{-1})) = \begin{cases} 
% 0, \quad i \neq  \kappa(\L)\\
%S^\circ(\Delta(X, \L)) = \#(\Delta^\circ_{\w}(X, \L) \cap \z^n), \quad i=\kappa(\L) \\
%\end{cases}$$
\end{itemize}
\end{Th}

As before we let $\chi(X, \L)$ denote the Euler characteristic of the bundle $\L$ defined by $\chi(X, \L) = \sum_{i=0}^n (-1)^i \dim(H^i(X, \L))$.
\begin{Cor} \label{cor-dim-H^0-polynomial}
Let $X$ be a smooth projective spherical variety and $\L_1, \ldots, \L_k$ globally generated $G$-linearized line bundles on $X$.
For integers $m_1, \ldots, m_k$ put:
$$\psi(m_1, \ldots, m_k) = \chi(X, \L_1^{\otimes m_1} \otimes \cdots \otimes \L_k^{\otimes m_k}).$$
Then $\psi$ is a polynomial in the $m_i$. Moreover when the $m_i$ are nonnegative we have:
$$\psi(m_1, \ldots, m_k) = \dim(H^0(X, \L_1^{\otimes m_1} \otimes \cdots \otimes \L_k^{\otimes m_k})).$$ 
\end{Cor}
\begin{proof}
The corollary follows from the Hirzbruch-Riemann-Roch theorem and Theorem \ref{th-coh-line-bundle}(a). 
\end{proof}

In particular, let $\L$ be a globally generated $G$-linearized line bundle on a projective spherical variety $X$. 
Let $\psi$ be the polynomial in $m \in \z$ such that $\psi(m) = \dim(H^0(X, \L^{\otimes m}))$ for any nonnegative integer $m$. 

\begin{Cor} \label{cor-dim-H^k-L^-1}
As above let $\kappa = \kappa(\L)$ be the Itaka dimension of $\L$.
We can compute the dimension of $H^\kappa(X, \L^{-1})$ by:
$$\dim(H^\kappa(X, \L^{-1})) = (-1)^\kappa \psi(-1).$$
(Note that by Theorem \ref{th-coh-line-bundle}(b) all other cohomology groups of $\L^{-1}$ are $0$.)
\end{Cor}
\begin{proof}
The corollary follows from the Hirzbruch-Riemann-Roch theorem and Theorem \ref{th-coh-line-bundle}(b).
\end{proof}

\subsection{Moment polytope} \label{subsec-moment}
In this section we discuss the notion of the moment polytope $\Delta(R)$ of a graded $G$-algebra $R = \bigoplus_m R_m$. When $R$ is the homogeneous coordinate ring of 
a smooth projective $G$-variety (equivariantly embedded in a projective space) the moment polytope can be identified with the moment polytope (or Kirwan polytope) of $X$ in 
symplectic geometry (see Remark \ref{rem-moment-polytope-symplectic} below). We will be interested in the moment polytopes 
of graded $G$-linear systems $R$ over a spherical variety $X$. When $X$ is normal, the integral points in the dilated polytope $m\Delta(R)$ correspond to the irreducible $G$-modules 
$V_\lambda$ appearing in $R_m$ (Theorem \ref{th-H^0-moment-poly}). We will express our main formula for the genus of a complete intersection in terms of integral points in 
moment polytopes (Theorem \ref{th-main}).

\subsubsection{Moment polytope of a graded $G$-algebra}
Let $R = \bigoplus_m R_m$ be a graded $G$-algebra over $\c$ where the $R_m$ are finite dimensional. 
Moreover, assume that $R$ is contained in a finitely generated graded algebra. 
The rings $R$ we will be interested in the upcoming sections are $G$-invariant graded linear systems on a spherical variety or a homogeneous space. That is, graded subalgebras of
rings of sections of $G$-line bundles on a spherical variety.

Consider the additive semigroup $S(R) \subset \n \times \Lambda$ defined by:
$$S(R) = \bigcup_{m} \{ (m, \lambda) \mid V_\lambda \textup{ appears in } R_m \}.$$ 
Let $C(R)$ denote the closure of the convex hull of 
$S(R)$ in the vector space $\r \times \Lambda_\r$. It is a closed convex cone with apex at the origin.
One defines the \emph{moment convex body} $\Delta(A)$ to be the slice of the cone $C(R)$ at $m=1$ (\cite{Brion-moment} and \cite{KKh-reductive}). That is:
$$\Delta(R) = C(R) \cap (\{1\} \times \Lambda_\r).$$

Alternatively, after projection $\r \times \Lambda_\r \to \Lambda_\r$ on the second factor, the polytope $\Delta(R)$ can be defined as:
$$\Delta(R) = \overline{\bigcup_{m > 0} \{ \lambda / m \mid V_\lambda \textup{ appears in } R_m \}}.$$

\begin{Rem}
If the algebra $R$ is finitely generated then one shows that the semigroup $S(R)$ is a finitely generated semigroup and hence $\Delta(R)$ is a rational convex polytope.
In this case we will refer to $\Delta(R)$ as the \emph{moment polytope}. The moment polytope is also called the \emph{Brion polytope}.
It was first introduced in the paper \cite{Brion-moment}.
%All the graded algebras we will deal with in this paper will be finitely generated.
\end{Rem}

\begin{Rem}[Connection with moment polytope in symplectic geometry] \label{rem-moment-polytope-symplectic}
Let $K$ be a compact Lie group and let $X$ be a compact Hamiltonian $K$-manifold with the moment map
$\phi: X \to \Lie(K)^*$. It is a well-known result due to F. Kirwan
that the intersection of the image of the moment map with the positive Weyl chamber is a
convex polytope usually called the {\it moment polytope} or {\it Kirwan polytope} of the
Hamiltonian $K$-space $X$.

As usual let $G$ be a complex connected reductive group. Let $G$ act linearly on a finite dimensional complex vector space $V$.
Let $X$ be a closed irreducible $G$-stable subvariety of the projective space $\p(V)$. Let $R$ denote the homogenous coordinate ring of $X$.
Fix a $K$-invariant inner product on $V$ where $K$ is a maximal compact subgroup of $G$. This induces a $K$-invariant symplectic structure on $\p(V)$ and hence on 
the smooth locus  of $X$.
\begin{enumerate}
\item[(1)] With this symplectic structure, the smooth locus of $X$ is a Hamiltonian $K$-manifold.
\item[(2)] When $X$ is smooth it can be proved that the Kirwan polytope of $X$ coincides with $\Delta(R)$ (see \cite{Ness},
\cite{G-S} and \cite{Brion-moment}). More precisely, the Kirwan polytope identifies with $\Delta(R)$ 
after taking the involution $\lambda \mapsto \lambda^*$, where $\lambda^* = -w_0 \lambda$.
\item[(3)] The above is still true if $X$ is non-smooth. In this case one considers the moment map of $\p(V)$ (as a Hamiltonian $K$-space) and restricts it to $X$.
Then the intersection of the image of $X$ (under the restricted moment map) with the positive Weyl chamber can be identified with $\Delta(R)$. 
\end{enumerate}
\end{Rem}

\subsubsection{Moment polytope of a line bundle on a spherical variety} \label{subsubsec-moment-poly-line-bundle}
In this section we describe the linear inequalities defining the moment polytope of the ring of sections of a $G$-linearized line bundle 
on a normal projective spherical variety. We will use it to describe the space of sections of the line bundle, as a $G$-module, in terms of the lattice points in its moment polytope.

Let $\L$ be a $G$-linearized line bundle on a normal projective spherical $G$-variety $X$. Then $H^0(X, \L)$ is a finite dimensional $G$-module. 
Let us assume that $H^0(G, \L) \neq \{0\}$. 
By the \emph{ring of sections of $\L$} we mean the algebra $R(X, \L)$ defined by:
$$R(X, \L) = \bigoplus_{m} H^0(X, \L^{\otimes m}).$$ 
We denote the moment convex body of this algebra by $\Delta(X, \L)$.

%\color{blue}
%Define the lattice of highest weights $\Lambda(X, \L)$ in $\z \times \Lambda$.
%\color{black}

Since $H^0(X, \L)$ is a finite dimensional $G$-module there is a $B$-eigensection $\sigma$ in $H^0(X, \L)$.
Then the divisor $D$ of the section $\sigma$ is a $B$-stable divisor. 
Let $D_1, \ldots, D_s$ be all the $B$-stable prime divisors in $X$.
Thus we can write:
%Now let $\D(L)$ be the divisor of the subspace $L$ as in Section \ref{sec-divisor}. We can show that $\D(L)$ is a $B$-stable divisor.
%Let us write $\D(L)$ as a sum of $B$-stable prime divisors:
$$D = \sum_i a_i D_i.$$
Then for a rational function $f \in \c(X)$ the corresponding mermorphic section $f\sigma$ belongs to $H^0(X, \L)$ if and only if it satisfies:
%Then: $$\overline{L} = \{ f \in \c(X) \mid \ord_{D_i}(f) \geq -a_i, 1 \leq i \leq s\}.$$
$$\ord_{D_i}(f) \geq -a_i, \quad \forall i=1, \ldots, s,$$
where $\ord_{D_i}$ is the order of zero-pole along the prime divisor $D_i$
(notice that here we are using the assumption that $X$ is normal).
Via restriction, the function $\ord_{D_i}$ now defines a linear function $\ell_{D_i}$ on the lattice $\Lambda' = \Lambda(G/H) \cong \c(G/B)^{(B)} / \c^*$ (see Remark \ref{rem-Lambda'}).
Let $\alpha$ denote the weight of the $B$-eigensection $\sigma \in H^0(X, \L)$ that we fixed. Then the $G$-spectrum of $H^0(X, \L)$ can be described as:
\begin{multline}
\Spec_G(H^0(X, \L)) = \alpha + \{ \gamma \in \Lambda' \mid \ell_{D_i}(\gamma) \geq -a_i, 1 \leq i \leq s\}, \\
= \{ \lambda \in \alpha + \Lambda' \mid \ell_{D_i}(\lambda) \geq -a_i + \ell_{D_i}(\alpha), 1 \leq i \leq s \}.
\end{multline} 
Applying this to all the $H^0(X, \L^{\otimes m})$ for $m>0$, we get the following description of the moment polytope of the ring of sections $R(X, \L)$:
\begin{equation} \label{equ-moment-polytope}
\Delta(X, \L) = \{ x \in \alpha + \Lambda'_\r \mid \ell_{D_i}(x) \geq -a_i + \ell_{D_i}(\alpha), 1 \leq i \leq s\}.
\end{equation}
($\Lambda_\r'$ is the vector space spanned by the lattice $\Lambda'$.)

Conversely, suppose $\lambda = \alpha + \gamma \in \alpha + \Lambda'$ is a shifted lattice point which lies in the polytope $\Delta(X, \L)$. By \eqref{equ-moment-polytope} 
for any $i = 1, \ldots, s$ we have $\ell_{D_i}(\gamma) \geq -a_i$. And hence if $f$ is a $B$-eigenfunction with weight $\gamma$ then for any $i$ we have $\ord_{D_i}(f) \geq -a_i$
which implies that $f\sigma \in H^0(X, \L)$. Thus we have proved:
\begin{Th}[Brion] \label{th-H^0-moment-poly}
$$H^0(X, \L) = \bigoplus_{\lambda \in \Delta(X, \L) \cap (\alpha + \Lambda')} V_\lambda.$$
\end{Th}

Theorem \ref{th-H^0-moment-poly} immediately gives us the dimension of the space of sections $H^0(X, \L)$:
\begin{equation} \label{equ-H^0-dim}
\dim(H^0(X, \L)) = \sum_{\lambda \in \Delta(X, \L) \cap (\alpha + \Lambda')} \dim(V_\lambda) 
%\color{blue} = \#(\tilde{\Delta}(X, \L) \cap (\alpha+\Lambda' \times \z^N)). \color{black}
\end{equation}

For a rational polytope $\Delta \subset \Lambda_\r$ and $a \in \Lambda$ we define the number $S(\Delta, a)$ by:
\begin{equation} 
S(\Delta, a) = \sum_{\lambda \in \Delta \cap (a + \Lambda')} f(\lambda), 
\end{equation}
where $f$ is the {\it Weyl polynomial} given by the formula:
\begin{equation} \label{equ-Weyl-poly}
f(\lambda) = \dim(V_\lambda) = \prod_{\alpha \in \Phi^+} \langle \lambda+\rho, \alpha \rangle / \langle \rho, \alpha \rangle.
\end{equation}
By the Weyl dimension formula, for any dominant weight $\lambda$, the dimension of $V_\lambda$ is equal to $f(\lambda)$.
With this notation in place, we can restate \eqref{equ-H^0-dim} as:

\begin{Prop} \label{prop-H^0-S}
For any interger $m > 0$ we have:
$$\dim(H^0(X, \L^{\otimes m})) = S(m\Delta(X, \L), m\alpha).$$
\end{Prop}

%\begin{equation} \label{equ-H^0-S}
%\dim(H^0(X, \L)) = S(\Delta(X, \L), \alpha).
%\end{equation}

%\color{blue}
%Theorem about Euler characteristic of a tensor product of bundles.
%\begin{Th}
%Let $\L_1, \ldots, \L_k$ be $G$-linearized line bundles on projective spherical $G$-variety $X$. For any integers $m_1, \ldots, m_k$ we have the following:
%\begin{itemize}
%\item[(1)]
%$$\chi(X, \L_1^{\otimes m_k} \otimes \cdots \otimes \L_k^{\otimes m_k}) = S(\Delta(X, \L_1^{\otimes m_k} \otimes \cdots \otimes \L_k^{\otimes m_k})).$$
%\item[(2)] Moreover, assume that the moment polytope map is additive i.e. for all integers $m_1, \ldots, m_k$,  $\Delta() = \Delta() + \cdots + \Delta()$. Then:
%$$\chi(X, \L_1^{\otimes m_k} \otimes \cdots \otimes \L_k^{\otimes m_k}) = S(m_1\Delta_1 + \cdots + m_k\Delta_k),$$
%where $\Delta_i$ is the moment polytope $\Delta(X, \L_i)$. Note that the lefthand side of the above is the polynomial in Section ... We also note that when 
%$k=1$ the moment polytope 
%map is always additive. 
%\end{itemize}
%\end{Th}
%\color{black}

Next we consider the line bundle $\L^{-1}$. To describe the cohomology group $H^\kappa(X, \L^{-1})$, where $\kappa = \kappa(\L)$, we need some more notation.
Let $\Delta$ be a rational polytope in $\Lambda_\r$ and $a \in \Lambda$. By Corollary \ref{cor-virtual-3}, the function:
$\Psi(m) = S(m\Delta, ma)$ is a polynomial. We put $S^\circ(\Delta, a)$ to be $(-1)^{\dim(\Delta)+d_\Delta} \Psi(-1)$, 
where $d_\Delta$ is the degree of the Weyl polynomial $f$ restricted to the affine span of $\Delta$ in $\Lambda_\r$. 
In the light of Theorem \ref{th-lambda-virtual}(1), $S^\circ(\Delta, a)$ can also be defined as: 
\begin{equation} \label{equ-S^circ}
S^\circ(\Delta, a) = (-1)^{d_\Delta} \sum_{\lambda \in \Delta^\circ \cap (a + \Lambda')} f(-\lambda).
\end{equation}
As usual $\Delta^\circ$ denotes the interior of $\Delta$ (in the topology of the affine span of $\Delta$).

The following theorem determines the dimension of the cohomology group $H^\kappa(X, \L)$ where $\kappa = \kappa(\L)$ is the Itaka dimension. We note that 
by Theorem \ref{th-H^0-moment-poly} the Itaka dimension $\kappa = \kappa(\L)$ is equal to $\dim(\Delta(X, \L)) + d_{\Delta(X, \L)}$.
\begin{Th} \label{th-coh-L^-1}
$\dim(H^\kappa(X, \L^{-1})) = S^\circ(\Delta(X, \L), \alpha).$
%\color{blue} = \#(\Delta^\circ(X, \L) \cap (\alpha+\Lambda' \times \z^N)).$$ \color{black} 
\end{Th}
\begin{proof}
%Consider the function $f: \z \to \z$:
%$$f(m) = \chi(X, \L^{\otimes m}) = \sum_{i=0}^n (-1)^i \dim(H^i(X, \L^{\otimes m})).$$
%By the Hirzebruch-Riemann-Roch theorem we know that $f$ is a polynomial for all $m \in \z$. We note that by Theorem \ref{th-coh-line-bundle}(a) and Theorem 
%\ref{th-H^0-moment-poly}
%for all $m>0$ we have 
%\begin{equation} \label{equ-f(m)}
%f(m) = \dim(H^0(X, \L^{\otimes m})) = S(m\Delta(X, \L), m\alpha). 
%\end{equation}
%We also know
%by Theorem \ref{th-coh-line-bundle}(b) that  for all $i \neq \kappa(\L)$ the cohomology groups $H^i(X, \L^{-1})$ vanish. 
%Thus we have: 
%\begin{equation} \label{equ-f(-1)}
%f(-1) = (-1)^\kappa \dim(H^\kappa(X, \L)).
%\end{equation}
By Corollary \ref{cor-virtual-3} we know that $\Psi(m) = S(m\Delta(X, \L), m\alpha)$ is a polynomial in $m$. Moreover, by Corollary \ref{cor-dim-H^0-polynomial}
we know that for any nonnegative integer $m$:
$$\chi(X, \L^{\otimes m}) = \dim(H^0(X, \L^{\otimes m})) = S(m\Delta(X, \L), m\alpha),$$
where $\chi(X, \L^{\otimes m})$ denotes the Euler characteristic of the bundle $\L^{\otimes m}$.  
Putting these together we conclude that $\dim(H^\kappa(X, \L^{-1})) = (-1)^{\kappa} \Psi(-1)$. Note that 
the polynomial $\chi(X, \L^{\otimes m})$ has degree $\kappa$ which is equal to $\dim(\Delta) + d_\Delta$. %by Corollary \ref{cor-virtual-3}. 
This finishes the proof.
\end{proof}

\subsection{Newton-Okounkov polytope} \label{subsec-NO-poly}
Fix a reduced decomposition $\w$ for the longest element $w_0$ in the Weyl group of $G$.
To any $\lambda$ in the positive Weyl chamber $\Lambda_\r^+$ 
one can associate a convex rational polytope $\Delta_\w(\lambda) \subset \r^N$ called the {\it string polytope} associated to $\lambda$ (and $\w$).
Here $N$ is the number of positive roots. The polytope $\Delta_\w(\lambda)$ has the property that, when $\lambda$ is a dominant weight, 
the number of integral points in $\Delta_\w(\lambda)$ is equal to the dimension of 
the irreducible $G$-module $V_\lambda$, i.e.:
\begin{equation} \label{equ-dimV-lambda-string-poly}
\dim(V_\lambda) = \#(\Delta_\w(\lambda) \cap \z^N).
\end{equation} 
In fact, the integral points in $\Delta_\w(\lambda)$ are in one-to-one correspondence with the so-called {\it canonical basis} for $V_\lambda$
(\cite{Littelmann, BZ}). 

The string polytopes generalize the well-known Gelfand-Zetlin polytopes $\Delta_{\textup{GZ}}(\lambda)$ associated to irreducible representations of $\GL(n, \c)$ (\cite{GZ}). That is, for a specific choice of a reduced 
decomposition, namely $w_0 = (s_1)(s_2s_1)(s_3s_2s_1)иии(s_ns_{n-1} \cdots s_1)$, for the longest element of the Weyl group of $\GL(n, \c)$ the string polytopes can be 
identified with the Gelfand-Zetlin polytopes. 

\begin{Rem} \label{rem-string-poly-add}
The dependence of $\Delta_\w(\lambda)$ is piecewise linear, in the sense that there is a rational polyhedral cone $C_\w$ (with apex at the origin) in the vector space 
$\r^N \times \Lambda_\r$ such that for each $\lambda \in \Lambda_\r^+$ the string polytope $\Delta_\w(\lambda)$ is the slice of the cone $C_\w$ at $\lambda$, i.e. 
$\Delta_\w(\lambda) = C_\w \cap \pi^{-1}(\lambda)$ where $\pi: \r^N \times \Lambda_\r \to \Lambda_\r$ is the projection on the second factor. Note that we can then define 
$\Delta_\w(\lambda)$ for all $\lambda \in \Lambda_\r$ (not just integral $\lambda$).

Moreover, when $G = \GL(n, \c)$ it is easy to see 
from the defining inequalities of the Gelfand-Zetlin polytopes that $\lambda \mapsto \Delta_{\textup{GZ}}(\lambda)$ is additive, namely for any $\lambda, \gamma \in \Lambda_\r^+$ we have:
$$\Delta_{\textup{GZ}}(\lambda + \gamma) = \Delta_{\textup{GZ}}(\lambda) + \Delta_{\textup{GZ}}(\gamma).$$
\end{Rem}

%\begin{Rem}
%In the past few years interesting constructions have been introduced by V. Kiritchenko and D. Anderson ....
%\end{Rem}

Now let $X$ be a spherical $G$-variety with a $G$-linearized line bundle $\L$. Let $R = \bigoplus_m R_m$ be a graded $G$-subalgebra of the ring of sections
$\bigoplus_m H^0(X, \L^m)$. Moreover assume that $R_m$ is finite dimensional for all $m$ (this is automatic if $X$ is projective). We call $R$ a {\it graded 
$G$-linear system}. We would like to associate a polytope $\tilde{\Delta}(R)$ to $R$ which is responsible for the dimensions of the homogeneous pieces $R_m$.
\begin{Def} \label{def-NO-poly}
Let $\tilde{\Delta}(R)$ denote the Newton-Okounkov polytope of $R$, that is, the polytope over $\Delta(R)$ with string polytopes $\Delta_{\w}(\lambda)$ as fibers:
\begin{equation} \label{equ-NO-poly}
\tilde{\Delta}(R) = \bigcup_{\lambda \in \Delta(R)} (\{ \lambda \} \times \Delta_{\w}(\lambda)) \subset \Lambda^+_\r \times \r^N. 
\end{equation}
If $\L$ is a $G$-linearized line bundle on a projective spherical $G$-variety $X$ we denote the Newton-Okounkov polytope of the ring of sections
$R(X, \L)$ by $\tilde{\Delta}(X, \L)$.
\end{Def}

\begin{Rem} \label{rem-NO-body}
The above notion of the Newton-Okounkov polytope of a graded $G$-linear system over a spherical variety is a special case of the more general notion of a 
{\it Newton-Okounkov body} of a graded linear system on an arbitrary variety (see \cite{LM, KKh-Annals} and the references therein).
\end{Rem}

Let $\L$ be a $G$-linearized line bundle on a normal projective spherical $G$-variety $X$ which contains $G/H$ as the open $G$-orbit. 
As usual $\Lambda' = \Lambda(G/H)$ denotes the lattice of weights of $B$-eigenfunctions in $\c(G/H)$. 
Let us assume that $H^0(X, \L) \neq \{0\}$ and fix a weight $\alpha$ of a $B$-eigensection in $H^0(X, \L)$. 

Let $\tilde{\Delta}$ be a rational polytope in $\r^n$. Fix a lattice $L \subset \r^n$ and a point $a \in \r^n$.
We denote by $N(\tilde{\Delta}, a)$ (respectively $N^\circ(\tilde{\Delta}, a)$) 
the number of points in the shifter lattice $a + L$ which lie in $\tilde{\Delta}$ (respectively in the interior of $\tilde{\Delta}$ in the topology of its affine span). 
We note that, by Theorem \ref{th-lambda-virtual}(1), $N^\circ(\tilde{\Delta}, a)$ is equal to the $(-1)^{\dim(\tilde{\Delta})}$ times the value of the polynomial 
$m \mapsto N(m\tilde{\Delta}, ma)$ at $m=-1$. These notations slightly extend $N$ and $N^\circ$ defined in Section \ref{sec-toric-12}.  

\begin{Prop}[Newton-Okounkov polytope and dimension of space of sections] \label{prop-NO-poly-dim-H^0}
Fix the lattice $\Lambda' \times \z^N$ in the (real) vector space $\Lambda_\r \times \r^N$. Consider $\alpha \in \Lambda$ as an element of $\Lambda \times \z^N$.
Then for any integer $m>0$ we have:
$$\dim(H^0(X, \L^{\otimes m})) = N(m\tilde{\Delta}(X, \L), m\alpha).$$
Moreover:
$$\dim(H^\kappa(X, \L^{-1})) = N^\circ(\tilde{\Delta}(X, \L), \alpha),$$
where as before $\kappa = \kappa(\L)$ is the Itaka dimension. 
\end{Prop}
\begin{proof}
The proposition follows from Proposition \ref{prop-H^0-S} and Theorem \ref{th-coh-L^-1}.
\end{proof}

%\color{blue}
%Theorem about Euler characteristic of a tensor product of bundles.
%\begin{Th}
%Let $\L_1, \ldots, \L_k$ be $G$-linearized line bundles on projective spherical $G$-variety $X$. For any integers $m_1, \ldots, m_k$ we have the following:
%\begin{itemize}
%\item[(1)]
%$$\chi(X, \L_1^{\otimes m_k} \otimes \cdots \otimes \L_k^{\otimes m_k}) = N(\tilde{\Delta}(X, \L_1^{\otimes m_k} \otimes \cdots \otimes \L_k^{\otimes m_k})).$$
%\item[(2)] Moreover, assume that the Newton-Okounkov polytope map is additive, i.e. for all integers $m_1, \ldots, m_k$,  $\Delta() = \Delta() + \cdots + \Delta()$. Then:
%$$\chi(X, \L_1^{\otimes m_k} \otimes \cdots \otimes \L_k^{\otimes m_k}) = N(m_1\tilde{\Delta}_1 + \cdots + m_k\tilde{\Delta}_k).$$
%where $\Delta_i$ is the moment polytope $\Delta(X, \L_i)$. Note that the lefthand side of the above is the polynomial in Section ... We also note that when 
%$k=1$ the moment polytope map is always additive. 
%\end{itemize}
%\end{Th}
%\color{black}

\subsection{Euler characteristic of line bundles over a projective spherical variety} \label{subsec-Euler-char-sph}
Let $X$ be a projective spherical $G$-variety. 
In this section we consider the cases where the moment polytope or the Newton-Okounkov polytope are additive. In these cases, extending the case of toric varieties, 
the formula for Euler characteristic of line bundles on $X$ can be expressed as polynomial measures on convex chains (Section \ref{subsec-finitely-add-measure}).

Let $\L_1,\ldots, \L_k$ be $G$-linearized line bundles over a projective spherical variety $X$. We assume that the bundles are globally generated 
(i.e. the $H^0(X, L_i)$ are base point free linear systems on $X$). Consider the bundle $\L_1^{\otimes m_1} \otimes \cdots \otimes\L_k^{\otimes m_k}$.
Consider any collection of non negative numbers $m_1,\dots,m_k$, such that at least one number is not zero. Below we will assume that one the following two conditions on these bundles hold:
\begin{itemize}
\item[(I)] The Newton-Okounkov polytope $\tilde{\Delta}$ of the bundle $\L_1^{\otimes m_1} \otimes \cdots \otimes \L_k^{\otimes m_k}$ is 
equal to the Minkowski sum $\tilde{\Delta}_1 + \cdots + \tilde{\Delta}_k$ where $\tilde{\Delta}_1, \ldots, \tilde{\Delta}_k$ are 
the Newton-Okounkov polytope of the bundles $\L_1, \ldots, \L_k$ respectively.

\item[(II)] The moment polytope $\Delta$ of the bundle $\L_1^{\otimes m_1} \otimes \cdots \otimes \L_k^{\otimes m_k}$ is equal to the Minkowski 
sum $\Delta_1 + \cdots + \Delta_k$ where $ \Delta_1, \ldots, \Delta_k$ are the  moment polytopes of the bundles $\L_1,\dots,\L_k$ respectively.
\end{itemize}

\begin{Rem} \label{rem-additive}
(1) For $k=1$ The Conditions I and II always hold but when $k>1$ they do not always hold.

(2) The condition (II) above holds in the classes of horospherical homogeneous spaces as well as the group case 
(see Sections \ref{subsec-horosph} and \ref{subsec-gp-embed}). The condition (I) holds in these
two classes of examples provided that the string polytopes are also additive. This for example happens for the well-known Gelfand-Zetlin polytopes 
(which are a special case of the string polytopes for $G = \GL(n, \c)$).
\end{Rem}

Let as usual $\Lambda'$ denotes weight lattice of $X$, i.e. the sublattice of the weight lattice $\Lambda$ consisting of all the weights of $B$-eigenfunctions on $X$.
For each $i=1, \ldots, k$ fix a weight $a_i \in \Lambda$ of a $B$-eigensection in $H^0(X, \L_i)$.  
\begin{Th}
Assume that for the bundles $\L_1, \ldots, \L_k$ the condition (I) holds. Then for any $k$-tuple of integral numbers $m_1,\dots,m_k$ the Euler characteristic of $X$ with 
the coefficients in the bundle $\L_1^{\otimes m_1} \otimes \cdots \otimes \L_k^{\otimes m_k}$ is given by: 
\begin{equation} \label{equ-Euler-char-1}
\chi(X, \L_1^{\otimes m_1} \otimes \cdots \otimes \L_k^{\otimes m_k}) = 
\sum_{x \in (m_1a_1 + \cdots + m_ka_k)+(\Lambda' \times \z^N)} \tilde{\gamma}_1^{m_1}* \cdots * \tilde{\gamma}_k^{m_k},
\end{equation}
where $\tilde{\gamma}_i$ is the characteristic function of the Newton-Okounkov polytope $\tilde{\Delta}(X, \L_i)$.
\end{Th}
\begin{proof}
By the Hirzbruch-Riemann-Roch theorem the Euler characteristic is a polynomial in the $m_1,\dots,m_k$. Moreover by Proposition \ref{prop-NO-poly-dim-H^0}
it coincides with the righthand side of \eqref{equ-Euler-char-1} for $m_i\geq 0$. On the other hand, by Corollary \ref{cor-virtual-3} 
the righthand side of \eqref{equ-Euler-char-1} is also a polynomial in $m_1,\dots,m_k$. Thus the two polynomials must coincide.
\end{proof}
\begin{Th} 
Assume that for the bundles $\L_1,\dots, \L_k$ the condition (II) holds. Then for any $k$-tuple of integral numbers $m_1,\dots,m_k$ the 
Euler characteristic of $X$ with the coefficients in the bundle $\L_1^{\otimes m_1} \otimes \cdots \otimes \L_k^{\otimes m_k}$ is given by: 
\begin{equation} \label{equ-Euler-char-2}
\chi(X, \L_1^{\otimes m_1} \otimes \cdots \otimes \L_k^{\otimes m_k}) 
= \sum_{x \in (m_1a_1 + \cdots + m_ka_k) + \Lambda'} f(x) \gamma_1^{m_1} * \cdots* \gamma_k^{m_k},
\end{equation}
where $\gamma_i$ is the characteristic function of the moment polytope $\Delta(X, \L_i)$ and $f$ is the Weyl polynomial (see \eqref{equ-Weyl-poly}).
\end{Th}
\begin{proof} 
By the Hirzbruch-Riemann-Roch theorem  
%By Corollary \ref{cor-dim-H^0-polynomial}, Proposition \ref{prop-H^0-S} and Theorem \ref{th-coh-L^-1}, 
the Euler characteristic is a polynomial in the $m_1,\dots,m_k$. Also by Proposition \ref{prop-H^0-S}
it coincides with the righthand side of \eqref{equ-Euler-char-2} for $m_i \geq 0$. Again 
by Corollary \ref{cor-virtual-3} the righthand side is also polynomial and the two polynomials must coincide.
\end{proof}

\subsection{Invariant linear systems on a spherical homogeneous space}
In this section we discuss invariant linear systems on a spherical homogenous space $G/H$ and their associated moment polytopes. 
First we consider the case of invariant subspaces of regular functions and then the general case of invariant linear systems. 
\subsubsection{Moment polytope of an invariant subspace of regular functions}
Let us assume that the homogeneous space $G/H$ is quasi-affine. In this section we consider the case of trivial line bundles, i.e. we have 
finite dimensional subspaces of regular functions on $G/H$. This situation is closer to the classical toric case and Newton polytope theory which is concerned with finite dimensional 
subspaces of Laurent polynomials spanned by monomials. 
\begin{Ex}
\begin{itemize}
\item[(a)] Let $H = U$ be a maximal unipotent subgroup of $G$. One can show that $G/U$ is quasi-affine and as a $G$-module (for the left $G$-action on $G/U$) 
the space of regular functions $\c[G/U]$ decomposes as:
$$\c[G/U] = \bigoplus_{\lambda \in \Lambda^+} V_\lambda.$$
\item[(b)] Consider the left-right action of $G \times G$ on $G$. Then $G \cong (G \times G) / G_{diag}$ is affine and the space of 
regular functions on $G$, as a $(G \times G)$-module, decomposes as:
$$\c[G] = \bigoplus_{\lambda \in \Lambda^+} \End(V_\lambda).$$ 
\end{itemize}
\end{Ex}

As usual let $\Lambda' = \Lambda(G/H) \subset \Lambda$ be the lattice of weights of $B$-eigenfunction in $\c(G/H)$. Also let $\Lambda'^+ = \Lambda^+(G/H)$ denote the 
weights of the $B$-eigenfunctions in the algebra $\c[G/H]$. It is clear that $\Lambda'^+$ is a semigroup in $\Lambda$.
One knows that the $G$-algebra $\c[G/H]$ decomposes into finite dimensional irreducible $G$-modules. Thus every $B$-eigenfunction in 
$\c[G/H]$ is actually a highest weight vector for $G$ and generates an irreducible $G$-module. Thus: 
$$\Lambda'^+ = \{ \lambda \in \Lambda^+ \mid V_\lambda \textup{ appear in } \c[G/H] \}.$$ 
The following is well-known (see \cite{Timashev}):
\begin{Prop}
The semigroup $\Lambda'^+$ generates the lattice $\Lambda'$. 
%Moreover, the semigroup $\Lambda'^+$ is the intersection of the lattice
%$\Lambda'$ with the positive Weyl chamber $\Lambda^+_\r$.
\end{Prop}

Now let $\A \subset \Lambda'^+$ be a finite subset. To $\A$ there corresponds the finite dimensional $G$-invariant subspace: 
$$L_\A = \bigoplus_{\lambda \in \A} V_\lambda \subset \c[G/H].$$

Since $L_\A$ is $G$-invariant it is automatically base point free and hence its Kodaira map is defined everywhere on $G/H$.
Let $\Phi_\A: G/H \to \p(L_\A^*)$ denote the Kodaira map of the $G$-invariant subspace $L_\A$. It is a $G$-equivariant morphism. Also let 
$Y_\A$ denote the closure of the image of $G/H$ in the projective space $\p(L_\A^*)$. We note that the dimension of the projective variety $Y_\A$ could 
possibly be smaller than that of $G/H$. 

Let $\overline{R(\A)}$ be the integral closure of the algebra $R(\A) = \bigoplus_{m} L_\A^m$ in $\bigoplus_m \c(G/H)$. We know that:
$$\overline{R(\A)} = \bigoplus_m \overline{L_\A^m},$$ where $\overline{L}$ denotes the integral closure (or completion) of a subspace $L$ in the field of rational functions 
$\c(G/H)$ (see \cite[Appendix 4]{SZ}). We denote the moment polytope of this graded algebra by $\Delta(\A)$. It is a convex polytope containing $\A$. 

We have the following description of the $G$-modules $\overline{L_\A^m}$ in terms of the moment polytope $\Delta(\A)$:
\begin{Th} \label{th-subspace-moment-poly}
For every integer $m > 0$, the $G$-module $\overline{L_\A^m}$ decomposes as:
$$\overline{L_\A^m} = \bigoplus_{\lambda \in m\Delta(\A) \cap \Lambda'} V_\lambda.$$
\end{Th}
\begin{proof}
As in the proof of Theorem \ref{th-suff-full-exists}, given the $G$-invariant subspace $L_\A$ 
we can find a normal projective spherical $G$-variety $X$ which contains $G/H$ as the open orbit and moreover
the Kodaira map $\Phi_\A$ extends to the whole $X$.  
Let $\L = \Phi_\A^*(\mathcal{O}(1))$ be the pull-back of the line bundle $\mathcal{O}(1)$ on the projective space to 
$X$. One shows that for each $m>0$ the integral closure $\overline{L_\A^m}$ can be identified with $H^0(X, \L^{\otimes m})$.
The theorem now follows from Theorem \ref{th-H^0-moment-poly}. Note that $\alpha$ is the identity character because $\c[G/H]$ contains the constant function $1$ 
which is invariant under the action of $G$ and hence is a $B$-eigensection with weight $0$.
\end{proof}

\subsubsection{Moment polytope of an invariant linear system} \label{subsub-lin-system}
Let $\E$ be a $G$-linearized line bundle on $G/H$. 
Then the space of sections $H^0(G/H, \E)$ is a $G$-module. Moreover, since $G/H$ is spherical, $H^0(G/H, \E)$ is multiplicity-free, that is, 
every irreducible $G$-module appears in it with multiplicity $0$ or $1$.

Take a finite nonempty subset $\A$ of the $G$-spectrum of the space of global sections $H^0(G/H, \E)$. Let $E_\A$ denote the $G$-invariant subspace of $H^0(G/H, \E)$
determined by $\A$, that is:
$$E_\A = \bigoplus_{\lambda \in \A} V_\lambda.$$
%Since $E_\A$ is $G$-invariant it is automatically base point free on $G/H$.
Since $E_\A$ is $G$-invariant then the base locus of $E_\A$, i.e. the locus of points where all the sections in $E_\A$ vanish, is a $G$-invariant subvariety of $G/H$. 
But as $E_\A$ is nonzero we conclude that the base locus of $E_\A$ is empty, in other words, $E_\A$ is base point free. 
The base point free linear system $E_\A$ gives rise to a Kodaira map $\Phi_{\A}: X \to \p(E_\A^*)$, the projective space of the dual space $E_\A^*$. We denote the 
closure of the image of the Kodaira map by $Y_\A$. The Kodaira map is $G$-equivariant and hence $Y_\A$ is a $G$-stable subvariety of the projective space $\p(E_\A^*)$. Note that the 
dimension of the projective variety $Y_\A$ may be smaller than that of $X$. To $\A$ we can associate a graded algebra:
$$R(\A) = \bigoplus_{m} E_\A^m,$$
where $E_\A^m$ denotes the image of $E_\A \otimes \cdots \otimes E_\A$ ($m$ times) in $H^0(X, \E^{\otimes m})$ under the product map $H^0(X, \E) \otimes \cdots \otimes H^0(X, \E)$ ($m$ times).
The algebra $R(\A)$ is a graded $G$-subalgebra of the ring of sections $\bigoplus_{m} H^0(X, \E^{\otimes m})$. 
The homogneous coordinate ring of the projective variey $Y_\A \subset \p(E_\A^*)$ can naturally be identified with $R(\A)$. 

Let $\c(\E)$ denote the space of meromorphic sections of $\E$. It can be identified with $\c(G/H)$ via a choice of a nonzero section $\sigma \in \c(\E)$. 
We denote the integral closure of the algebra $R(\A)$ in $\bigoplus_{m} \c(\E^{\otimes m})$ by $\overline{R(\A)}$.

We have the following extension of Theorem \ref{th-subspace-moment-poly}:
\begin{Th} \label{th-lin-system-moment-poly}
For every integer $m > 0$, the $G$-module $\overline{E_\A^m}$ decomposes as:
$$\overline{E_\A^m} = \bigoplus_{\lambda \in m\Delta(\A) \cap (m\alpha + \Lambda')} V_\lambda.$$
Here $\alpha$ denotes the weight of a $B$-eigensection $\sigma$ in $H^0(G/H, \E)$.
\end{Th}
\begin{proof}
As in the proof of Theorem \ref{th-subspace-moment-poly}.
\end{proof}

\subsection{Simultaneous resolution of singularities} \label{subsec-resolution}
Let $\E_1, \ldots, \E_k$ be globally generated $G$-linearized line bundles on a spherical homogeneous space $G/H$. 
For each $i=1, \ldots, k$ let $E_i$ be a nonzero $G$-invariant linear system for $\E_i$ i.e. $E_i$ is a finite dimensional $G$-invariant subspace of 
$H^0(X, \E_i)$. 
%For each $i$ since $E_i$ is $G$-invariant then the base locus of $E_i$, i.e. the locus of points where all the sections in $E_i$ vanish, is a $G$-invariant subvariety of $G/H$. 
%But as $E_i$ is nonzero we conclude that the base locus of $E_i$ is empty, in other words, $E_i$ is base point free. 
Each $E_i$ is $G$-invariant and hence it is base point free. Thus the Kodaira map $\Phi_{E_i}$ is defined on the whole $G/H$. As usual we denote the closure of the image of the 
Kodaira map $\Phi_{E_i}$ by $Y_{E_i}$. It is a projective $G$-subvariety of $\p(E_i^*)$.

We will be interested in a generic complete intersection of $E_1, \ldots, E_k$ in $G/H$. To apply topological methods, we need to work with a 
compactification of $G/H$ which behaves nicely with respects to the linear systems $E_i$. This is the content of the next definition. It is an extension of 
the similar notion for toric varieties (\cite{Askold-toroidal}).
\begin{Def}
Let $X$ be a $G$-equivariant completion of $G/H$, that is, $X$ is a complete spherical $G$-variety which contains $G/H$ as the open orbit.
Let us say that $X$ is {\it sufficiently complete} with respect to the linear systems $E_1, \ldots, E_k$ if: 
\begin{itemize}
\item[(a)] $X$ is smooth.
\item[(b)] For each $i=1, \ldots, k$, the Kodaira map $\Phi_{E_i}$ extends to 
a morphism on the whole $X$.
\end{itemize}
\end{Def}

\begin{Th} \label{th-suff-full-exists}
Given base point free $G$-invariant linear systems $E_1, \ldots, E_k$ on $G/H$ there exists a projective spherical variety $X$ which is sufficiently complete with respect to 
$E_1, \ldots, E_k$.
\end{Th}
\begin{proof}
One knows that the homogeneous space $G/H$ has an equivariant projective completion. 
Let $Z$ be such a projective completion, i.e. $Z$ is a projective spherical $G$-variety which contains $G/H$ as the open orbit.
Consider the map $\Phi: G/H \to Z \times Y_{E_1} \times \cdots \times Y_{E_k}$ given by:
$$\Phi(x) = (x, \Phi_{E_1}(x),  \ldots, \Phi_{E_k}(x)),$$
and let $Y$ be the closure of the image of $\Phi$ . The map $\Phi$ is a $G$-equivariant embedding and $Y$ is a projective 
subvariety of $Z \times Y_{E_1} \times \cdots \times Y_{E_k}$. We note that $Y$ contains $G/H$ as an open orbit and hence is a spherical $G$-variety.
We have the following commutative diagram:
\begin{equation}
\xymatrix{
G/H \ar[r] \ar[rd]^{\Phi_{E_1 \cdots E_k}} & \p(E_1^*) \times \cdots \times \p(E_k^*) \ar[d] \\
& \p((E_1 \cdots E_k)^*) \\
}
\end{equation}
where the vertical arrow is the Segre map and $E_1 \ldots E_k$ is the linear system which is the image of $E_1 \otimes \cdots \otimes E_k$ in $H^0(G/H, \E_1 \otimes \cdots \otimes \E_k)$. 
Thus in defining $Y$, instead of the Kodaira maps $\Phi_{E_1}, \ldots, \Phi_{E_k}$, 
we alternatively could use the Kodaira map $\Phi_{E_1 \cdots E_k}$.
 
The following theorem guarantees that $Y$ has a $G$-equivariant resolution of singularities (see \cite{Perrin}).
\begin{Th}[Resolution of singularities for spherical varieties] \label{th-spherical-res-sing}
Every spherical $G$-variety has a $G$-equivariant resolution of singularities. 
%More precisely if $Y$ is a spherical $G$-variety then there is 
%a smooth projective spherical $G$-variety $X$ and a $G$-equivariant morphism $\pi: X \to Y$.
\end{Th}

Let $\pi: X \to Y$ be a $G$-equivariant resolution of singularities of $Y$ as in the above theorem.
For each $i$ let $\pi_i$ be the projection $Z \times Y_{E_1} \times \cdots \times Y_{E_k} \to Y_{E_i}$ restricted to $Y$.
For each $i=1, \ldots, k$ we have a commutative diagram:
\begin{equation} \label{equ-comm-diag}
\xymatrix{
X \ar[r]^{\pi} & Y \ar[d]^{\pi_i} \\
G/H \ar[ur]^{\Phi}  \ar[r]^{\Phi_{E_i}} & Y_{E_i} \\
}
\end{equation}
Note that $\Phi$ (respectively $\pi$) is an isomorphisms from $G/H$ (respectively the open orbit in $X$) to the open orbit in $Y$. Let us identify the homogeneous space 
$G/H$ with the open orbit in $X$ via $\pi^{-1} \circ \Phi$.
Now for each $i$ define $\tilde{\Phi}_{E_i}$ to be $\pi_i \circ \pi$. Clearly $\tilde{\Phi}_{E_i}$ is $G$-equivariant and extends the Kodaira map 
$\Phi_{E_i}: G/H \to Y_{E_i}$. This finishes the proof of the theorem.
\end{proof}

%\color{blue} The following is a corollary of Bertini-Sard's theorem and is proved in Askold's notes. \color{black}
The following is a direct corollary of Thom's transversality theorem (Theorem \ref{th-7}):
\begin{Th}[Transversality of generic hyperplane sections]
As above let $E_1, \ldots, E_k$ be $G$-invariant nonzero linear systems on a spherical homogeneous space $G/H$ and $X$ 
a sufficiently complete completion of $G/H$ with respect to the $E_i$.
For $f_i \in E_i$ let $H_i = \overline{\{ x \in G/H \mid f_i(x) = 0\}} \subset X$ be the closure of the hypersurface defines by $f_i$. Let $f_i \in E_i$ be generic. We then have:
\begin{itemize}
\item[(1)] For each $i$ the hypersurface $H_i \subset X$ is smooth.
\item[(2)] For each $i$ the hypersurace $H_i$ intersects all the $G$-orbits in $X$ transversely.
\item[(3)] Either the intersection of the hypersurfaces $H_1, \ldots, H_k$ is empty or they intersect transversely.
\end{itemize}
\end{Th}

\subsection{Genus and $h^{p,0}$ numbers of a complete intersection} \label{subsec-genus-sph}
As in Section \ref{subsec-resolution} let $\E_1, \ldots, \E_k$ be a $G$-linearized line bundle on $G/H$. For each $i=1, \ldots, k$ fix a dominant weight $\alpha_i$ for a 
$B$-eigensection in $H^0(G/H, \E_i)$. Also for each $i=1, \ldots, k$ let $\A_i$ be a finite subset of $\Spec_G(H^0(G/H, \E_i))$ and: 
$$E_i = \bigoplus_{\lambda \in \A_i} V_\lambda \subset H^0(G/H, \E_i),$$ 
the corresponding $G$-invariant linear system.
Also we denoted by $\Delta(E_i)$ the moment polytope of the $G$-algebra $\overline{\bigoplus_m E_i^m}$. 

By Theorem \ref{th-suff-full-exists} we can find a sufficiently complete projective completion $X$ of $G/H$ with respect to $E_1, \ldots, E_k$. 
For each $i=1, \ldots, k$ let $\L_i = \Phi_{E_i}^*(\mathcal{O}(1))$ be the pull-back of the line bundle $\mathcal{O}(1)$ on the projective space $\p(E_i^*)$ 
to $X$. The line bundle $\L_i$ is globally generated because it is the pull-back of a globally generated line bundle.

One shows that for any $m>0$ the integral closure $\overline{E^m_i}$ can be identified with the space of sections $H^0(X, \L_i^{\otimes m})$.
Hence the moment polytope $\Delta(X, \L_i)$ coincides with $\Delta(E_i)$, that is the moment polytope of the integral closure $\overline{\bigoplus_m E^m_i}$ (in particular the 
moment polytope $\Delta(X, \L_i)$ is independent of the sufficiently complete completion $X$).

The next theorem gives a necessary and sufficient condition for a generic complete intersection $X_k$ from $E_1, \ldots, E_k$ to be nonempty, 
in terms of the dimensions of the Newton-Okounkov or moment polytopes. It is a direct corollary of 
Theorems \ref{th-nec-cond-nonempty} and \ref{th-suff-cond-nonempty}.
\begin{Th} \label{th-sph-nonempty}
%The necessary and sufficient condition for nonemptiness of a generic complete intersection in terms of the dimensions of the NO polytopes here.
A generic complete intersection $X_k$ from $E_1, \ldots, E_k$ is nonempty if and only if $E_1, \ldots, E_k$ are independent. That is,
for any $J \subset \{1,\ldots, k\}$ we have $\dim(\tilde{\Delta}_J) \geq |J|$ (equivalently $\dim(\Delta_J)+d_J \geq |J|$).
Here $\tilde{\Delta}_J$ (respectively $\Delta_J$) is the Newton-Okounkov polytope (respectively the moment polytope) of the linear system $E_J = \prod_{i \in J} E_i$ and
$d_J$ is the degree of the Weyl polynomial restricted to the affine span of $\Delta_J$. 
\end{Th}
For the rest of the paper we assume that the condition in Theorem \ref{th-sph-nonempty} is satisfied and hence a generic complete intersection $X_k$ in nonempty.
%For the rest of this section we will assume that a generic complete intersection $X_k$ is nonempty. By Theorem \ref{th-suff-cond-nonempty} this is the case if 
%$E_1, \ldots, E_k$ are independent in the sense of Definition \ref{def-independent}.

Let us recall some notation. As in Proposition \ref{prop-NO-poly-dim-H^0} 
fix the lattice $\Lambda' \times \z^N$ in the vector space $\Lambda_\r \times \r^N$. Recall that for a rational polytope $\tilde{\Delta} \subset \Lambda_\r \times \r^N$
and a point $a \in \Lambda_\r \times \z^N$, we denote by $N(\tilde{\Delta}, a)$ (respectively $N^\circ(\tilde{\Delta}, a)$) the number of points in the shifted lattice 
$a + (\Lambda' \times \z^N)$ which lie in $\tilde{\Delta}$ (respectively in the interior of $\tilde{\Delta}$). 
Moreover, we define $N'(\tilde{\Delta}, a)$ to be: $$N'(\tilde{\Delta}, a) = (-1)^{\dim(\tilde{\Delta})} N^\circ(\tilde{\Delta}, a).$$ 
In fact, $N'(\tilde{\Delta}, a)$ is the value of the polynomial $m \mapsto N(m\tilde{\Delta}, ma)$ at $m = -1$. Similarly if $\Delta$ is a rational polytope in the 
vector space $\Lambda_\r$ and $\alpha \in \Lambda$ a weight, we denote by $S(\Delta, \alpha)$ the sum of values of the Weyl polynomial $f$ on 
the shifted lattice points $\alpha + \Lambda'$ which lie in $\Delta$. Moreover, we denote by $S'(\Delta, \alpha)$ the value of the polynomial 
$m \mapsto S(m\Delta, m\alpha)$ at $m = -1$. Using the notation introduced in the paragraph before Theorem \ref{th-coh-L^-1} we can write: 
$$S'(\Delta, \alpha) = (-1)^{\dim(\Delta) + d_\Delta} S^\circ(\Delta, \alpha).$$ Here $S^\circ(\Delta, \alpha)$ 
is $(-1)^{d_\Delta}$ times the sum of values of the polynomial $f(-\lambda)$ on the shifted lattice points $\alpha + \Lambda'$ which lie in the interior of 
$\Delta$.

Also recall that for each $i = 1, \ldots, k$, $\alpha_i \in \Lambda$ is the weight of a $B$-eigensection in $E_i$. The moment polytope $\Delta(E_i)$ lies in the affine subspace
$\alpha_i + \Lambda'_\r$. We also consider $\alpha_i$ as $(\alpha_i, 0)$ in the larger lattice $\Lambda \times \z^N$. The Newton-Okounkov polytope 
$\tilde{\Delta}(E_i)$ lies in the affine subspace $\alpha_i + (\Lambda'_\r \times \r^N)$.

For each $i$, let $\kappa_i$ be the Itaka dimension of the linear system $E_i$.
From Proposition \ref{prop-H^0-S} and Theorem \ref{th-coh-L^-1}, applied to the line bundles $\L_i$ on $X$, we obtain the following: 
\begin{Th} With notation as above we have:
\begin{itemize}
\item[(a)] $$\dim(H^0(X, \L_i)) = S(\Delta(E_i), \alpha_i) = N(\tilde{\Delta}(E_i), \alpha_i).$$
\item[(b)] $$\dim(H^{\kappa_i}(X, \L_i^{-1})) = S^\circ(\Delta(E_i), \alpha_i) = N^\circ(\tilde{\Delta}(E_i), \alpha_i).$$
In particular, the cohomology groups of $(X, \L_i)$ and $(X, \L_i^{-1})$ only depend on the $E_i$, i.e. are independent of the choice of the 
sufficiently complete completion $X$ for $G/H$.
\end{itemize}
 \end{Th}

We now use the above to compute the genus of a complete intersection from $E_1, \ldots, E_k$ in $G/H$.
Let $f_1, \ldots, f_k$ be generic elements in $E_1, \ldots, E_k$ respectively. 
For each $i =1, \ldots, k$ let $D_i = \{ x \in G/H \mid f_i(x) = 0 \}$ be the hypersurface defined by $f_i$
and let $X_k = D_1 \cap \cdots \cap D_k$.

The next theorem is one of our main results. 
\begin{Th}[Genus of a complete intersection in a spherical homogenous space] \label{th-main}
With notation as above we have:
\begin{multline} \label{equ-arith-genus-smooth-sph1}
\chi(X_k) = 1 - \sum_{i_1} S'(\Delta(E_{i_1}), \alpha_{i_1}) + \sum_{i_1 < i_2} S'(\Delta(E_{i_1}E_{i_2}), \alpha_{i_1}+\alpha_{i_2}) - \cdots \\
+ (-1)^k S'(\Delta(E_1 \cdots E_k), \alpha_{i_1} + \cdots + \alpha_{i_k}).
\end{multline}
or equivalently:
\begin{multline} \label{equ-arith-genus-smooth-sph2}
\chi(X_k) =  1 - \sum_{i_1} N'(\tilde{\Delta}(E_{i_1}), \alpha_{i_1}) + \sum_{i_1 < i_2} N'(\tilde{\Delta}(E_{i_1} E_{i_2}), \alpha_{i_1} + \alpha_{i_2}) - \cdots \\
+ (-1)^k N'(\tilde{\Delta}(E_1 \cdots E_k), \alpha_{i_1} + \cdots + \alpha_{i_k}).
\end{multline}
\end{Th}
\begin{proof}
The theorem follows from Theorem \ref{th-genus-complete-int}, Corollary \ref{cor-dim-H^k-L^-1}, Theorem \ref{th-coh-L^-1} and Proposition \ref{prop-NO-poly-dim-H^0}.
\end{proof}

\begin{Rem} \label{rem-main-th-additive}
(1) When the moment polytope map $E \mapsto \Delta(E)$ is additive, that is, 
$\Delta(E_1E_2) = \Delta(E_1)+\Delta(E_2)$ for any two $G$-linear systems $E_1$, $E_2$, 
then the formula \eqref{equ-arith-genus-smooth-sph1} can be computed only in terms of the polytopes $\Delta(E_i)$ (see Sections \ref{subsec-horosph} and \ref{subsec-gp-embed}).

(2) Similarly, when the map $E \mapsto \tilde{\Delta}(E)$ is additive, that is, 
$\tilde{\Delta}(E_1E_2) = \tilde{\Delta}(E_1)+\tilde{\Delta}(E_2)$ for any two $G$-linear systems $E_1$, $E_2$, 
then the formula \eqref{equ-arith-genus-smooth-sph2} can be computed only in terms of the 
polytopes $\tilde{\Delta}(E_i)$ (see Sections \ref{subsec-horosph} and \ref{subsec-gp-embed}).
\end{Rem}

The following is the extension of Definition \ref{def-critical-number-toric} to the spherical case:
\begin{Def} \label{def-critical-number-sph}
Let $E_1, \ldots, E_k$ be as above.
We say that a nonnegative integer $i$ is {\it critical}
for the $E_1, \ldots, E_k$ if there is a nonempty set $J\subset \{1,\ldots, k\}$ such that $N^\circ(\tilde{\Delta}_J)>0$ and $\dim(\tilde{\Delta}_J) - |J| = i$
(equivalently $\dim(\Delta_J)+d_J - |J| = i$).
Recall that $\tilde{\Delta}_J$ (respectively $\Delta_J$) is the Newton-Okounkov polytope (respectively the moment polytope) of the linear system $E_J = \prod_{i \in J} E_i$. 
%and $Y_J$ is the closure of the image of the Kodaira map of 
%$\prod_{i \in J} E_i$.
\end{Def}

\begin{Th}[$h^{p, 0}$ numbers of a complete intersection in a spherical homogeneous space] \label{th-h-p-complete-int-sph}
Let $p$ be a nonnegative integer which is not critical for $E_1, \ldots, E_k$ in the sense of Definition \ref{def-critical-number-sph}.
Then: 
$$h^{p,0}(X_k) = \begin{cases} 
1 \quad p = 0 \\
0 \quad p \neq 0.\\
\end{cases}$$
\end{Th}
\begin{proof}
Follows from Theorem \ref{th-h-p-sph} and Theorem \ref{th-22}.
\end{proof}

\begin{Cor} \label{cor-h-p-complete-int-sph}
Suppose all the numbers $0 \leq i < n-k$ are not critical for $E_1, \ldots, E_k$. In particular, this is the case if all the Newton-Okounkov polytopes 
$\tilde{\Delta}_1, \ldots, \tilde{\Delta}_k$ have full dimension $n = \dim(G/H)$.
Then all the numbers $h^{i,0}(X_k)$, $0 \leq i \leq n-k$, are zero 
except for $h^{0,0}(X_k) = 1$ and $h^{n-k, 0}(X_k)$ which can be computed from \eqref{equ-arith-genus-smooth-sph1} or \eqref{equ-arith-genus-smooth-sph2}.
\end{Cor}
\begin{proof}
Under the assumptions in the corollary we have $d(J) \geq n - k$ and hence no $p$ with $1 \leq p < n-k$ is a critical number and hence $h^{p,0}(X_k) = 0$ by 
Theorem \ref{th-h-p-complete-int-sph}.
\end{proof}

%%%%%%%%%%%%%%%%%%%%%%%%%%%%%%%%%%%%%%%%%%%%%%%%%%%%%%%%%%%%%%%%%%

\section{Examples} \label{sec-examples}
\subsection{Horospherical varieties} \label{subsec-horosph}
A subgroup $H \subset G$ is called {\it horospherical} if it contains a maximal unipotent subgroup. The corresponding homogeneous space $G/H$
is called {\it horospherical homogeneous space}. It can be shown that a horospherical homogeneous space is spherical, that is, it has an open $B$-orbit.
A spherical $G$-variety $X$ is horospherical if the open $G$-orbit is horospherical.

It is well-known that: 
\begin{Prop} \label{prop-horosph-P'}
A subgroup $H \subset G$ is horospherical if and only if there is a parabolic subgroup 
$P \subset G$ such that $P' \subset H \subset P$ (see \cite[Section 2.1]{KKh-horosph} for a proof).
\end{Prop}

\begin{Ex}
Let $U$ be a maximal unipotent subgroup of $G$. Then $G/U$ is a quasi-affine horospherical homogeneous space. The natural map $G/U \to G/B$ is a fibration over the flag variety 
$G/B$ and the fibers are isomorphic to the maximal torus $T$.
If $G = \GL(n, \c)$, $U$ can be taken to be the subgroup of upper-triangular matrices with $1$'s on the diagonal. 
\end{Ex}

\begin{Rem}
The name horospherical comes from hyperbolic geometry (see \cite[Example 7.1]{Timashev}).
\end{Rem}

Let $\lambda \in \Lambda^+$. For a rational $G$-module $M$ let us denote the $\lambda$-isotypic component of 
$M$ by $M_\lambda$, that is, $M_\lambda$ is the sum of all the copies of the irreducible $G$-module $V_\lambda$ in $M$. Clearly, if $M$ is multiplicity free then $M_\lambda$ is either $\{0\}$ or 
isomorphic to $V_\lambda$. In particular, if $\E$ is a $G$-linearized line bundle over a spherical homogeneous space $G/H$ then $H^0(G/H, \E)_\lambda$ is either $\{0\}$ or $V_\lambda$. 
The following is well-known (see \cite{Popov}):
\begin{Th} \label{th-horosph}
Let $G/H$ be a horospherical homogeneous space. Let $\E_1$ and $\E_2$ be $G$-linearized line bundles on $G/H$ with $E_i \subset H^0(G/H, \E_i)$, $i=1, 2$, 
two $G$-invariant linear systems.
Then for any $\lambda, \gamma \in \Lambda$ the product $(E_1)_\lambda (E_2)_\gamma$ lies in $(E_1 E_2)_{\lambda + \gamma}$, where 
$E_1 E_2$ is the linear system in $\E_1 \otimes \E_2$ spanned by all the products $f_1 f_2$, $f_i \in E_i$, and $(E_1 E_2)_{\lambda + \gamma}$ is the 
$(\lambda+\gamma)$-isotypic component of $E_1E_2$.
\end{Th}

Let $\E$ be a $G$-linearized line bundle on $G/H$ and take a finite subset $\A$:  
$$\A \subset \Spec_G(H^0(G/H, \E)) = \{ \lambda \mid V_\lambda \textup{ appears in } H^0(G/H, \E)\}.$$ As before 
we let $E_\A$ denote the linear system:
$$E_\A = \bigoplus_{\lambda \in \A} V_\lambda.$$  
As in \cite{KKh-horosph} one shows that:
\begin{Cor} \label{cor-moment-poly-horosph}
With notation as above, the moment polytope of the graded algebra $\overline{\bigoplus_m E_\A^m}$ coincides with the convex hull of $\A$.
\end{Cor}

\begin{Def} \label{def-NO-polytope-horosph}
For each finite subset $\A \subset \Lambda$ let $\Delta(\A)$ denote the convex hull of $\A$. In fact $\Delta(\A)$ is the moment polytope 
$\Delta(E_\A)$ of its associated linear system. Also 
$$\tilde{\Delta}(\A) = \bigcup_{\lambda \in \Delta(\A)} \{\lambda \} \times \Delta_{\w}(\lambda),$$ the corresponding 
Newton-Okounkov polytope. Recall that $\w$ is a fixed reduced decomposition for the longest element $w_0$ in the Weyl group $W$ and
$\Delta_\w(\lambda) \subset \r^N$ is the string polytope associated to $\lambda$ and $\w$ (see Section \ref{subsec-NO-poly}).
\end{Def}

From Theorem \ref{th-horosph} it follows that for any two finite subsets $\A_1, \A_2$  we have $E_{\A_1}E_{\A_2} = E_{\A_1 + \A_2}$. 
As in \cite{KKh-horosph} we have the following:
\begin{Prop}[Additivity of the moment and Newton-Okounkov polytopes]
\begin{itemize}
\item[(i)]  The map $E \mapsto \Delta(E)$, which associates to an invariant linear system $E$ its moment polytope, is additive. This is basically 
the condition (II) in Section \ref{subsec-Euler-char-sph}.
\item[(ii)]  Suppose the string polytope map $\lambda \mapsto \Delta_\w(\lambda)$ is additive (e.g. $G = \GL(n, \c)$ 
and Gelfand-Zetlin polytopes). Then the map
$E \mapsto \tilde{\Delta}(E)$ is also additive. This is basically the condition (I) in Section \ref{subsec-Euler-char-sph}. 
\end{itemize}
\end{Prop}

Let $\E_1, \ldots, \E_k$ be $G$-linearized line bundles. For each $i = 1, \ldots, k$ let $\A_i$ be a finite subset of $\Spec_G(H^0(G/H, \E_i)) \subset \Lambda^+$ 
and $E_i \subset H^0(G/H, \E_i)$ the corresponding $G$-invariant linear system. Also to simplify the notation, 
for each $i$ let $\Delta_i$ denote the convex polytope $\Delta(\A_i)$, i.e. the convex hull of the finite set $\A_i$.

\begin{Cor}[Genus of a complete intersection in a horospherical homogenous space] \label{cor-genus-horosph}
Let $f_1, \ldots, f_k$ be generic elements in $E_1, \ldots, E_k$ respectively. For each $i =1, \ldots, k$ let $D_i = \{ x \in G/H \mid f_i(x) = 0 \}$ be the 
hypersurface defined by $f_i$
and let $X_k = D_1 \cap \cdots \cap D_k$. Then:
\begin{multline} \label{equ-arith-genus-horosph}
\chi(X_k) =  1 - \sum_{i_1} S'(\Delta_{i_1}, \alpha_{i_1}) + \sum_{i_1 < i_2} S'(\Delta_{i_1}+\Delta_{i_2}, \alpha_{i_1}+\alpha_{i_2}) - \cdots \\
+ (-1)^k S'(\Delta_1 + \cdots + \Delta_k, \alpha_1+\cdots+\alpha_k). 
\end{multline}
\end{Cor}

\begin{Rem}  \label{rem-horo-NO-poly-add}
As in Remark \ref{rem-main-th-additive}(2) if the string polytopes are additive for the choice of the reduced decomposition $\w$ (e.g. $G = \GL(n, \c)$ 
and Gelfand-Zetlin polytopes) then the Newton-Okounkov polytope is also additive. Let $\tilde{\Delta}_i$ denote the Newton-Okounkov polytope $\tilde{\Delta}(\A_i)$ 
associated to the finite subset $\A_i$.
Then, under the assumption of additivity of the string polytopes, we can rewrite the formula \eqref{equ-arith-genus-horosph} for the genus in terms of the 
Newton-Okounkov polytopes $\tilde{\Delta}_i$ and their sums:
\begin{multline} \label{equ-arith-genus-horosph-NO-poly}
\chi(X_k) =  1 - \sum_{i_1} N'(\tilde{\Delta}_{i_1}, \alpha_{i_1}) + \sum_{i_1 < i_2} N'(\tilde{\Delta}_{i_1} + \tilde{\Delta}_{i_2}, \alpha_{i_1}+\alpha_{i_2}) - \cdots \\
+ (-1)^k N'(\tilde{\Delta}_1 + \cdots + \tilde{\Delta}_k, \alpha_1+\cdots+\alpha_k). 
\end{multline}
\end{Rem}

In particular if $G/H$ is a quasi-affine variety we can consider the finite $G$-invariant subsets of the ring of regular functions $\c[G/H]$, that is, 
when we take all the line bundles $\E_i$ to be the trivial line bundle. 
In fact, as we now explain, to each face of the Weyl chamber $\Lambda^+_\r$ there corresponds 
a quasi-affine horospherical homogeneous space and by Proposition \ref{prop-horosph-P'} every horospherical homogeneous space $G/H$ is a 
quotient of such a quasi-affine horospherical homogeneous space: Let $\sigma$ be a face of the Weyl chamber $\Lambda^+_\r$. Let $P$ denote the corresponding parabolic subgroup. When $G = \GL(n, \c)$, the parabolic subgroups 
(up to conjugation) are block upper triangular subgroups. Let $P'$ denote the commutator subgroup of $P$. One shows that the homogeneous space $G/P'$ is quasi-affine and moreover, the algebra 
of regular functions $\c[G/P']$, for the natural action of $G$, decomposes as follows (\cite{Popov-Vinberg}):
$$\c[G/P'] = \bigoplus_{\lambda \in \sigma \cap \Lambda^+} V_\lambda^*.$$
Now each finite subset $\A \subset \sigma \cap \Lambda^+$ determines a $G$-invariant subspace $L_\A = \bigoplus_{\lambda \in \A} V_\lambda$ of $\c[G/P']$. 
Let us take finite subsets $\A_1, \ldots, \A_k$ of the semigroup $\sigma \cap \Lambda$. Then Corollary \ref{cor-genus-horosph} gives us a formula for the genus of a generic complete intersection from the subspaces $L_1, \ldots, L_k$ in the quasi-affine variety $G/P'$, in terms of the convex hulls $\Delta_i$ of the subsets $\A_i$.

Finally, below is a concrete example of Corollary \ref{cor-genus-horosph}.
\begin{Ex} \label{ex-S-variety}
Let $V$ be a finite dimensional $G$-module. Let $v_1, \ldots, v_s$ be highest weight vectors of $V$ with highest weights $\lambda_1, \ldots, \lambda_s$ respectively. 
Put $v = v_1 + \cdots + v_s$ and let $X$ be the closure of the $G$-orbit of $v$ in $V$. It is an affine horospherical subvariety of $V$. Let $L$ be the linear subspace of $\c[X]$
consisting of linear functions in $V^*$ restricted to $X$. As above one observes that the moment polytope $\Delta$ of the subspace $L$, or equivalently the graded algebra
$\bigoplus_m \overline{L^m}$, is $\Delta = \conv\{ \lambda_i \mid i = 1, \ldots, s\}$. Also let $\tilde{\Delta}$ denote the corresponding 
Newton-Okounkov polytope. 
Then if $f$ is a generic element in $L$ defining a hypersurface $H_f = \{ x \in X  \mid f(x) = 0\}$, 
Corollary \ref{cor-genus-horosph} implies that the genus of  $H_f$ is equal to $S^\circ(\Delta) = N^\circ(\tilde{\Delta})$, 
i.e. the number of integral points in the interior of the polytope $\tilde{\Delta}$.
\end{Ex}

\subsection{Group embeddings} \label{subsec-gp-embed}
%Define a group compactification associated to a representation
%Definition of a regular group compactification.
%Define the weight polytope
%Define the NO polytope for a representation
%Theorem: given k representations formula for the arithmetic genus and geometric genus of a compete intersection in the group
%Theorem: explicit form of forms of top degree on a hypersurface, what is the canonical class of a group compactification?
%Example of SL(2)
Let $\pi: G \to \GL(n, \c) \subset \textup{Mat}(V)$ be a
finite dimensional representation of a connected reductive group $G$.
Let $\pi_{ij}: G \to \c$, $i,j=1, \ldots, n$ be the matrix elements, i.e. the
entries of $\pi$. Let $L_\pi$ be the subspace of regular functions on $G$ spanned by the
$\pi_{ij}$. Consider the action of $G \times G$ on $G$ by the multiplication
from left and right. The subspace $L_\pi$ is a $(G \times G)$-invariant subspace.
Let $\Lambda_\r^+$ (respectively $W$) denote the positive Weyl chamber (respectively the Weyl group) of $G$.

\begin{Def}[Weight polytope]
The convex hull of the Weyl orbit of the highest weights of the representation $\pi$ is called the {\it weight polytope} of $\pi$. We will denote the weight polytope by $P_\pi$ and 
its intersection with the positive Weyl chamber by $P^+_\pi$.
\end{Def}

As in \cite{Kazarnovskii}, one can show that:
\begin{Th}
The moment polytope of the $(G \times G)$-algebra $\overline{\bigoplus_{m} L^m_\pi}$
(which lives in $\Lambda_\r^+ \times \Lambda_\r^+$)
can be identified with $P^+_\pi$.
%the convex hull of the $W$-orbit of the highest weights of the representation $\pi$
%intersected with the positive Weyl chamber $\Lambda_\r^+$. 
To identify the moment polytope and the polytope $P^+_\pi$ we should send a point $(\lambda, \lambda^*)$ to $\lambda$. 
%Finally, the integral closure $\overline{A_{L_\pi}}$ has the same moment polytope 
%as $A_{L_\pi}$.
\end{Th}

Let $F$ denote the Weyl polynomial for the group $G \times G$. It is a polynomial on the vector space $\Lambda_\r \times \Lambda_\r$ (see Section \ref{subsec-moment}, paragraph after Theorem \ref{th-H^0-moment-poly} for the definition of the Weyl polynomial). 
For any dominant weight of $G \times G$ of the form $(\lambda, \lambda^*)$ we have $F(\lambda, \lambda^*) = \dim(V_\lambda \times V_\lambda^*) = \dim(V_\lambda)^2$.

Let us take $k$ representation $\pi_1, \ldots, \pi_k$ where $k \leq \dim(G)$. For each $\pi_i$ let $L_i$ (respectively $P_i^+$) be its subspace of matrix elements
(respectively its weight polytope intersected with the positive Weyl chamber).
\begin{Cor}[Genus of a complete intersection in a group] \label{cor-genus-gp}
Let $f_1, \ldots, f_k$ be generic elements in $L_{\pi_1}, \ldots, L_{\pi_k}$ respectively. For each $i =1, \ldots, k$ let $D_i = \{ x \in G/H \mid f(x) = 0 \}$.
and let $X_k = D_1 \cap \cdots \cap D_k$. Then:
\begin{multline} \label{equ-arith-genus-gp}
\chi(X_k) =  1 - \sum_{i_1} S'(P_1^+) + \sum_{i_1 < i_2} S'(P_1^+ + P_2^+) - \cdots \\
+ (-1)^k S'(P_1^+ + \cdots P_k^+).
\end{multline}
\end{Cor}

As usual to a representation $\pi$ we can associate a Newton-Okounkov polytope $\tilde{\Delta}_\pi$ defined by:
$$\tilde{\Delta}_\pi =  \bigcup_{\lambda \in P^+_\pi} (\{(\lambda, \lambda^*)\} \times \Delta_{\w}(\lambda) \times \Delta_\w(\lambda^*)).$$

As in \cite{KKh-Kazarnovskii} we have the following:
\begin{Prop}[Additivity of the moment and Newton-Okounkov polytopes]
\begin{itemize}
\item[(i)]  The map $\pi \mapsto P^+_\pi$, which associates to a representation $\pi$ its polytope $P^+_\pi$, is additive with respect to the tensor product of representations. 
This is basically 
the condition (II) in Section \ref{subsec-Euler-char-sph}.
\item[(ii)]  Suppose the string polytope map $\lambda \mapsto \Delta_\w(\lambda)$ is additive (e.g. $G = \GL(n, \c)$ 
and the Gelfand-Zetlin polytopes). Then the map
$\pi \mapsto \tilde{\Delta}_\pi$ is also additive. This is basically the condition (I) in Section \ref{subsec-Euler-char-sph}. 
\end{itemize}
\end{Prop}

As in (ii) above, suppose the string polytope map is additive. Let $\tilde{\Delta}_i$ denote the Newton-Okounkov polytope of 
the representaiton $\pi_i$ in Corollary \ref{cor-genus-gp}. Then we can rewrite the formula \eqref{equ-arith-genus-gp} for the genus in terms of the 
Newton-Okounkov polytopes $\tilde{\Delta}_i$ and their sums:
\begin{multline} \label{equ-arith-genus-gp-NO-poly}
\chi(X_k) =  1 - \sum_{i_1} N'(\tilde{\Delta_i}) + \sum_{i_1 < i_2} N'(\tilde{\Delta}_{i_1} + \tilde{\Delta}_{i_2})) - \cdots \\
+ (-1)^k N'(\tilde{\Delta}_1 + \cdots + \tilde{\Delta}_k). 
\end{multline}

\subsection{Flag varieties} \label{subsec-flag}
Let $X = G/B$ be the complete flag variety of a connected complex reductive algebraic group. 

\begin{Cor}[Genus of a complete intersection in the flag variety] \label{cor-genus-flag-var}
Let $\lambda_1, \ldots, \lambda_k$ be dominant weights. Let $D_1, \ldots, D_k$ be smooth and transversely intersecting divisors for the corresponding line bundles 
$L_{\lambda_1}, \ldots, L_{\lambda_k}$. Let $X_k = D_1 \cap \cdots \cap D_k$. Then:
\begin{multline} \label{equ-arith-genus-smooth-sph}
\chi(X_k) =  1 - \sum_{i_1} N'(\Delta_{\w}(\lambda_{i_1})) + \sum_{i_1 < i_2} N'(\Delta_{\w}(\lambda_{i_1}+\lambda_{i_2})) - \cdots \\
+ (-1)^k N'(\Delta_{\w}(\lambda_1 + \cdots + \lambda_k)). 
\end{multline}
\end{Cor}

Finally we briefly discuss the case of the flag variety of $G = \GL(n, \c)$ and Gelfand-Zetlin polytopes.
Let $G = \GL(n, \c)$. The flag variety of $G$ can be identified with the variety of all flags of linear subspace in $\c^n$:
$$\{0\} \subsetneqq F_1 \subsetneqq \cdots \subsetneqq F_n = \c^n.$$
Each dominant weight $\lambda$ of $G$ can be represented as an increasing $n$-tuple of integers:
$$\lambda = (\lambda_1 \leq \cdots \leq \lambda_n).$$
In their well-known work \cite{GZ}, given a dominant weight $\lambda$, 
Gelfand and Zetlin construct a natural vector basis for the irreducible representation $V_\lambda$ whose elements are parameterized with the 
integral points $x_{i,j}$ satisfying the following set of interlacing inequalities:
\begin{equation} \label{equ-GZ}
\left. \begin{matrix} \lambda_1 & \lambda_2 & \cdots & \cdots &
\cdots & \lambda_n \cr &&&&& \cr & x_{1,n-1} & x_{2,n-1} & \cdots &
\cdots & x_{n-1,n-1} \cr &&&&& \cr && x_{2,n-2} & x_{2,n-2} & \cdots
& x_{n-2,n-2} \cr &&&&& \cr &&& \cdots & \cdots & \cdots \cr &&&&&
\cr &&&& x_{1,2} & x_{2,2} \cr &&&&& \cr &&&&&x_{1,1}\cr
\end{matrix} \right.
\end{equation}
where the notation $$\left. \begin{matrix} a & b \cr &c \end{matrix}\right.$$
means $a \leq c \leq b$.

The set of all points $(x_{i,j})$ in $\r^{n(n-1)/2}$ satisfying \eqref{equ-GZ} is called the {\it Gelfand-Zetlin polytope} associated to $\lambda$ denoted by $\Delta_{\textup{GZ}}(\lambda)$. 
Let $\L_\lambda$ be the $G$-line bundle on the flag variety associated to a dominant weight $\lambda$. 
If $H$ is a generic divisor of $\L_\lambda$ then Corollary \ref{cor-genus-flag-var}
states that the genus of $H$ is equal to the number of integral points in $\r^{n(n-1)/2}$ lying in the interior of $\Delta_{\textup{GZ}}(\lambda)$. In other words, 
the number of integral points in $\r^{n(n-1)/2}$ which satisfy the inequalities in \eqref{equ-GZ} where all the inequalities are strict.

%\begin{Cor}[Genus of a complete intersection in the flag variety] \label{cor-h^p-flag-var}
%The variety $X_k$ is connected, i.e. $h^{0,0}(X_k) = 1$, and $h^{p, 0}(X_k) = 0$ for $0 < p < N-k$, i.e. it has no holomorphic forms of intermediate dimension.
%Finally: 
%\begin{multline} \label{equ-geo-genus-smooth-sph}
%p(X_k) = h^{N-k, 0}(X_k) = 
%(-1)^k \sum_{i_1} N'(\Delta_{\w}(\lambda_{i_1})) + (-1)^{k-1} \sum_{i_1 < i_2} N'(\Delta_{\w}^\circ(\lambda_{i_1} + \lambda_{i_2})) + \cdots \\
%+ N'(\Delta_{\w}(\lambda_1 + \cdots + \lambda_k)). 
%\end{multline}
%\end{Cor}

%%%%%%%%%%%%%%%%%%%%%%%%%%%%%%%%%%%%%%%%%%%%%%%%%%%%%%%%%%%%%%%%%

%\bibliographystyle{amsplain}


\begin{thebibliography}{99}

\bibitem[Alexeev-Brion04]{AB} Alexeev, V.; Brion, M. {\it Toric
degeneration of spherical varities}.
Selecta Math. (N.S.)  10  (2004),  no. 4, 453--478.

\bibitem[Bernstein75]{Bernstein} Bernstein, D. N.
{\it The number of roots of a system of equations}.
English translation: Functional Anal. Appl. 9 (1975), no. 3, 183--185 (1976).

%\bibitem[Berenstein-Zelevinsky88]{B-Z1} Berenstein, A.; Zelevinsky, A.
%{\it Tensor product multiplicities and convex polytopes in partition space}. J. Geom. Phys. 5 (1988), no. 3, 453--472.

\bibitem[Berenstein-Zelevinsky01]{BZ} Berenstein, A.; Zelevinsky, A.
{\it Tensor product multiplicities, canonical bases and totally
positive varieties}. Invent. Math. 143 (2001), no. 1, 77--128.

\bibitem[Brion89]{Brion-Picard} Brion, M.
{\it Groupe de Picard et nombres caracteristiques des varieties spheriques}.
Duke Math. J. 58 (1989), no. 2, 397--424.

\bibitem[Brion87]{Brion-moment} Brion, M.
{\it Sur l'image de l'application moment}. S\'{e}minaire d'alg\`{e}bre Paul
Dubreil et Marie-Paule Malliavin (Paris, 1986), 177--192, Lecture
Notes in Math., 1296, Springer, Berlin, 1987.

\bibitem[Brion90]{Brion-vanishing} 
Brion, M. {\it Une extension du th\'eor\`eme de Borel-Weil}. Math. Ann. 286 (1990), no. 4, 655--660.
%Brion, M.; Inamdar, S. P.
%{\it Frobenius splitting of spherical varieties}. Algebraic groups and their generalizations: classical methods (University Park, PA, 1991), 207--218, 
%Proc. Sympos. Pure Math., 56, Part 1, Amer. Math. Soc., Providence, RI, 1994. 

\bibitem[Carrell-Liebermann73]{CL} Carrell, J. B.; Lieberman, D. I. {\it Holomorphic vector fields and Kaehler manifolds}, 
Invent. Math. 21 (1973), 303--309.

\bibitem[DeConcini-Procesi82]{DP}
De Concini, C., and C. Procesi, {\it Complete symmetric varieties}. 
Invariant theory (Montecatini, 1982), Lecture Notes in Math., 996, Springer, Berlin (1983), 1--44.

\bibitem[Gelfand-Zetlin50]{GZ} Gelfand, I.M.; Cetlin, M.L. {\it Finite dimensional representations of
the group of unimodular matrices}, Doklady Akad. Nauk USSR (N.S.)
,71 (1950), 825--828.

\bibitem[Guillemin-Sternberg84]{G-S} Guillemin, V.; Sternberg, S.
{\it Geometric quantization and multiplicities of group representations}.
Invent. Math. 77 (1984), 533--546.

\bibitem[Hartshorne77]{Hartshorne} Hartshorne, R. {\it Algebraic geometry}.
Graduate Texts in Mathematics, No. 52. Springer-Verlag, New York-Heidelberg, 1977.

%\bibitem[Kaveh03]{Kiumars-note-spherical} Kaveh, K. {Note on the cohomology rings
%of spherical varieties and volume polynomial}. Journal of Lie Theory 21 (2011), No. 2, 263--283

\bibitem[Kaveh04]{Kiumars-thesis}
Kaveh, K. {\it Morse theory and Euler characteristic of sections of spherical varieties}. Transform. Groups 9 (2004), no. 1, 47--63.

\bibitem[Kaveh]{Kiumars-string} Kaveh, K. {Crystal bases and Newton-Okounkov bodies}. To appear in 
Duke Mathematical Journal.

%\bibitem[Kaveh-Khovanskii08a]{Askold-Kiumars-affine} Kaveh, K.; Khovanskii, A. G.
%{\it Convex bodies and algebraic equations on affine varieties}. Preprint: {\sf arXiv:0804.4095v1}.
%A short version with title {\it Algebraic equations and convex bodies}
%to appear in {\it Perspectives in Analysis, Topology and Geometry},
%Birkh\"aser series {\sf Progress in Mathematics}.

%\bibitem[Kaveh-Khovanskii08b]{Askold-Kiumars-MMJ} Kaveh, K.; Khovanskii, A. G.
%{\it Mixed volume and an extension of intersection theory of divisors}. 
%Moscow Math. J. (2010), Vol. 10, No. 2, 343--375 

\bibitem[Kaveh-Khovanskii10]{KKh-Kazarnovskii} Kaveh, K.; Khovanskii, A. G.{\it  Moment polytopes, semigroup of representations and Kazarnovskii's theorem}. 
J. Fixed Point Theory Appl. 7 (2010), no. 2, 401--417.

\bibitem[Kaveh-Khovanskii11]{KKh-horosph}
Kaveh, K.; Khovanskii, A. G. {\it Newton polytopes for horospherical spaces}. 
Mosc. Math. J. 11 (2011), no. 2, 265--283, 407. 

\bibitem[Kaveh-Khovanskii12a]{KKh-Annals} Kaveh, K.; Khovanskii, A. G.
{\it Newton-Okounkov bodies, semigroups of integral points, graded algebras and intersection theory}.
Ann. of Math. (2) 176 (2012), no. 2, 925--978.

\bibitem[Kaveh-Khovanskii12b]{KKh-reductive} Kaveh, K.; Khovanskii, A. G. {\it Convex bodies associated to actions of reductive groups}. Moscow
Mathematical Journal, 12 (2012) no. 2., 369--396, 461.

\bibitem[Kazarnovskii87]{Kazarnovskii} Kazarnovskii, B. {\it Newton polyhedra and the
Bezout formula for matrix-valued functions of finite dimensional
representations}. Functional Analysis and its
applications, v. 21, no. 4, 73--74 (1987).

\bibitem[Khovanskii77]{Askold-toroidal} Khovanskii, A. G., {\it Newton polyhedra
and toroidal varieties}. Funkcional. Anal. i Prilo\v zen. 11 (1977),
no. 4, 56--64, 96.

\bibitem[Khovanskii78]{Askold-genus} Khovanskii, A. G., {\it Newton polyhedra
and the genus of complete intersections}. Funktsional. Anal. i
Prilozhen.  12  (1978), no. 1, 51--61.

%\bibitem[Khovanskii84]{Khovanskii-formulas} Khovanskii, A. G. {\it Geometry of Formulas}, (pp. 67-91, Section 3) in  
%V.I. Arnold , A.N. Varchenko , A.B. Givental and A.G.Khovanskii
%"Singularities of functions, wave fronts, caustics and multidimensional
%integrals", pp. 1-91 in Soviet Scientific Reviews, Section C, MATHEMATICAL
%PHYSICS REVIEWS, Volume 4 (1984).

\bibitem[Khovanskii15]{Askold-new}
Khovanskii, A. G. {\it Newton polyhedra and irreducible components of complete
intersections}. To appear in Izvestiya RAN, Ser. Matematika (2015).

\bibitem[Khovanskii-Pukhlikov93]{Kh-P}
Khovanskii, A. G.; Pukhlikov, A. V. {\it Finitely additive measures of virtual polyhedra}. (Russian) Algebra i Analiz 4 (1992), no. 2, 161--185; translation in St. 
Petersburg Math. J. 4 (1993), no. 2, 337--356

\bibitem[Kiritchenko06]{Valentina1}
Kiritchenko, V. {\it Chern classes of reductive groups and an adjunction formula}. Ann. Inst. Fourier (Grenoble) 56 (2006), no. 4, 1225--1256.

\bibitem[Kiritchenko07]{Valentina2}
Kiritchenko, V. {\it On intersection indices of subvarieties in reductive groups}. Mosc. Math. J. 7 (2007), no. 3, 489--505, 575.

\bibitem[Kushnirenko76]{Kushnirenko} Kushnirenko, A. G.
{\it Polyedres de Newton et nombres de Milnor}. (French)
Invent. Math. 32 (1976), no. 1, 1--31.

%\bibitem[Lazarsfeld04]{Lazarsfeld} Lazarsfeld, R.
%{\it Positivity in algebraic geometry. I. Classical setting: line bundles and linear series}.
%Ergebnisse der Mathematik und ihrer Grenzgebiete. 3.
%Folge. A Series of Modern Surveys in Mathematics, 48. Springer-Verlag, Berlin, 2004.

\bibitem[Lazarsfeld-Mustata09]{LM}
Lazarsfeld, R.; Mustata, M. {\it Convex bodies associated to linear series}.
Annales scientifiques de l'ENS 42, no. 5 (2009), 783--835.

\bibitem[Littelmann98]{Littelmann} Littelmann, P. {\it Cones, crystals, and patterns}.
Transform. Groups 3 (1998),  no. 2, 145--179.

\bibitem[McMullen77]{McMullen}
McMullen, P.
{\it Valuations and Euler-type relations on certain classes of convex polytopes}. 
Proc. London Math. Soc. (3) 35 (1977), no. 1, 113--135. 
%McMullen, P.
%{\it Metrical and combinatorial properties of convex polytopes}. Proceedings of the International Congress of Mathematicians 
%(Vancouver, B. C., 1974), Vol. 1, pp. 491--495. Canadian Math. Congress, Montreal, Quebec, 1975. 

\bibitem[Ness84]{Ness} Ness, L.
{\it A stratification of the null cone via the moment map}.
With an appendix by David Mumford.
Amer. J. Math. 106 (1984), no. 6, 1281--1329.

\bibitem[Okounkov96]{Okounkov-Brunn-Minkowski} Okounkov, A.
{\it Brunn-Minkowski inequality for multiplicities}.
Invent. Math. 125 (1996), no. 3, 405--411.

\bibitem[Okounkov97]{Okounkov-spherical} Okounkov, A.
{\it A remark on the Hilbert polynomial of a spherical variety}.
Func. Anal. and Appl., 31 (1997), 82--85.

\bibitem[Okounkov03]{Okounkov-log-concave} Okounkov, A.
{\it Why would multiplicities be log-concave?}
The orbit method in geometry and physics (Marseille, 2000),
329--347, Progr. Math., 213, Birkha"user Boston, Boston, MA, 2003.

\bibitem[Perrin14]{Perrin}
Perrin, N. {\it On the geometry of spherical varieties}. Transform. Groups 19 (2014), no. 1, 171--223. 

\bibitem[Popov86]{Popov} Popov, V. L.
{\it Contractions of actions of reductive algebraic groups}. Mat.
Sb. (N.S.) 130(172) (1986), no. 3, 310--334, 431.

\bibitem[Popov-Vinberg72]{Popov-Vinberg} Popov, V. L.; Vinberg, E. B. {\it A certain class of quasihomogeneous affine varieties}. (Russian) 
Izv. Akad. Nauk SSSR Ser. Mat. 36 (1972), 749--764.

\bibitem[Samuel-Zariski60]{SZ} Samuel, P.; Zariski, O.
{\it Commutative algebra}. Vol. II. Reprint of the 1960 edition.
Graduate Texts in Mathematics, Vol. 29.

\bibitem[Timashev06]{Timashev} Timashev, D.
{\it Homogeneous spaces and equivariant embeddings}. 

\bibitem[Viro88]{Viro} Viro, O. Ya.
{\it Some integral calculus based on Euler characteristic}. Topology and geometry--Rohlin Seminar, 127--138, 
Lecture Notes in Math., 1346, Springer, Berlin, 1988. 

\end{thebibliography}
\end{document}